\documentclass[10pt,a4paper]{article}
\usepackage[english]{babel}
\usepackage{amssymb}
\usepackage{amsmath,mathrsfs}
\usepackage{indentfirst}
\usepackage{epsfig}
\usepackage{xcolor}
\usepackage[colorlinks,citecolor=blue,linkcolor=red]{hyperref}
\usepackage{tikz}
\usetikzlibrary {matrix}

\usepackage{graphicx}
\usepackage{caption}
\usepackage{subcaption}
\usepackage{tikz} 
\usetikzlibrary{fit,positioning}
\usetikzlibrary{arrows,positioning,shapes,backgrounds,fit}  

\usepackage{amsthm}
\newcommand\norm[1]{\left\lVert#1\right\rVert}
\usepackage{multirow}  
\usepackage{amsfonts}
\usepackage{mathtools} 
\usepackage{setspace} 

\newtheorem{cor}{Corollary}[section]  
\newtheorem{defi}[cor]{Definition}  
\newtheorem{assumption}[cor]{Assumption} 

\newtheorem{lem}[cor]{Lemma}
\newtheorem{teo}[cor]{Theorem}

\newtheorem{rem}[cor]{Remark}

\newcommand{\pa}{\partial}

\numberwithin{equation}{section}

\newtheorem{theorem}{Theorem}
\theoremstyle{plain}

\numberwithin{equation}{section}

\begin{document} 
\title{Higher-order Riemannian spline interpolation problems: a unified approach via gradient flows}
\author{
{\sc Chun-Chi Lin} \thanks{Department of Mathematics, National Taiwan Normal University, Taipei, 116 Taiwan, \url{chunlin@math.ntnu.edu.tw}} 
~ and ~
{\sc Dung The Tran} \thanks{Department of Mathematics, National Taiwan Normal University, Taipei, 116 Taiwan;  
Current Address: University of Science, Vietnam National University, Hanoi, Vietnam, 
\url{tranthedung56@gmail.com}} 
}

\date{October 17, 2025}

\maketitle



\begin{abstract}  
This paper addresses the problems of spline interpolation on smooth Riemannian manifolds, with or without the inclusion of least-squares fitting. Our unified approach utilizes gradient flows for successively connected curves or networks, providing a novel framework for tackling these challenges. This method notably extends to the variational spline interpolation problem on Lie groups, which is frequently encountered in mechanical optimal control theory. As a result, our work contributes to both geometric control theory and statistical shape data analysis. 
We rigorously prove the existence of global solutions 
in H\"{o}lder spaces for the gradient flow and demonstrate that the asymptotic limits of these solutions validate the existence of solutions to the variational spline interpolation problem. This constructive proof also offers insights into potential numerical schemes for finding such solutions, reinforcing the practical applicability of our approach. 
\end{abstract}

\tableofcontents

\baselineskip=10.0 pt plus .5pt minus .5pt

\section{Introduction} 
\label{Sec:Intro}


In an $m$-dimensional Riemannian manifold $M$, the variational spline interpolation problem involves identifying critical points of a higher-order energy functional, given a set of ordered points 
$\mathcal{P}=\{p_0, \ldots, p_q\}\subset M$. These critical points are selected from a set of maps $\gamma:[x_{0},x_{q}]\rightarrow M$, which represent sufficiently smooth maps satisfying $\gamma(x_j)=p_j$ for all $j$, 
and $x_0<x_1<\ldots<x_{q-1}<x_q$. 
Such a critical point $\gamma$ is referred to as a Riemannian spline or a geometric $k$-spline, extending the notion of polynomial splines from Euclidean spaces to Riemannian manifolds. The importance of data-fitting problems in curved spaces became apparent in the late 1980s, driven by advancements in mechanics, robotics industries, and data science. Given the necessity of extending spline-based methods to manifold settings in physics and applied science, considerable research efforts have been devoted to geometric $k$-splines. Particularly, when $k=2$, these splines are commonly referred to as Riemannian cubics, garnering considerable attention in the literature. These endeavors have yielded numerous research works on higher-order Riemannian splines or geometric $k$-splines with a broad array of applications, as documented in \cite{GK85, GHMRV12-1, GHMRV12-2, KDLS21}.

\bigskip 

{\bf Problem 1: The variational spline interpolation on Riemannian manifolds. } 
In the variational spline interpolation problem on smooth Riemannian manifolds, the objective is to determine curves or maps that start at $p_0$, end at $p_q$, and pass through the specified points 
$\{p_1, \ldots, p_{q-1}\}$ while minimizing the energy functional, 
\begin{align} 
\label{eq:E-energy}
\mathcal{E}_{k} [\gamma] 
&=\frac{1}{2} \sum\limits_{l=1}^{q}  
\int_{I_l} |D_{x}^{k-1} \gamma_{l,x}(x)|^2 \, dx 
\text{,}
\end{align}  
where 
$\gamma_{l,x}=\partial_x\gamma_l$ represents 
the velocity
of $\gamma_l=\gamma_{\lfloor_{\bar{I}_l}}$, $I_l=(x_{l-1},x_l)$, 
and $D_x:=D_{\partial_x\gamma}$ denotes the (directional) covariant derivative on $M$. 
Geometric control theory has motivated the approach of variational methods to tackle the variational spline interpolation problem on a Riemannian manifold. 
In the literature, this approach typically involves working with Euler-Lagrange equations and applying Pontryagin's maximum principle. 
However, existing results on \emph{global} solutions for Riemannian splines are still limited and often come with specific assumptions. 
The foundational work for generalizing splines from Euclidean spaces to Riemannian manifolds was established by Gabriel and Kajiya in \cite{GK85}, and independently by Noakes, Heinzinger, and Paden in \cite{NHP89}. 
Camarinha, Silva Leite, and Crouch \cite{CSC95-1} extended the concept of cubic splines to higher-order polynomial, or geometric, splines on complete Riemannian manifolds $M$. In \cite{NHP89, CSC95-1}, the construction of such splines relies on reducing the Euler–Lagrange equations associated with variational splines to nonlinear systems of ordinary differential equations (ODEs), or on explicit formulas involving exponential maps. These approaches critically depend on the assumption that $M$ is a Lie group. Consequently, the existence theory for variational splines on general complete Riemannian manifolds remained a significant challenge at the time. 
Addressing this issue from a variational viewpoint, Giambò, Giannoni, and Piccione \cite{GGP02, GGP04} later established the existence of global weak solutions to Riemannian spline boundary-value problems on complete Riemannian manifolds using min-max methods. 
Recently, Camarinha, Silva Leite, and Crouch in \cite{CSC22} used biexponential maps to prove the \emph{local} existence and uniqueness of Riemannian cubics subject to boundary conditions on positions and velocities at a point. 
Note, however, that in these works, the variational spline interpolation problem is formulated as a local problem, meaning that the boundary conditions at interior boundary points (i.e., knot points) are prescribed as constant values. 
Heeren and Rumpf addressed this issue in \cite{HRW19} specifically for the case of Riemannian cubics. They demonstrated the existence of weak solutions for Riemannian cubics in the elliptic Sobolev space $W^{2,2}$ Sobolev space by employing $\Gamma$-convergence techniques and direct methods. 
It's worth noting that the boundary conditions at interior boundary points in their work differ from the prescribed constant values seen in previous studies such as \cite{CSC95-1, GGP02, GGP04}. This variation in boundary conditions adds complexity to the construction of Riemannian splines, especially given the scarcity of explicit solutions and different boundary conditions at the interior boundary points. Consequently, applying the biexponential maps approach becomes challenging in this context. 
Given the difficulties in computing Riemannian cubic splines globally rather than locally, there has been increasing interest in establishing their existence and finding numerical approximations. For instance, Noakes and Ratiu explored the problem of approximating Riemannian cubics in a Lie group using bi-Jacobi fields in \cite{NR16}.

\bigskip

{\bf Problem 2: The variational spline interpolation versus least-squares fitting on Riemannian manifolds. } 
The variational spline interpolation versus least-squares fitting problem on smooth Riemannian manifolds seeks to construct curves or maps that begin at a prescribed point $p_0$, end at $p_q$, and pass as closely as possible through a given set of intermediate points $\{p_1, \ldots, p_{q-1}\}$, while minimizing a higher-order energy functional $\mathcal{E}_{k}$. 
This formulation can be viewed as a relaxed or complementary variant of the classical variational spline interpolation problem. Instead of requiring exact interpolation at the intermediate points, the problem incorporates a penalty term that softly enforces proximity to the target data. Specifically, the modified energy functional includes the penalization term: 
\begin{equation} 
\label{eq:penalty}
\frac{1}{2\sigma^2} \sum_{j=1}^{q-1} \text{dist}^2_{M}(p_j, \gamma(x_j))
\text{,}
\end{equation} 
where $\sigma>0$ is a small parameter. As $\sigma\rightarrow 0$, the fitted curve is increasingly constrained to interpolate the points $\{p_0, \ldots, p_{q}\}$ more precisely. 

Data fitting in curved spaces focuses on capturing desired geometric features or trajectories without necessarily enforcing exact interpolation of the given data on the manifold. This problem originally arose in the context of mechanical control theory, particularly in trajectory planning, and has since found significant applications in statistical shape analysis. Its relevance has expanded across various research areas, including image processing and shape data analysis (e.g., \cite{CZKI17, HFJ14, JK87, MSK10, MTM17}). 
More recently, such problems have appeared in quantum computing, where the constraints of quantum mechanics naturally lead to least-squares fitting formulations involving so-called quantum splines \cite{ACCS18, BHM12}. A central challenge in these settings is the identification of critical points or minimizers of higher-order energy functionals. 
In response, researchers have developed various strategies to generalize classical interpolation and approximation techniques to Riemannian manifolds. 
One of these directions is the study of template matching problems, which aims to align or compare geometric data while respecting the manifold structure (see, for example, \cite{GHMRV12-1, GHMRV12-2}).

The combination of higher-order Riemannian splines and least-squares fitting plays a crucial role in accurately representing data on smooth curved surfaces. Many investigations in the literature introduce penalization techniques to minimize the discrepancy between predefined data points and the optimized trajectory. Machado, Silva Leite, and Krakowski \cite{MSK10} explored the variational spline interpolation problem versus least-squares fitting, examining a limiting process that involves smoothing geometric splines when the underlying manifold $M$ is Euclidean. However, despite these efforts, there remains -- to the best of the authors' knowledge -- a lack of rigorous proofs establishing the existence of solutions to the geometric $k$-spline interpolation versus least-squares fitting problem on Riemannian manifolds. While some works focus on the necessary or sufficient conditions for spline existence (e.g., \cite{MSK10, MTM17}), others primarily develop numerical algorithms and offer practical solutions (e.g., \cite{BHM12, SASK12, MTM17, ACCS18, KDLS21}). The numerical algorithms presented in \cite{SASK12, MTM17} employ gradient steepest descent methods that leverage the Palais metric within the infinite-dimensional space of admissible curves. 

For a more comprehensive overview of the development of Riemannian splines and related interpolation problems, we refer interested readers to the recent survey articles \cite{CSC23, MP24}.


\subsection{Our Approach}

In this article, we develop a unified framework for addressing both Problem 1 and Problem 2. Our approach begins with the analysis of gradient flows associated with the variational spline interpolation problem, incorporating least-squares fitting on Riemannian manifolds (Problem 2). By introducing a suitable penalty term to replace the one in \eqref{eq:penalty}, we subsequently obtain solutions to the gradient flows corresponding to the variational spline interpolation problem without least-squares fitting (Problem 1). 
Namely, to take care of the penalty term in \eqref{eq:penalty}, we repalce it by  
\begin{align}
\label{eq:T-energy-chi} 
\frac{1}{\sigma^2} \sum\limits_{l=1}^{q-1} \mathcal{T} [\chi_{l}] 
= \frac{1}{2\sigma^2} \sum\limits_{l=1}^{q-1} 
\int_{I_l} | \chi_{l,x}(x)|^2 \, dx, 
\end{align} 
where 
\[
\mathcal{T} [\chi_{l}] 
= \frac{1}{2} \int_{I_l} | \chi_{l,x}(x)|^2 \, dx
\]
represents the tension energy functional, and 
$\chi_l$ is a map, representing a curve connecting $p_l$ and $\gamma(x_l)$ for each $l$. 
When $\chi_l$ is a critical point of the tension energy functional $\mathcal{T}$, it represents a geodesic in the Riemannian manifold, and $\mathcal{T}$ satisfies 
\begin{align*}
\mathcal{T}[\chi_l]=
\frac{1}{2 |I_l|} \left(\text{Length}(\chi_l)\right)^2 
\text{.} 
\end{align*}

Note that, if $\chi_l$ is a geodesic in a Riemannian manifold $M$ and $\text{Length}(\chi_l)$ is sufficiently small, then 
$\text{dist}_{M}(\chi_l(x=0),\chi_l(x=1))=\text{Length}(\chi_l)$.  
Thus, we investigate the energy functional, 
$\mathcal{E}_{k}^{\lambda,\sigma}: 
\Theta_{\mathcal{P}}
\rightarrow [0,\infty)$, defined by   
\begin{align} 
\label{eq:E_s-energy-f} 
\nonumber 
\mathcal{E}_{k}^{\lambda, \sigma} 
[(\gamma, \chi_{1}, \ldots, \chi_{q-1} )] 
&= \frac{1}{2} \sum\limits_{l=1}^{q}  
\int_{I_l} |D_{x}^{k-1} \gamma_{l,x}(x)|^2 \, dx 
+ \frac{\lambda}{2} \sum\limits_{l=1}^{q} \int_{I_l} | \gamma_{l,x}(x)|^2 \, dx 
\\ 
& ~~~~~~~~~~~~~~~~~~~~~~~~~~~~~~~~~~~~ 
+ \frac{1}{2\sigma^2} \sum\limits_{l=1}^{q-1} \int_{I_l} | \chi_{l,x}(x)|^2 \, dx, 
\end{align}
where $\Theta_{\mathcal{P}}$ is the set of admissible maps $(\gamma, \chi_{1}, \ldots, \chi_{q-1})$ defined in 
\S \ref{Subsec:2.2}, 
and $\lambda\in[0,\infty)$, $\sigma \in(0,\infty)$ are constants. 
Each element in $\Theta_{\mathcal{P}}$ corresponds to a union of suitably connected arcs, where the points ${p_0, \ldots, p_q}$ serve as endpoints of certain arcs. 
In fact, we will study the gradient flow associated with the energy functional \eqref{eq:E_s-energy-f}, which comprises a coupled system: a higher-order gradient flow for the map $\gamma$ and a harmonic map heat flow for each $\chi_l$, $l \in \{1, \ldots, q-1\}$. These components are intrinsically linked through their boundary conditions, leading to a nontrivial interaction between $\gamma$ and the auxiliary maps $\chi_l$.

In \S\ref{Sec:NPSN}, we formulate the variational spline interpolation problems and reformulate them as nonlinear parabolic systems. 
In \S\ref{Sec:linear}, we linearize these nonlinear systems to facilitate the application of Solonnikov's theory. 
Since the orders of parabolicity (see Remark \ref{rem:parabolicity order} for the meaning of this terminology) of the equations vary across different arcs of elements in $\Theta_{\mathcal{P}}$, with values ranging between 
$2k$ and $2$, we decompose the linearized system into two subsystems so that Solonnikov's framework can be applied appropriately to each. 
A major technical challenge in this analysis lies in the verification of the \emph{complementary conditions}—a subtle yet essential step to ensure the well-posedness of the system. Whether these complementary conditions are satisfied depends critically on how the boundary conditions are prescribed. 
The (geometric) \emph{compatibility conditions} on the initial data at the boundary, as defined in Definitions \ref{def:cc-order-p-fitting} and \ref{def:geom-cc-order-p-fitting}, play a key role in achieving the desired regularity of solutions up to the initial time for each arc of an element of $\Theta_{\mathcal{P}}$ for the solutions of the coupled linear parabolic system in \S\ref{Sec:linear}. These compatibility conditions are also crucial for implementing the bootstrapping argument and applying Solonnikov's theory, both of which are necessary for establishing the higher regularity of local solutions to the nonlinear parabolic system in \S\ref{Sec:STE}.

Following the interpolation problem in \cite[(2.12)]{GHMRV12-1} and the study of Riemannian cubics with tension, as investigated in various works (e.g., \cite{PN12}), we likewise introduce the parameter $\lambda \ge 0$ in the definition of the energy functional $\mathcal{E}^{\lambda, \sigma}_{k}$ in \eqref{eq:E_s-energy-f} (and similarly in $\mathcal{E}^{\lambda}_{k}$ in \eqref{eq:E_s-energy}).   
For the case $\lambda = 0$, \cite[Lemma 2.15]{HRW19} provides an example demonstrating that a minimizing sequence for Riemannian splines may fail to recover the classical Riemannian cubics. 
This failure occurs because control over the curve parametrization is lost, demonstrating the critical role of the tension term in preserving regularity and compactness.  
Note that the velocities at endpoints in \cite[Lemma 2.15]{HRW19} are not fixed, which results in a loss of control over the curve reparametrization and ultimately leads to nonexistence. A closely related result has been recently established in \cite[Theorem 1.7]{CGP25} for the case $k = 2$, where the existence of Riemannian cubics can be recovered from a minimizing sequence by prescribing the velocities at the endpoints. In the present article, the fixed velocities and all derivatives up to order $k-1$ at the endpoints, together with the fact that the total energy is nonincreasing in time, provide a control over the parametrization speed and hence yield the required uniform upper bounds for the tension energy. Consequently, our main results remain valid for $\lambda \ge 0$, as kindly pointed out by the anonymous referee.

While Theorem \ref{thm:Main_Thm-fitting}, proved in \S \ref{Sec:Interpolation+Penalty}, establishes the existence of classical global solutions and their asymptotic limits for the gradient flow of the penalized energy functional $\mathcal{E}^{\lambda, \sigma}_{k}$ in \eqref{eq:E_s-energy-f} for each fixed penalty parameter $\sigma > 0$,
Theorem \ref{thm:Main_Thm-fitting-appl}, proved in \S \ref{Sec:penalty_method}, demonstrates that this framework provides a penalty method in constructing higher-order Riemannian splines for the interpolation problem. 
Specifically, we show that as the penalty parameter $\sigma$ from \eqref{eq:penalty} tends to zero (i.e., $\sigma \to 0$), the corresponding family of gradient flow solutions obtained via Theorem \ref{thm:Main_Thm-fitting} admits a convergent subsequence. This subsequence converges to a global solution of the gradient flow associated with the unpenalized energy functional \eqref{eq:E_s-energy}, subject to the prescribed initial and boundary conditions. Importantly, the asymptotic behaviour of this limiting solution as $t \to \infty$ yields a variational spline interpolating the given points on the Riemannian manifold. This convergence result thus provides a rigorous justification for interpreting our method as a penalty method in constructing higher-order Riemannian splines for the interpolation problem, thereby offering a complete resolution to Problem 1.


\subsection{Our Contribution}

To the best of our knowledge, the approach presented in this article is new in the literature on the existence and regularity of solutions to Riemannian spline interpolation problems.   
In Theorem \ref{thm:Main_Thm-fitting}, we establish the existence of global solutions in parabolic H\"{o}lder spaces for the gradient flow of the energy functional 
$\mathcal{E}^{\lambda, \sigma}_{k}: \Theta_{\mathcal{P}}\rightarrow\mathbb{R}$, defined in \eqref{eq:E_s-energy-f} (or see \S \ref{Subsec:2.2} for details). 
Note that each element $(\gamma, \chi_{1}, \ldots, \chi_{q-1})\in\Theta_{\mathcal{P}}$ 
is an admissible map  corresponding to a union of suitably connected arcs, referred to as a {\it network} later, on an $m$-dimensional smooth Riemannian manifold $M$. 
The manifold may be either compact or complete, provided it has uniformly bounded Riemannian curvature tensors in the $C^\infty$-topology. 
These solutions exhibit asymptotic limits that confirm the existence of solutions addressing the problem of variational spline interpolation versus least-squares fitting
on Riemannian manifolds (Problem 2).

We adapt and extend the techniques developed for the elastic flow of closed curves in \cite{DKS02}, open curves in \cite{Lin12}, and subsequently for networks in \cite{DLP19, DLP21, GMP19, GMP20}, to establish Theorem \ref{thm:Main_Thm-fitting}. A key distinction of our approach lies in the treatment of the boundary conditions. 
While the works in \cite{DLP19, DLP21, GMP19, GMP20} achieve only short-time existence and conditional long-time existence results, we succeed in establishing full long-time existence in Theorem \ref{thm:Main_Thm-fitting}. 
Furthermore, our results exceed the anticipated $C^1$-regularity for geometric $2$-splines discussed in 
\cite[p.419]{GHMRV12-1}, and improve upon the $C^{2,\frac{1}{2}}([\delta,1-\delta])$ interior regularity obtained in \cite{HRW19}. 
The constructive nature of our approach also indicates potential for practical applications, particularly in the development of effective numerical schemes.

In the remainder of this article, we denote by $\mathbb{N}_0 = \mathbb{N} \cup \{0\}$ the set of non-negative integers. To avoid potential confusion with parabolic H\"{o}lder spaces, we adopt a simplified notation for elliptic H\"{o}lder spaces: instead of using the standard form $C^{[\varrho],\{\varrho\}}(\bar{I})$, we write $C^{\varrho}(\bar{I})$, where $\varrho \in (0,\infty) \setminus \mathbb{Z}$, $\{\varrho\} = \varrho - [\varrho]$ denotes the fractional part of $\varrho$, and $[\varrho]$ is the greatest integer less than $\varrho$.

\bigskip

The following theorem addresses Problem 2. 

\begin{theorem} 
\label{thm:Main_Thm-fitting} 
Let $\ell \in \mathbb{N}_0$, $\alpha_j \in (0,1)$, and $\varrho_j=\ell+\alpha_j$, $j\in\{1,2\}$, such that $\alpha_1>k \alpha_2$. 
Assume that $M$ is a smooth $m$-dimensional Riemannian manifold, which is either compact without boundary or complete with uniformly bounded Riemannian curvature tensors in a $C^\infty$-topology. 
Suppose that 
$(\gamma_0, \chi_{1, 0}, \ldots, \chi_{q-1, 0})\in \Theta_{\mathcal{P}}$ is an initial datum 
fulfilling the geometric compatibility conditions of order $\ell$, 
as defined in Definition \ref{def:geom-cc-order-p-fitting}. 
Then, there exists a global solution 
$(\gamma, \chi_1, \ldots, \chi_{q-1})$ to 
\eqref{eq:E_s-flow-fitting}$\sim$\eqref{eq:BC-higher-order-5-fitting} 
with the regularity 
\begin{equation*}
\begin{cases} 
\gamma_{l} 
\in C^{\frac{2k+\varrho_1}{2k},2k+\varrho_1} 
\left([0,\infty)\times [x_{l-1}, x_{l}]\right) \bigcap C^{\infty}
\left((0,\infty)\times [x_{l-1}, x_{l}]\right), ~~~~~~~~~~\forall\, l\in\{1, \ldots, q\},\\
\chi_{l} 
\in C^{\frac{2+\varrho_2}{2},2+\varrho_2} 
\left([0,\infty)\times [x_{l-1}, x_{l}]\right) \bigcap C^{\infty}
\left((0,\infty)\times [x_{l-1}, x_{l}]\right), ~~~~~~~\forall\, l\in\{1, \ldots, q-1\},\\
\gamma(t,\cdot) \in C^{2k-2}([x_{0}, x_{q}]), ~~~~~~~~~~~~~~~~~~~~~~~~~~~~~~~~~~~~~~~~~~~~~~~~~~~~~~~~~~~~ \forall\, t\in[0,\infty). 
\end{cases}
\end{equation*} 
Furthermore, regarding the asymptotics, there exists a pair $(\gamma_{\infty}, \chi_{1, \infty}, \ldots, \chi_{q-1, \infty}) \in \Theta_{\mathcal{P}}$ with the regularity properties 
$\gamma_{\infty} \in C^{2k-2}([x_{0}, x_{q}])$, 
$\gamma_{l,\infty} \in C^{\infty} \left([x_{l-1}, x_{l}]\right)$, 
and $\chi_{l, \infty} \in C^{\infty} \left([x_{l-1}, x_{l}]\right)$, 
$\forall\, l$, obtained from a convergent subsequence such that $\gamma_{\infty}(\cdot)=\underset{t_j\rightarrow\infty}\lim \gamma(t_j,\cdot)$, 
$\chi_{l, \infty}(\cdot)=\underset{t_j\rightarrow\infty}\lim \chi_{l}(t_j,\cdot)$. 
Moreover, $\gamma_{l,\infty}$ satisfies \eqref{1.2}, thereby representing a Riemannian spline, while the map $\chi_{l,\infty}$ behaves as a geodesic 
for each $l$, and $(\gamma_{\infty}, \chi_{1,\infty}, \ldots, \chi_{q-1,\infty})$ resolves the variational spline interpolation versus least-squares fitting on Riemannian manifolds.  
\end{theorem}

\smallskip

The following theorem addresses Problem 1. 

\begin{theorem}
\label{thm:Main_Thm-fitting-appl} 
Let $(\gamma_0,\chi_{1,0}, \ldots, \chi_{q-1,0})=(\gamma^{\sigma}_0,\chi^{\sigma}_{1,0}, \ldots, \chi^{\sigma}_{q-1,0})\in\Theta_{\mathcal{P}}$ for all $\sigma\in(0,1)$ represent the initial data, where 
$\gamma_{l,0}=(\gamma_0)_{\lfloor I_l}$ satisfies the regularity 
$\gamma_{l,0}\in C^{\infty}(\bar I_l)$ 
and the geometric compatibility conditions of any order $\ell \in\mathbb{N}_0$, 
$\forall\, l\in\{1, \ldots, q\}$. 
Furthermore, we assume that the initial data 
$\chi_{l,0}=p_l$, for all $l\in\{1, \ldots, q-1\}$, are constant maps. 
Let $\big(\gamma^{\sigma}, \chi^{\sigma}_1, \ldots, \chi^{\sigma}_{q-1} \big)$ denote a global solution obtained from 
Theorem \ref{thm:Main_Thm-fitting} for any fixed $\sigma\in(0,1)$. 
Then, the following statements hold: 
\begin{itemize}
\item[\bf (i)] 
There exists a subsequence 
$\{\gamma^{\sigma_j}\}_{j\in\mathbb{N}}$ and its limit $\gamma^{0}$ 
such that $\gamma^{\sigma_j}_{l}, \gamma^{0}_{l}\in C^{\infty}([0,\infty)\times \bar{I}_l])$, $\forall\, l\in\{1, \ldots, q\}$,   
\begin{align}
\label{eq:gamma^0_in_C^s}
\lim_{j\rightarrow\infty}\|\gamma^{\sigma_j}_{l}-\gamma^{0}_{l}\|_{C^{\mu-2}([0,T]\times\bar{I}_{l})}=0, 
~~~~~ \forall\, \mu\in \mathbb{N}\setminus\{1\}, \, \forall\, l\in\{1, \ldots, q\}, \forall\, T>0, 
\end{align} 
and  
$\gamma^{0}=\big(\gamma^{0}_{1}, \ldots, \gamma^{0}_{q} \big)$ is a classical global solution 
to the higher-order gradient flow \eqref{eq:E_s-flow} 
and satisfies the initial-boundary conditions demonstrated 
in \eqref{eq:IC-higher-order}$\sim$\eqref{eq:BC-higher-order-3}. 

\item[\bf (ii)]  
The asymptotic limit 
$\gamma^{0}_{\infty}(\cdot)
=\lim\limits_{j\rightarrow\infty}\gamma^{0}(t_j,\cdot)$, for some sequence $t_j\rightarrow+\infty$, 
fulfills 
$\gamma^{0}_{l,\infty}\in C^{\infty}(\bar I_l), 
\,\forall\, l\in\{1, \ldots, q\}$,  
and $\gamma^{0}_{\infty}\in C^{2k-2}(\bar I)$ 
resolves the variational spline interpolation problem on Riemannian manifolds. 
\end{itemize} 

\end{theorem}

Our results, Theorems \ref{thm:Main_Thm-fitting} and \ref{thm:Main_Thm-fitting-appl}, rigorously establish the existence of solutions to the problems of higher-order spline interpolation on Riemannian manifolds, both with and without least-squares fitting. 
These results address significant open questions and challenges previously posed in the literature, including those in \cite{GHMRV12-1, GHMRV12-2, HRW19, MSK10}. Specifically, Theorems \ref{thm:Main_Thm-fitting} and \ref{thm:Main_Thm-fitting-appl} yield globally $C^{2k-2}$-smooth and piecewise smooth classical solutions to the geometric $k$-spline interpolation problem on any smooth Riemannian manifold—either compact or complete with uniformly bounded Riemannian curvature tensors in the $C^\infty$-topology. Consequently, these two higher-order interpolation problems can be effectively resolved through gradient flow methods and the penalty approach, which we collectively refer to as a unified gradient flow framework in this article.


Although our approach establishes compactness results for obtaining global solutions to the parabolic flow and determining asymptotic limits -- identified as the so-called geometric $k$-splines -- by extracting a convergent subsequence of $\gamma(t_j,\cdot)$, an open question remains:
Can the uniqueness of the gradient flow convergence for geometric $k$-splines be ensured when passing to the asymptotic limits? 
Given that elasticae can be interpreted as nonlinear $2$-splines or nonlinear Riemannian cubics in an analogous manner, e.g., see \cite{GJ82}, 
a promising approach to addressing this uniqueness question is to employ techniques derived from {\L}ojasiewicz–Simon gradient inequalities, e.g., see \cite{DPS16, MP21}. These techniques have also been effectively utilized in analyzing the convergence of elastic flows or $p$-elastic flows on analytic Riemannian manifolds (e.g., see \cite{Pozzetta22} for valuable insights along this direction.). We leave this investigation for future research.

\subsection{The Structure of Article} 


The remainder of this article is organized as follows.
\S \ref{Sec:problem_formulation} introduces the necessary mathematical background, establishes our notation, and formulates the two principal types of spline interpolation problems on Riemannian manifolds found in the literature. 
\S \ref{Sec:Interpolation+Penalty}, which constitutes one of our main contributions, reformulates the spline interpolation problem as a nonlinear parabolic system and outlines the proof strategy for Theorem \ref{thm:Main_Thm-fitting}. This result resolves the variational spline interpolation problem with least-squares fitting on Riemannian manifolds (Problem 2).
In \S \ref{Sec:penalty_method}, we apply Theorem \ref{thm:Main_Thm-fitting} to address the exact interpolation problem on Riemannian manifolds via a penalty method, thereby establishing Theorem \ref{thm:Main_Thm-fitting-appl} and completing our unified framework.
Finally, \S \ref{Sec:Appendix} provides essential supplementary material, including terminology, notation, key technical lemmas, and a concise summary of Solonnikov's theory for linear parabolic systems, following \cite{Solonnikov65}.


\section{Higher-order Riemannian Spline Interpolation Problems } 
\label{Sec:problem_formulation}

In this section, we introduce two higher-order Riemannian spline interpolation problems.
We provide the formulation, including relevant terminology and setup, for solving these problems using gradient flows.

Let $M$ denote a smooth $m$-dimensional Riemannian manifold. 
Suppose $X$, $Y$, and $Z$ are smooth vector fields in $M$, and define the covariant derivative of $Y$ with respect to $X$ by $D_X Y$. 
The Riemannian curvature tensor $R(\boldsymbol{\cdot}, \boldsymbol{\cdot}) \boldsymbol{\cdot}: TM \times TM \times TM \rightarrow TM$ is defined by the identity, $R(X,Y) Z=D_X D_Y Z- D_Y D_X Z- D_{[X,Y]} Z$, where $[X,Y]$ denotes the Lie bracket of the vector fields $X$ and $Y$.  
Denote by $D$ the Levi-Civita connection associated with $M$. 
We use the notation 
\begin{align}
\label{def:R_M}
R_M(\gamma)(\boldsymbol{\cdot},\boldsymbol{\cdot})\boldsymbol{\cdot}
=(R\circ\gamma) (\boldsymbol{\cdot},\boldsymbol{\cdot})\boldsymbol{\cdot}
\end{align} 
to represent the Riemannian curvature tensors at $\gamma\in M$.

Thanks to Nash's isometric embedding theorem, we may assume that $M$ is smoothly and isometrically embedded in 
$\mathbb{R}^n$ for some $n>m$. 
Denote by $\mathrm{I}_{\gamma}(\boldsymbol{\cdot},\boldsymbol{\cdot}):T_\gamma M\times T_\gamma M\rightarrow\mathbb{R}$ 
the first fundamental form of $M$ at $\gamma\in M$, 
and by 
$\Pi_{\gamma}(\cdot,\cdot): 
T_{\gamma}M\times T_{\gamma}M\rightarrow (T_{\gamma}M)^\perp$ 
the second fundamental form of $M$ at $\gamma\in M$. 
For differentiation of the first and second fundamental forms at $\gamma\in M$ with respect to the coordinates $(u_1,\cdots,u_n) \in \mathbb{R}^n$, they are expressed as 
$\partial^\sigma I_{\gamma}$ and $\partial^\sigma \Pi_{\gamma}$, respectively, 
where $\partial^{\sigma}=\partial_{u_1}^{\sigma_1} \cdots \partial_{u_n}^{\sigma_n}$ 
with $\sigma=(\sigma_1, \cdots, \sigma_n)$, $|\sigma|=\sum\limits_{k=1}^{n}\sigma_k$, and $\sigma_k\ge 0$ for each $k$. 
We let $C(M)>0$ be a universal constant for the upper bounds for the $\sup$-norm of any order of derivatives of the first and second fundamental forms of $M$, i.e., 
\begin{align}
\label{eq:C_infty-bdd_for_I+II}
\|\partial^{\sigma}\mathrm{I}_{M}\|_{C^{0}(M)}
\le C(M), 
~~~ 
\|\partial^{\sigma}\mathrm{\Pi}_{M}\|_{C^{0}(M)}
\le C(M), 
\end{align}
for any $\sigma$ with $|\sigma|\ge 0$.  
Since the universal constant $C(M)$ will only be used in the estimates in obtaining the short-time existence of this paper, it can be viewed as a locally required assumption on the smooth immersion of $M$ into a Euclidean space. 

\smallskip

When $\gamma: (a,b)\subset\mathbb{R} \rightarrow M \subset \mathbb{R}^n$ is a smooth map, 
denote the directional covariant derivative associated with the Levi-Civita connection $D$ by  $D_{\partial_x\gamma}$ or simply as $D_x$. 
The induced velocity 
$\partial_x \gamma: I \subset \mathbb{R}\rightarrow T_\gamma M$ 
and its restriction $\partial_x \gamma_l: I_l \subset \mathbb{R}\rightarrow T_\gamma M$ are denoted by  $\gamma_x$ and $\gamma_{l,x}$, respectively. 
The covariant differentiation can be expressed as 
\begin{align}
\label{eq:co-deri}
D_x V= 
\partial_x V- \Pi_\gamma(V, \partial_x\gamma), 
\end{align} 
where 
$V: (a,b) \rightarrow T_\gamma M \subset T_\gamma \mathbb{R}^n$ is a smooth vector field along $\gamma$, 
$\Pi_\gamma(\boldsymbol{\cdot},\boldsymbol{\cdot}):T_\gamma M\times T_\gamma M\rightarrow(T_\gamma M)^\perp$ 
represents the second fundamental form at 
$\gamma\in M\subset \mathbb{R}^n=\{(u_1, \ldots, u_n): u_j \in\mathbb{R}, \forall\, j\in\{1, \ldots, n\} \}$. 
We use \eqref{eq:co-deri} to convert covariant differentiation $D_x$ into partial differentiation $\partial_x$. 
From \eqref{eq:co-deri}, we can inductively deduce 
\begin{align}
\label{eq:co-deri-1} 
D_x^i \partial_x\gamma = \partial_x^{i+1}\gamma 
+ W_i(\partial_x^{i}\gamma, \cdots, \partial_x\gamma, \gamma), 
\end{align}  
where 
$W_i:\mathbb{R}^n \times \cdots \times \mathbb{R}^n \times M 
\cong \mathbb{R}^{in} \times M
\rightarrow \mathbb{R}^n$, $\forall\, i\in\mathbb{N}$,  
and $W_0(\gamma)=0$.  
From \eqref{eq:co-deri-1} 
and \eqref{eq:C_infty-bdd_for_I+II}, we have 
\begin{align}
\label{est:W_i-2nd-fund-form}
\|W_i\|_{C^{0}(\mathbb{R}^n \times \cdots \times \mathbb{R}^n \times M)} 
\le C(n,i, C(M)) 
\cdot \|\gamma\|_{C^{i}(M)}  
\text{.}
\end{align} 
Note that, when the domain of a continuously differentiable map is closed, e.g., $\gamma \in C^r([a,b]; M)$, or simply denoted as $C^r([a,b])$, where $r\in\mathbb{N}_{0}$, it means that its derivatives 
$\partial_{x}^\ell \gamma$ are continuous up to the boundary of $[a,b]$, $\forall\, \ell \in\{0,1,\cdots, r\}$.

\subsection{The Spline Interpolation on Riemannian Manifolds } 

For the problem of spline interpolation on Riemannian manifolds,  
let the set of admissible maps be 
\begin{align*} 
\Phi_{\mathcal{P}}
=&\big{ \{ }
\gamma:\bar I\rightarrow M 
\text{ \textbar} \, 
\gamma\in C^{2k-2}(\bar I), 
\gamma_{l}:=\gamma_{\lfloor \bar{I}_l} \in C^{2k+\varrho_1 }(\bar{I}_l), 
I=(x_0,x_q), I_l=(x_{l-1},x_l),  
\\ 
&~~~~~~~~~~~~~~~~~~ x_l=l \in \{0, 1, \ldots, q\}, 
\gamma(x_l)=p_l, \forall\, l \in \{0, 1, \ldots, q\}, 
\\ 
&~~~~~~~~~~~~~~~~~~ D^{\mu-1}_x \partial_x \gamma_l(x^{\ast})=v^{\mu}_{x^{\ast}}, \forall\, (l,x^{\ast}) \in \{(1,x_0), (q, x_q)\}, \forall\, \mu \in \{1, \ldots, k-1\} 
\big{ \} }, 
\end{align*} 
where $\varrho_1 \in (0,\infty) \setminus \mathbb{Z}$,  $v^{\mu}_{x^{\ast}}$ are constant vectors, and $k\ge 2$ is a positive integer. 
The image of each element in $\Phi_{\mathcal{P}}$ represents a piecewise smooth curve connecting the points in $\mathcal{P}$ in the order of their indices. 
Define the energy functional  
$\mathcal{E}_{k}^{\lambda}: 
\Phi_{\mathcal{P}}
\rightarrow [0,\infty)$ by 
\begin{align}
\label{eq:E_s-energy}
\mathcal{E}_{k}^\lambda [\gamma] 
&=\mathcal{E}_{k} [\gamma]+\lambda\cdot \mathcal{T} [\gamma],  
\end{align} 
where $\lambda\ge 0$ is a non-negative constant, $\mathcal{E}_k$ is defined in \eqref{eq:E-energy}, and $\mathcal{T}$ is a tension energy functional defined by 
\begin{align*} 
\mathcal{T} [\gamma] 
&= \sum\limits_{l=1}^{q} \frac{1}{2} 
\int_{I_l} | \gamma_{l,x}(x)|^2 \, dx. 
\end{align*}  
Here $|\cdot|=\sqrt{\langle\cdot,\cdot\rangle}$ denotes the norm in the tangent bundle $TM$, 
and $\langle \boldsymbol{\cdot},\boldsymbol{\cdot} \rangle$ is the inner product induced from the metric of $M$. 
When $\gamma\in \Phi_{\mathcal{P}}$,  
we apply the first variation formula in Lemma \ref{first-variation-higher-order} of \S\ref{Sec:Appendix} Appendix to obtain  
\begin{align} 
\nonumber 
&\frac{d}{d\varepsilon}_{\lfloor\varepsilon=0} 
\mathcal{E}_{k}^{\lambda}[\gamma^\epsilon]
= \sum_{l=1}^{q}\int_{I_l}\, \langle -\mathcal{L}_x^{2k}(\gamma_l), w \rangle \, dx +\sum^{k}_{\ell=1} \sum^{q}_{l=1}(-1)^{\ell-1}  
\bigg{\langle} D^{k-\ell}_{x}w, 
D^{k+\ell-2}_{x}\partial_{x} \gamma_{l} \bigg{\rangle} \bigg{|}^{x^{-}_{l}}_{x^{+}_{l-1}} 
\\ 
& ~~~~~~~~~~ +\sum^{k-1}_{\mu=2}\sum^{k-\mu}_{\ell=1}\sum^{q}_{l=1} (-1)^{\ell-1}  \bigg{\langle} 
D^{k-\mu-\ell}_{x} 
\left[R(w, \partial_{x} \gamma_{l}) D^{\mu-2}_{x}\partial_{x} \gamma_{l}\right], 
D^{k+\ell-2}_{x}\partial_{x} \gamma_{l} 
\bigg{\rangle} \bigg{|}^{x^{-}_{l}}_{x^{+}_{l-1}} 
\text{,} 
\label{eq:E-L-1}
\end{align} 
where $w$ denotes a collection of tangent vector fields, each defined on a corresponding arc $\gamma_{l}$. 
This collection is induced by a family of continuously differentiable admissible maps 
\begin{align*} 
\{\gamma^\epsilon\}_{\epsilon\in(-1,1)}\subset \Phi_{\mathcal{P}}   
\end{align*} 
satisfying 
$\gamma^0=\gamma$ and 
$w=\frac{d}{d\epsilon}_{\lfloor \epsilon=0} \gamma^\epsilon$.  
Note that since $\gamma\in \Phi_{\mathcal{P}}$, the condition 
$\gamma\in C^{2k-2}([x_0,x_q])$ ensures the so-called higher-order concurrency boundary conditions at the junction points in \eqref{eq:BC-higher-order-3}. 
Together with other boundary conditions in \eqref{eq:BC-higher-order-1}$\sim$\eqref{eq:BC-higher-order-3}, 
the sum of all boundary terms in \eqref{eq:E-L-1} vanishes due to cancellation.  
Note that the differential operator $\mathcal{L}_x^{2k}$ is given by 
\begin{align}
\label{Euler-eq-RPT}
\mathcal{L}_x^{2k}(\gamma_l)
:=(-1)^{k+1} D_{x}^{2k-1} \gamma_{l,x} 
+ \sum^{k}_{\mu=2} (-1)^{\mu+k+1} R\left(D^{2k-\mu-1}_{x}  \gamma_{l,x}, D^{\mu-2}_{x}  \gamma_{l,x} \right) \gamma_{l,x} +\lambda \cdot D_{x} \gamma_{l,x}
\text{,}
\end{align} 
which combines higher-order derivatives and curvature terms and  characterizes the Euler–Lagrange equations associated with the energy functional. 

Thus, if $\gamma_l $ is the critical point of $\mathcal{E}_{k}^{\lambda}$, then 
\begin{align}\label{1.1} 
\mathcal{L}_x^{2k}(\gamma_l)=0 
~~ \text{(intrinsic formulation)}  
~~ \text{ or }~~
\mathcal{L}_x^{2k}(\gamma_l) \perp T_{\gamma_l} M 
~~ \text{(extrinsic formulation)}. 
\end{align}

For any fixed $l$, $\gamma_l$ is called the geometric $k$-spline or Riemannian polynomial spline. As $k=2$, it is called a cubic spline or Riemannian cubic. The critical points of $\mathcal{E}_k^\lambda$ could also be viewed as poly-harmonic maps (as $k\ge2$) or biharmonic maps (as $k=2$) with tension energy. Moreover, we may call all of these critical points the higher-order Riemannian splines in the rest of the article (and Riemannian cubics as $k=2$). 
Thus, one could view the so-called Riemannian spline interpolations as the problem of finding the poly-harmonic maps or biharmonic maps $\gamma:[a,b]\rightarrow M$ such that $\gamma(x_l)=p_l\in M$, $\forall\, l\in \{0, \ldots, q\}$, and $a=x_0<x_1<\ldots<x_q=b$. 
 
\smallskip 

{\bf The Intrinsic Geometrical Parabolic System for Splines.} 
Let $\gamma_0 \in \Phi_{\mathcal{P}}$ be an initial curve; then, from the first variation formulae 
(see Lemma 
\ref{first-variation-higher-order} in \S\ref{Sec:Appendix} Appendix),    
the gradient flow of $\mathcal{E}_{k}^{\lambda}$ can be set up as, 
\begin{equation}
\label{eq:E_s-flow} 
\partial_t \gamma_{l}= \mathcal{L}_x^{2k}(\gamma_l), ~~~~~~~~~~~~~~~~~~~~~~~~~~~~~~~~~~~~~~~~~~~~~~~~~~~~~~~~~~~~~~
\forall ~ l\in\{1,\cdots,q\} 
\text{,}
\end{equation}
with the initial-boundary conditions, 
\begin{align} 
&\text{(initial datum)}
\nonumber
\\ 
&\label{eq:IC-higher-order} 
\gamma(0,x)=\gamma_{0}(x)\in \Phi_{\mathcal{P}}, 
~~~~~~~~~~~~~~~~~~~~~~~~~~~~~~~~~~~~~~~~~~~~~~~~~~~~~~~~~~~~~ x\in [x_0, x_q], 
\\
&\text{(Dirichlet boundary conditions)}
\nonumber
\\ 
&\label{eq:BC-higher-order-1} 
\gamma(t,x_{l})=p_{l}, ~~~~~~~~~~~~~~~~~~~~~~ l \in \{0, \ldots, q\}, ~~~~~~~~~~~~~~~~~~~~~~~~~~~~~~~~~~~~ t\in[0,T],  
\\
 &\label{eq:BC-higher-order-2} 
D_{x}^{\mu-1}\partial_x \gamma_l (t,x^{\ast})=v^{\mu}_{x^{\ast}} ~~~~~~~~~  (l,x^{\ast})\in \{(1, x_0), (q, x_q)\}, ~ 1\leq \mu \leq k-1, ~ t\in[0,T], 
\\ 
&\text{(The higher-order concurrency boundary conditions)} 
\nonumber
\\ 
&\label{eq:BC-higher-order-3}
[\Delta_{l} D_{x}^{\mu-1}\partial_x \gamma] (t)=0, 
~~~~~~~~~~~~~~ l \in \{1, \ldots, q-1\}, ~~~~~~~ 1 \leq \mu \leq 2k-2, ~ t\in[0,T], 
\end{align} 
where 
$\gamma_0$ denotes the initial data of $\gamma$; 
the boundary datum $\{v^{\mu}_{x^{\ast}}\}$ in 
\eqref{eq:BC-higher-order-2} are prescribed constant vectors; 
and 
\begin{align}\label{def:Delta_l}
[\Delta_{l} D_{x}^{\mu-1}\partial_x \gamma](t) 
:=D_{x}^{\mu-1}\partial_x \gamma_{l+1}(t,x_{l})
-D_{x}^{\mu-1}\partial_x \gamma_{l}(t,x_{l}). 
\end{align}  
The conditions, 
\eqref{eq:BC-higher-order-1} and \eqref{eq:BC-higher-order-3}, 
aim to keep $\gamma(t,\cdot)\in \Phi_{\mathcal{P}}$, 
for any fixed $t$.

\subsection{The Spline Interpolation versus Least-Squares Fitting on Riemannian Manifolds } 
\label{Subsec:2.2}

For the problem of spline interpolation versus least-squares fitting on Riemannian manifolds, let the set of admissible maps be 

\begin{align} 
\label{def:Theta_P} 
\Theta_{\mathcal{P}}
=& \big{ \{ } (\gamma, \chi_{1},\cdots,\chi_{q-1}) \text{ \textbar} \, 
\gamma:\bar I\rightarrow M, ~ \chi_{l}: \bar I_l \to M, 
I=(x_0,x_q), I_l=(x_{l-1},x_{l}), 
\\ \nonumber 
& ~~~~
x_l=l\in \{0, 1, \ldots, q\}, 
\gamma\in C^{2k-2}(\bar I), 
~ \gamma_{l}:=\gamma_{\lfloor \bar I_l}\in C^{2k+\varrho_1}(\bar I_l), ~ 
\chi_l \in C^{2+\varrho_2}(\bar I_l),  
\\ \nonumber
& 
\gamma(x_0)=p_0, \gamma(x_q)=p_q, 
\chi_{l}(x_{l-1})=p_{l}, \gamma_l(x_l)=\chi_{l}(x_l)=\gamma_{l+1}(x_l), \forall\, l \in \{1, \ldots, q-1\}, 
\\ \nonumber 
& ~~~~~~~~~~~~~~~~~~~~~~ 
D^{\mu-1}_x \partial_x \gamma_l(x^{\ast})=v^{\mu}_{x^{\ast}}, \forall\, \mu \in \{1, \ldots, k-1\}, 
(l,x^{\ast}) \in \{(1,x_0), (q, x_q)\} 
\big{\}}
\end{align} 
where $\varrho_j \in (0,\infty) \setminus \mathbb{Z}$, $[\varrho_1]=[\varrho_2]$, 
and $k\ge 2$ is a positive integer. 
The image of each element in $\Theta_{\mathcal{P}}$ represents a network connecting the points in $\mathcal{P}$ in the order of their indices.

When $(\gamma, \chi_{1}, \ldots, \chi_{q-1})\in \Theta_{\mathcal{P}}$,  
we apply the first variation formula in Lemma \ref{first-variation-higher-order} to obtain the first variation of the energy functional $\mathcal{E}_{k}^{\lambda,\sigma}$,   
\begin{align} 
\nonumber 
\frac{d}{d\varepsilon}_{\lfloor\varepsilon=0} & \mathcal{E}_{k}^{\lambda,\sigma}
[(\gamma^\epsilon, \chi_1^\epsilon, \ldots, \chi_{q-1}^\epsilon)] 
= \sum_{l=1}^{q}\int_{I_l}\, 
\langle -\mathcal{L}_x^{2k}(\gamma_l), w_{\lfloor \gamma_l} \rangle \, dx 
+\frac{1}{\sigma^2}\sum_{l=1}^{q-1}\int_{I_l}\, 
\langle -D_x \partial_x \chi_{l}, w_{\lfloor \chi_l} \rangle \, dx 
\\ \nonumber  
~~~ & +\sum^{k}_{\ell=1} \sum^{q}_{l=1}(-1)^{\ell-1}  
\bigg{\langle} D^{k-\ell}_{x}w_{\lfloor \gamma_l}, 
D^{k+\ell-2}_{x}\partial_{x} \gamma_{l} \bigg{\rangle} \bigg{|}^{x^{-}_{l}}_{x^{+}_{l-1}} 
\\ \nonumber 
~~~ &+\sum^{k-1}_{\mu=2}\sum^{k-\mu}_{\ell=1}\sum^{q}_{l=1} (-1)^{\ell-1}  \bigg{\langle} 
D^{k-\mu-\ell}_{x} 
\left[R(w_{\lfloor \gamma_l}, \partial_{x} \gamma_{l}) D^{\mu-2}_{x}\partial_{x} \gamma_{l}\right], 
D^{k+\ell-2}_{x}\partial_{x} \gamma_{l} 
\bigg{\rangle} \bigg{|}^{x^{-}_{l}}_{x^{+}_{l-1}} 
\\ 
&+\lambda \sum_{l=1}^{q} \langle w_{\lfloor \gamma_l},  \partial_{x} \gamma_{l} \rangle \bigg{\vert}^{x_{l}}_{x_{l-1}}+\frac{1}{\sigma^2} \sum_{l=1}^{q} \langle w_{\lfloor \chi_l},  \partial_{x} \chi_{l} \rangle \bigg{\vert}^{x_{l}}_{x_{l-1}} 
\text{,} 
\label{eq:E-L-2}
\end{align} 
where the operator $\mathcal{L}_x^{2k}(\gamma_l)$ is defined in \eqref{Euler-eq-RPT}, 
and  
$w=\frac{d}{d\epsilon}_{\lfloor \epsilon=0} (\gamma^\epsilon, \chi_1^\epsilon, \ldots, \chi_{q-1}^\epsilon)$ 
denotes the variation of the corresponding arc of the network,  
$(\gamma, \chi_1, \ldots, \chi_{q-1})$, 
induced by a family of continuously differentiable admissible maps   
\begin{align*} 
\{(\gamma^\epsilon, \chi_1^\epsilon, \ldots, \chi_{q-1}^\epsilon)\}_{\epsilon\in(-1,1)}\subset \Theta_{\mathcal{P}},  
\end{align*} 
satisfying the condition 
\begin{align*}
(\gamma^0, \chi_1^0, \ldots, \chi_{q-1}^0)=(\gamma, \chi_1, \ldots, \chi_{q-1}). 
\end{align*} 
Note, when $(\gamma, \chi_{1}, \ldots, \chi_{q-1}) \in \Theta_{\mathcal{P}}$, it implies that 
$\gamma \in C^{2k-2}([x_0, x_q])$, which is equivalent 
to satisfying both the concurrency boundary conditions given in \eqref{eq:BC-higher-order-2-fitting} and the higher-order concurrency boundary conditions at the junction points given in \eqref{eq:BC-higher-order-4-fitting}. 
Together with other boundary conditions in \eqref{eq:BC-higher-order-1-fitting}$\sim$\eqref{eq:BC-higher-order-5-fitting}, 
the sum of all boundary terms in \eqref{eq:E-L-2} vanishes due to cancellation. 

The critical point condition, obtained from the first variation being zero, indicates the balance of forces for $\gamma$ 
and $(\chi_{1}, \ldots, \chi_{q-1})$ under the influence of energy functionals. 
Thus, if $(\gamma, \chi_{1}, \ldots, \chi_{q-1})$ is the critical point, then 
\begin{align}
\label{1.2} 
\begin{cases}
\mathcal{L}_x^{2k}(\gamma_l)=0 
~\text{(intrinsic formulation)  }   
~\text{or }
\mathcal{L}_x^{2k}(\gamma_l) \perp T_{\gamma_l} M 
~ \text{(extrinsic formulation)},
\\
D_x \partial_x \chi_{l}=0 
~\text{(intrinsic formulation) }   
~\text{or } 
D_x \partial_x \chi_{l} \perp T_{\chi_l} M 
~ \text{(extrinsic formulation)}. 
\end{cases}
\end{align} 
This expression shows that the critical points of the energy functional $\mathcal{E}_{k}^{\lambda,\sigma}$ for the pair $(\gamma, \chi_{1}, \ldots, \chi_{q-1})$ in the admissible class of maps $\Theta_{\mathcal{P}}$ satisfy a system of coupled equations. The maps, $\gamma$ and $(\chi_{1}, \ldots, \chi_{q-1})$, are linked primarily through boundary conditions at the network's interior points, as they each possess their own energy functionals.

\bigskip

{\bf The Intrinsic Geometrical Parabolic System for Networks.} 
From the first variation formulae in 
Lemma \ref{first-variation-higher-order} and \eqref{1.2}, 
let the gradient flow of $\mathcal{E}_{k}^{\lambda, \sigma}$ be 
\begin{equation} 
\label{eq:E_s-flow-fitting} 
\begin{cases}
\partial_t \gamma_{l}= \mathcal{L}_x^{2k}(\gamma_l), ~~~~~~~~~~~~~~~~~~~~~~~~~~~~~~~~~~~~~~~~~~~~~~~~~~~~~~~~~ 
\forall ~ l\in\{1,\cdots,q\}, 
\\ 
\partial_t \chi_{l}= 
\sigma^2 D_x\partial_{x}\chi_{l}, ~~~~~~~~~~~~~~~~~~~~~~~~~~~~~~~~~~~~~~~~~~~~~~~~~
\forall ~ l\in\{1, \ldots, q-1\},
\end{cases}
\end{equation} 
with the initial-boundary conditions: 
\begin{align} 
\nonumber 
&\text{(initial datum)}
\\
&\label{eq:IC-higher-order-fitting} 
(\gamma, \chi_{1}, \ldots, \chi_{q-1})(0,\cdot)=(\gamma_0(\cdot), (\chi_{1,0}, \ldots, \chi_{q-1,0})(\cdot) )\in \Theta_{\mathcal{P}}, 
\\
\nonumber 
&\text{(Dirichlet boundary conditions)}
\\
&\label{eq:BC-higher-order-1-fitting} 
\gamma(t,x_{l})=p_{l}, ~~~~~~~~~~~~~~~~~~~~~~~~~~~~~~~~~~~~~~~~~~~~~~~~~~~~~~~~~~~~ l \in \{0,q\}, ~~ t\in[0,T],  
\\
&
\label{eq:BC-higher-order-3-fitting} 
D_{x}^{\mu-1}\partial_x \gamma_l (t,x^{\ast})=v^{\mu}_{x^{\ast}} ~~~~~~~~~  (l,x^{\ast})\in \{(1, x_0), (q, x_q)\}, ~ 1\leq \mu \leq k-1, ~ t\in[0,T], 
\\ 
&\label{eq:BC-2-order-1-fitting} 
\chi_{l}(t,x_{l-1})=p_{l}, ~~~~~~~~~~~~~~~~~~~~~~~~~~~~~~~~~~~~~~~~~~~~~~ l \in \{1,\ldots, q-1\}, ~~t\in[0,T],  
\\
\nonumber 
&\text{(concurrency boundary conditions at the junction points)}
\\
 &\label{eq:BC-higher-order-2-fitting}
\gamma_{l}(t,x_{l})=\chi_{l}(t,x_{l}) =\gamma_{l+1}(t,x_{l})
~~~~~~~~~~~~~~~~~~~~~~~~~~ l \in \{1, \ldots, q-1\}, ~~ t\in[0,T], 
\\ 
\nonumber 
&\text{(The higher-order concurrency boundary conditions at the junction points)}
\\ 
&\label{eq:BC-higher-order-4-fitting}
[\Delta_{l} D_{x}^{\mu-1}\partial_x \gamma](t)=0, 
~~~~~~~~~~~~~~~~~~ l \in \{1, \ldots, q-1\}, ~~ 1 \leq \mu \leq 2k-2, ~ t\in[0,T], 
\\
\nonumber 
&\text{(balancing conditions at the junction points)}
\\ 
&\label{eq:BC-higher-order-5-fitting}
(-1)^{k} [\Delta_{l} D_{x}^{2k-2}\partial_x \gamma](t)
+\frac{1}{\sigma^2}\partial_{x}\chi_{l}(t,x_{l})=0, 
~~~~~~~~~~~ l \in \{1, \ldots, q-1\}, ~~ t\in[0,T],
\end{align} 
where 
$\chi_{l,0}$ denotes the initial data of $\chi_{l}$ for each $l$, and the boundary data ${v^{\mu}_{x^{\ast}}}$, for all $x^{\ast} \in \{x_0, x_q\}$, are given constant vectors; and the quantity $\Delta_{l} D_{x}^{\mu-1} \partial_x \gamma$ is defined in \eqref{def:Delta_l}.

\bigskip

From the gradient flow equation \eqref{eq:E_s-flow-fitting} and the boundary conditions \eqref{eq:BC-higher-order-1-fitting}$\sim $\eqref{eq:BC-higher-order-5-fitting}, 
we derive the energy identity for \eqref{eq:E_s-energy-f}: 
\begin{equation} 
\label{eq:energy_ID_fitting}
\frac{d}{dt} \mathcal{E}^{\lambda, \sigma}_{k}[( \gamma, \chi_{1}, \ldots, \chi_{q-1})]= 
- \sum^{q}_{l=1} \int_{I_{l}} |\partial_t \gamma_l|^2 \,dx 
- \frac{1}{\sigma^4} \sum^{q-1}_{l=1} \int_{I_{l}} |\partial_t \chi_l|^2 \,dx. 
\end{equation}  
Hence, we may derive from the energy identity \eqref{eq:energy_ID_fitting} that,  $\forall\, t\ge 0$, 
\begin{equation}
\label{ineq:f_x-higher-order-f}
\begin{cases}
\sum\limits^{q}_{l=1}\|  \gamma_{l,x} \|_{L^2(I_{l})}^2 (t) 
\le \frac{2}{\lambda} \mathcal{E}^{\lambda, \sigma}_{k}
[( \gamma(0,\cdot), (\chi_{1}, \ldots, \chi_{q-1})(0,\cdot))], 
\\ 
\sum\limits^{q-1}_{l=1}\|  \chi_{l,x} \|_{L^2(I_{l})}^2 (t) 
\le 2 \sigma^{4} \mathcal{E}^{\lambda, \sigma}_{k}[(\gamma(0,\cdot),(\chi_{1}, \ldots, \chi_{q-1})(0,\cdot))]. 
\end{cases}
\end{equation}

\begin{rem}
\label{rem:new_bdd_L^2-d^{1}gamma}
The anonymous referee kindly suggested that Theorems \ref{thm:Main_Thm-fitting} and \ref{thm:Main_Thm-fitting-appl} may be extended to the case $\lambda \ge 0$ by replacing the uniform estimate of $\sum\limits_{l=1}^{q}\|\gamma_{l,x}\|_{L^2(I_{l})}^2 (t)$ obtained in \eqref{ineq:f_x-higher-order-f} with an alternative uniform bound. 
Such a bound plays a crucial role in deriving 
the G\"{o}nwall inequalities that ensure the long-time existence of solutions to the gradient flow. 
We outline the justification for this extension below and take this opportunity to express our sincere gratitude for the referee’s valuable suggestion.

Note that since $\gamma(t,\cdot):[x_0,x_q]=[0,q]\to M$ belongs to the class $C^{2k-2}([x_0,x_q])$ for every $t$, it follows that $\gamma \in C^{k}([x_0,x_q])$, as $k \ge 2$. 
By following a similar trick in the proof of \cite[Theorem~1.7]{CGP25}, 
we have 
\begin{align*}
&\left| 
\langle D^{k-2}_{x}\gamma_{x}, D^{k-2}_{x}\gamma_{x} \rangle  (\cdot,y)
-\langle D^{k-2}_{x}\gamma_{x}, D^{k-2}_{x}\gamma_{x} \rangle  (\cdot,x_0) 
\right| 
= 2 \left| \int_{x_0}^{y} \langle D^{k-1}_{x}\gamma_{x}, D^{k-2}_{x}\gamma_{x}\rangle \, dx 
\right| 
\\ 
& \le 2 
\left(\int_{x_0}^{y} |D^{k-1}_{x}\gamma_{x}|^2 \, dx \right)^{1/2} 
\cdot \left(\int_{x_0}^{y} |D^{k-2}_{x}\gamma_{x}|^2 \, dx \right)^{1/2} 
\\ 
&\le 2 \left(2 \cdot \mathcal{E}^{\lambda,\sigma}_{k}[(\gamma,\chi_1,\cdots,\chi_{q-1})] \right)^{1/2} 
\cdot \|D^{k-2}_{x}\gamma_{x}\|_{L^\infty(I)} \cdot \sqrt{q} 
\\ 
&\le 2 \left(2 \cdot \mathcal{E}^{\lambda,\sigma}_{k}(t=0) \right)^{1/2} 
\cdot \|D^{k-2}_{x}\gamma_{x}\|_{L^\infty(I)} \cdot \sqrt{q}
\, \text{,}
\end{align*} 
where $|\cdot|^2:=\langle\cdot,\cdot\rangle$ and 
$\mathcal{E}_{k}^{\lambda, \sigma}(t=0):=\mathcal{E}_{k}^{\lambda, \sigma} \big[(\gamma(0,\cdot), (\chi_1, \ldots, \chi_{q-1} ) (0,\cdot)) \big]$.   
By taking the supremum over $y\in I$ in the inequality above, we obtain 
\begin{align} 
\label{eq:sup_of_velocity^{k-1}}
\|D^{k-2}_{x}\gamma_{x}(t,\cdot)\|^2_{L^\infty(I)}  
\le \left(8 q \cdot \mathcal{E}^{\lambda,\sigma}_{k}(t=0) \right)^{1/2} 
\cdot 
\|D^{k-2}_{x}\gamma_{x}(t,\cdot)\|_{L^\infty(I)}  
+ \left|v_{x^\ast=0}^{k-1}\right|^2 
\, \text{,} 
\end{align} 
where $v_{x^\ast=0}^{k-1}=D^{k-2}_{x}\gamma_{x}$ is the boundary datum at $x^\ast=0$ and $t\in[0,\infty)$. 
From a straightforward algebraic manipulation of \eqref{eq:sup_of_velocity^{k-1}}, we obtain 
\begin{equation*}
\label{eq:L^infty-d^{k-1}gamma}
\|D^{k-2}_{x}\gamma_{x}(t,\cdot)\|_{L^\infty(I)} 
\le C(q, \mathcal{E}^{\lambda,\sigma}_{k}(t=0), v_{x^\ast=0}^{k-1})
\, , ~~~  
\forall\, t\in[0,\infty)
\text{,}
\end{equation*}
which in turn implies 
\begin{equation*}
\label{eq:L^2-d^{k-1}gamma}
\|D^{k-2}_{x}\gamma_{x}(t,\cdot)\|_{L^2(I)} 
\le C(q, \mathcal{E}^{\lambda,\sigma}_{k}(t=0), v_{x^\ast=0}^{k-1}) 
\, , ~~~  
\forall\, t\in[0,\infty) 
\text{.}
\end{equation*}
Proceeding inductively, we then derive 
\begin{equation}
\label{eq:L^2-d^{1}gamma}
\|\gamma_{x}(t,\cdot)\|_{L^2(I)} 
\le C(q, \mathcal{E}^{\lambda,\sigma}_{k}(t=0), v_{x^\ast=0}^{k-1}, \cdots, v_{x^\ast=0}^{1}) 
\, , ~~~  
\forall\, t\in[0,\infty)  
\text{.}
\end{equation} 
\end{rem}

\begin{rem}
\label{rem:BCs}
The boundary conditions appearing in the parabolic systems \eqref{eq:E_s-flow}$\sim$\eqref{eq:BC-higher-order-3} and 
\eqref{eq:E_s-flow-fitting}$\sim$\eqref{eq:BC-higher-order-5-fitting} are naturally induced by their respective sets of admissible maps, $\Phi_{\mathcal{P}}$ and $\Theta_{\mathcal{P}}$.

\begin{itemize}
\item
The higher-order concurrency boundary conditions 
in \eqref{eq:BC-higher-order-3} are imposed to ensure the $C^{2k-2}([x_0, x_q])$-regularity of the map $\gamma$. 

\item 
The concurrency boundary conditions in \eqref{eq:BC-higher-order-2-fitting} at the junction points  
are imposed to ensure the connectedness of the network $(\gamma, \chi_1, \ldots, \chi_{q-1})$ at all times. 

\item
The higher-order concurrency boundary conditions in \eqref{eq:BC-higher-order-4-fitting} are imposed to ensure the $C^{2k-2}([x_0, x_q])$-regularity of $\gamma$. 

\item
The balancing conditions at the junction points, as encoded in the boundary conditions \eqref{eq:BC-higher-order-1-fitting}$\sim$\eqref{eq:BC-higher-order-5-fitting}, arise from the first variation formula \eqref{eq:E-L-2} for the network. These conditions are designed to cancel the boundary terms that appear in the variation, thereby ensuring a well-posed variational structure.
\end{itemize}
\end{rem}


\section{Gradient Flow for Networks on Riemannian Manifolds 
}
\label{Sec:Interpolation+Penalty}

In this section, we establish the existence of global classical solutions to the problem of spline interpolation versus least-squares fitting on Riemannian manifolds. Additionally, we demonstrate that the asymptotic limits, as time approaches infinity, correspond to solutions of the variational spline interpolation versus least-squares fitting on Riemannian manifolds. 
In other words, we give the proof of Theorem \ref{thm:Main_Thm-fitting} in this section.


\subsection{The Nonlinear Parabolic Systems for Networks} 
\label{Sec:NPSN}

As outlined in \S \ref{Subsec:2.2}, the variational spline interpolation versus least-squares fitting problem on Riemannian manifolds is initially formulated as an Intrinsic Geometrical Parabolic System for Networks. To enable the application of analytical tools from the theory of parabolic systems, we reformulate this intrinsic system into an Extrinsic Geometrical Parabolic System, and subsequently into a fully developed nonlinear parabolic system for networks, as detailed below. The short-time existence of solutions to the nonlinear parabolic system for networks crucially depends on this extrinsic formulation, whereas the long-time existence analysis is conducted within the intrinsic framework. To adapt the nonlinear evolution equations to the structure required by Solonnikov's theory, it is necessary to reparametrize certain curve segments in the network to reverse their orientation. To carry out this transformation, we introduce additional notation that allows for a systematic reorganization of the boundary conditions. This facilitates the recasting of the system into a parabolic framework, ensuring that the evolution equations satisfy the compatibility conditions necessary for applying Solonnikov's theory. 

\bigskip

{\bf The Extrinsic Geometrical Parabolic System for Networks.} 
To apply Solonnikov's theory to linear parabolic systems in deriving the short-time existence results for our nonlinear problem, we reformulate the intrinsic system \eqref{eq:E_s-flow-fitting}$\sim$\eqref{eq:BC-higher-order-5-fitting} in an extrinsic form by using \eqref{eq:co-deri} and \eqref{eq:co-deri-1} as below. 

\begin{align}
\label{eq:E_s-flow-extrinsic-chi}
&\partial_t \gamma_{l}= (-1)^{k+1} \partial_x^{2k}\gamma_{l}
+ F_{l} (\partial_x^{2k-1}\gamma_{l}, \ldots, \partial_x\gamma_{l}, \gamma_{l})~~~~~~\text{ in } (0,T) \times I_{l}, ~ l \in \{1, \ldots, q\}, 
\\ 
\label{eq:E_s-flow-extrinsic-chi-1} 
&\partial_t \chi_{l} 
= \sigma^2
\partial_x^{2}\chi_{l} 
+ \sigma^2 W_1(\partial_x\chi_{l}, \chi_{l}) 
~~~~~~~~~~~~~~~~~~~~~ \text{ in } (0,T) \times I_{l}, ~ l \in \{1, \ldots, q-1\}, 
\\
\label{eq:E_s-flow-extrinsic-IC-gamma} 
&\gamma_{l}(0,x)=\gamma_{l,0}(x), ~~~~~~~~~~~~~~~~~~~~~~~~~~~~~~~~~~~~~~~~~~~~~~~~~~~~~~~x \in I_{l}, ~ l \in \{1, \ldots, q\}, 
\\ 
\label{eq:E_s-flow-extrinsic-IC-chi} 
&\chi_{l}(0,x)=\chi_{l, 0}(x), ~~~~~~~~~~~~~~~~~~~~~~~~~~~~~~~~~~~~~~~~~~~~~~~~~~ x \in I_{l}, ~ l \in \{1, \ldots, q-1\}, 
\\ 
\label{eq:E_s-flow-extrinsic-chi-BC-1} 
&\gamma(t,x_{l})=p_{l}, ~~~~~~~~~~~~~~~~~~~~~~~~~~~~~~~~~~~~~~~~~~~~~~~~~~~~~~~~~~~~~~ l \in \{0,q\}, ~~ t\in[0,T],  
\\
\label{eq:E_s-flow-extrinsic-chi-BC-2} 
&\chi_{l}(t,x_{l-1})=p_{l}, ~~~~~~~~~~~~~~~~~~~~~~~~~~~~~~~~~~~~~~~~~~~~~~~ l \in \{1,\ldots, q-1\}, ~~t\in[0,T],  
\\
\label{eq:E_s-flow-extrinsic-chi-BC-3} 
&\partial_{x}^{\mu}\gamma_{l}(t,x^{\ast})=b^{\mu}_{l, x^{\ast}}, 
 ~~~~~~~~~ \mu \in \{1,\ldots, k-1\}, ~ (l, x^{\ast}) \in \{(1, x_0), (q, x_q)\}, ~ t\in[0,T], 
\\
\label{eq:E_s-flow-extrinsic-chi-BC-4}
&\gamma_{l}(t,x_{l})=\chi_{l}(t,x_{l}) =\gamma_{l+1}(t,x_{l}), 
~~~~~~~~~~~~~~~~~~~~~~~~~~ l \in \{1, \ldots, q-1\}, ~~ t\in[0,T], 
\\
\label{eq:E_s-flow-extrinsic-chi-BC-5}
&\partial_{x}^{\mu}\gamma_{l+1}(t, x_{l})-\partial_{x}^{\mu}\gamma_{l}(t, x_{l})=0,
~~ \mu \in \{1,\ldots, 2k-2\}, ~ l \in \{1, \ldots, q-1\}, 
~ t\in[0,T],
\\
\label{eq:E_s-flow-extrinsic-chi-BC-6} 
&
(-1)^{k}\left[\partial_{x}^{2k-1}\gamma_{l+1}(t, x_{l})-\partial_{x}^{2k-1}\gamma_{l}(t, x_{l}) \right] + \frac{1}{\sigma^2}\partial_{x}\chi_{l}(t, x_{l})=0, 
\\ \nonumber 
&~~~~~~~~~~~~~~~~~~~~~~~~~~~~~~~~~~~~~~~~~~~~~~~~~~~~~~
~~~~~~~~~~~~~~~~ l \in \{1, \ldots, q-1\}, ~ t\in[0,T], 
\end{align} 
where $\gamma_{l,0}$ denotes the initial data of $\gamma_{l}$, for each $l$, 
$$
W_1(\partial_x\chi_l,\chi_l)=-\Pi_{\chi_l}(\partial_x\chi_l, \partial_x\chi_l), 
$$ 
and 
\begin{align}
\label{eq:F_l} 
&F_l(\partial_x^{2k-1}\gamma_l, \ldots, \gamma_l) 
\\ \nonumber 
& = (-1)^{k+1} W_{2k-1}(\partial_x^{2k-1}\gamma_l, \ldots, \gamma_l) 
+\lambda \cdot \left( \partial_x^2\gamma_l+ W_{1}(\partial_x\gamma_l, \gamma_l) \right)
\\ \nonumber 
&~~~~ + \sum^{k}_{\mu=2} (-1)^{\mu+k+1} 
R\left(
\partial_x^{2k-\mu}\gamma_l
+W_{2k-\mu-1}(\partial_x^{2k-\mu-1}\gamma_l, \ldots, \gamma_l), 
\right.  
\\ \nonumber 
& ~~~~~~~~~~~~~~~~~~~~~~~~~~~~~~~~~~~~~~~~~~~~~~~~~~~~~~~~~~~~~~~~ 
\left. 
\partial_x^{\mu-1}\gamma_l
+W_{\mu-2}(\partial_x^{\mu-2}\gamma_l, \ldots, \gamma_l) 
\right)
\partial_x \gamma_l 
\text{,}
\end{align} 
each $W_i(\partial_x^{i}\gamma_l, \ldots, \gamma_l)$, $\forall\, i\in\mathbb{N}_0$, is defined in \eqref{eq:co-deri} and \eqref{eq:co-deri-1}, 
and each coefficient $b^{\mu}_{l, x^\ast}$ is a constant.

\smallskip

Notice that applying \eqref{eq:co-deri-1} together with the 
higher-order concurrency boundary conditions in \eqref{eq:BC-higher-order-4-fitting} ensures that, for each $\mu \in \{1, \cdots, 2k - 2\}$, the partial derivative $\partial_x^\mu \gamma(t, x)$ remains continuous at every $x \in \{x_1, \cdots, x_{q-1}\}$ for any time $t$. 
Additionally, we find that
$[\Delta_{l} D_{x}^{i-1}\partial_x \gamma](t)
=[\Delta_{l} \partial_{x}^{i} \gamma](t)$ for any $i\in\{1,\cdots,2k-2\}$. 
Thus, we may impose the boundary conditions in \eqref{eq:E_s-flow-extrinsic-chi-BC-5} as a replacement of the higher-order concurrency boundary conditions. 
Through this reformulation, terms involving covariant derivatives and intrinsic properties are expressed explicitly in terms of coordinates and the geometry of the embedding space.

\bigskip

{\bf The Parabolic System for Networks.} 
By utilizing the assumption $x_l=l$ for all $l\in\mathbb{Z}$, we define the parametrization as,
\begin{align}
\label{eq:parametrization-gamma}
g_{l}(t,y) = \gamma_{l}\left(t, (-1)^{l-1}y + x_{2\left[\frac{l}{2}\right]}\right),
\quad g_{l,0}(y) = \gamma_{l,0}\left((-1)^{l-1}y + x_{2\left[\frac{l}{2}\right]}\right), 
\end{align}  
\begin{align}
\label{eq:parametrization-chi}
h_{l}(t,y)=\chi_{l}(t, (-1)^{l-1}y+x_{2\left[\frac{l}{2}\right]}), 
\quad 
h_{l,0}(y)=\chi_{l,0}((-1)^{l-1}y+x_{2\left[\frac{l}{2}\right]}).   
\end{align} 
The orientation of $g_{l}$ and $h_{l}$ are obtained by reversing those of $\gamma_{l}$ and $\chi_{l}$ respectively, when $l$ is an even number. 
This transformation facilitates the application of Solonnikov's theory and the Banach Fixed Point Theorem to obtain the short-time existence. 
By applying \eqref{eq:parametrization-gamma} and \eqref{eq:parametrization-chi}, we can convert the evolution equations for 
$(\gamma, \chi_{1}, \ldots, \chi_{q-1})$ in 
\eqref{eq:E_s-flow-extrinsic-chi}$\sim$\eqref{eq:E_s-flow-extrinsic-chi-BC-6} 
into the nonlinear parabolic system for $(g,h)$: 
\begin{equation}
\label{AP-q-even-fitting}
\begin{cases}
\partial_t g_{l}= (-1)^{k+1}\cdot \partial_y^{2k} g_{l}
+G_l(\partial_y^{2k-1} g_{l}, \ldots, \partial_y g_{l}, g_{l}), 
~~~~~~~~~~~~~~~~ l \in\{1, \ldots, q\}, 
\\ 
\partial_t h_{l}= \sigma^2 \partial_y^{2} h_{l} 
+ \sigma^2 W_1(\partial_y h_l, h_l),  
~~~~~~~~~~~~~~~~~~~~~~~~~~~~~~~~~ 
l \in\{1, \ldots, q-1\}, 
\end{cases}
\end{equation} 
with the initial-boundary conditions: 
\begin{equation} 
\label{BC-q-even-fitting} 
\begin{cases}
\text{(initial conditions)} 
\\
g_{l}(0,y)= g_{l,0}(y),  ~~~~~~~~~~~~~~~~~~~~~~~~~~~~~~~~~~~~~~~~~~~~~~~~~~~~~~~~~~~ 
l\in \{1,\ldots, q\}, 
\\
h_{l}(0,y)= h_{l,0}(y), ~~~~~~~~~~~~~~~~~~~~~~~~~~~~~~~~~~~~~~~~~~~~~~~~~~~~~ 
l\in \{1,\ldots, q-1\}, 
\\ 
\text{(boundary conditions)} 
\\
g_{1}(t,0)= p_{0}, ~~~~ 
g_{q}(t, 2\{\frac{q}{2}\})=p_{q},  
\\ 
h_{l}(t, y^{\ast}_l)= p_{l}, 
~~~~~~~~~~~~~~~~~~~~~~~~~~~~~~~~~~~~~~ 
l \in \{1,\ldots, q-1\}, ~
y^{\ast}_l= 1- 2\{\frac{l}{2}\},  
\\
\partial_{y}^{\mu}g_{1}(t,0)= 
b^{\mu}_{1, x_0},   
~~ 
\partial_{y}^{\mu}g_{q}(t, 2\{\frac{q}{2}\})= 
(-1)^{(q-1)\cdot\mu} b^{\mu}_{q, x_q}, 
~~~~~~~ 
\mu \in \{1,\ldots, k-1\},  
\\ 
g_{l}(t, y^{\ast}_{l})= h_{l}(t, y^{\ast}_l) 
=g_{l+1}(t, y^{\ast}_{l}), 
~~~~~~~~~~~  
l \in \{1,\ldots, q-1\}, ~
y^{\ast}_l= 1- 2\{\frac{l+1}{2}\}, 
\\ 
\partial_{y}^{\mu-1}g_{l}(t, y^{\ast}_{l})
+ (-1)^{\mu-1} \partial_{y}^{\mu} g_{l+1}(t, y^{\ast}_{l})=0, 
\\ 
~~~~~~~~~~~~~~~~~~~~~~~~~~~
\mu \in \{1, \ldots, 2k-2\}, 
~ l \in \{1, \ldots, q-1\}, 
~ y^{\ast}_l= 1- 2\{\frac{l+1}{2}\},  
\\ 
\partial_{y}^{2k-1} g_{l}(t, y^{\ast}_{l})
+\frac{(-1)^{k+1}}{\sigma^2}\partial_{y}h_{l}(t, y^{\ast}_{l}) 
+\partial_{y}^{2k-1} g_{l+1}(t, y^{\ast}_{l})=0, 
\\ 
~~~~~~~~~~~~~~~~~~~~~~~~~~~~~~~~~~~~~~~~~~~~~~~~~~~~~~ 
l \in \{1, \ldots, q-1\}, ~
y^{\ast}_l= 1- 2\{\frac{l+1}{2}\}, 
\end{cases} 
\end{equation} 
where $t\in[0,T]$, $y\in[0,1]$, $\{\frac{b}{a}\}:=\frac{b}{a}-[\frac{b}{a}]\in[0,1)$, and 
\begin{align}\label{eq:G_l-F_l-gamma+chi}
G_l(\partial_y^{2k-1} g_{l}, \ldots, g_{l}):
=F_{l} \left((-1)^{(l-1)(2k-1)}\partial_y^{2k-1} g_{l}, \ldots, (-1)^{(l-1)\mu}\partial_y^{\mu} g_{l}, \ldots, g_{l}\right).   
\end{align}

\begin{rem}
To simplify the notation, the nonlinear parabolic system
\eqref{AP-q-even-fitting}$\sim$\eqref{BC-q-even-fitting}
can be formally written as  
\begin{equation}
\label{eq:Parabolic.Sys_simplified}
\begin{cases}
\partial_t \Gamma=\mathcal{E}(\Gamma), 
\\
\mathcal{B} \cdot \Gamma^{T}|_{[0,T]\times \partial[0,1] }=\Phi, 
\\ 
\mathcal{C} \cdot \Gamma^{T} |_{t=0}=\phi,
\end{cases}
\end{equation} 
where $\Gamma^{T}$ denotes the transpose of the vector $\Gamma$, $\mathcal{E}(\Gamma)$ denotes a spatial (elliptic) differential operator $\mathcal{E}$ acting on the vector $\Gamma=\Gamma(t,y)$, $\mathcal{B}=\mathcal{B}(y^\ast,\partial_y)$ denotes the matrix-valued boundary operator, 
$\Phi=\Phi(y^\ast)$, 
and $\mathcal{C}$ denotes the matrix encoding the initial conditions (an identity matrix) with $\phi=\phi(y)$. 
Each component $\Gamma_{j}$ of the vector $\Gamma=(\Gamma_1,\ldots,\Gamma_{n(2q-1)})$  
is a real-valued function that belongs to one of the following two sets:  
$$
G_{\ast}:=\{g_\nu^\mu: \nu=1, \ldots, q, ~ \mu=1, \ldots, n, \text{ where } g_\nu=(g_\nu^1,\ldots, g_\nu^n) \}
$$ 
or  
$$
H_{\ast}:=\{h_\nu^\mu: \nu=1, \ldots, q-1, ~ \mu=1, \ldots, n, \text{ where } h_\nu=(h_\nu^1,\ldots, h_\nu^n) \}. 
$$ 
Moreover, let us define 
$\Gamma^{g}=(\Gamma_{\sigma(1)},\ldots,\Gamma_{\sigma(nq)})$ and $\Gamma^{h}=(\Gamma_{\sigma(nq+1)},\ldots,\Gamma_{\sigma(n(2q-1))})$, 
where each $\Gamma_{\sigma(\ell)}\in G_\ast$ for $\ell=1,\ldots,nq$, and $\Gamma_{\sigma(\ell)}\in H_\ast$ for $\ell=nq+1,\ldots,n(2q-1)$. 
Here, $\sigma$ denotes a permutation in the symmetric group of degree $n(2q-1)$. 
Then, each component $\Gamma_j$ of the vector $\Gamma$ 
in the parabolic system $\partial_t \Gamma=\mathcal{E}(\Gamma)$ in \eqref{eq:Parabolic.Sys_simplified} can be written as one of the following two equations:  
\begin{align}
\label{eq:d_t(Gamma_j)=E(Gamma)}
\begin{cases}
\partial_t\Gamma_j=E_j(\partial_y^{2k}\Gamma^{g},\ldots,\Gamma^{g}), ~~~ \text{ if } ~ \Gamma_j\in G_\ast 
\text{,}
\\
\text{ or }
\\
\partial_t\Gamma_j=E_j(\partial_y^{2}\Gamma^{h},\ldots,\Gamma^{h}), 
~~~~ \text{ if } ~ \Gamma_j\in H_\ast 
\text{.}
\end{cases}
\end{align}   
\end{rem}

In the remainder of this subsection, we define the compatibility conditions for the nonlinear parabolic system \eqref{AP-q-even-fitting}$\sim$\eqref{BC-q-even-fitting}. This definition requires an explicit formulation of the boundary operator, denoted as $\mathcal{B}$, which is derived from the boundary conditions specified in \eqref{BC-q-even-fitting}. The explicit structure of $\mathcal{B}$ will also play a crucial role in verifying the complementary conditions for the associated coupled linear parabolic systems, introduced in \eqref{eq:higher-order-linear-g_l}$\sim$\eqref{eq:second-order-linear-h_l} in \S\ref{Sec:linear}, which is essential for applying Solonnikov's theory. 
Due to orientation reversals on certain parts of the network, the matrix representation of $\mathcal{B}$ depends on whether the number of spline components, $q$ is even or odd. To establish the required sufficient conditions, we undertake a detailed analysis of the matrix associated with $\mathcal{B}$. In what follows, we treat the cases of even and odd $q$ separately.


\vspace{2cm}  

\begin{figure}[h]
\setlength{\unitlength}{0.56 mm}
\begin{center}
\begin{picture}(120, 75)

{\color{red}
\put(0,25){\circle*{2}} 
\put(35,65){\circle*{2}} 
\put(80,25){\circle*{2}} 
\put(35,75){\circle*{2}} 
\put(80,15){\circle*{2}} 
\put(120,60){\circle*{2}} 
}

\put(-9,17){$p_0=g_1(t,0)$}
\put(27,55){$g_1(t,1)$}
\put(70,31){$g_2(t,0)$}
\put(100,65){$p_3=g_3(t,1)$}
\put(25,79){$p_1=h_{1}(t,0)$}
\put(71,09){$p_2=h_2(t,1)$}
\put(15, 50){$g_1$}
\put(47, 41){$g_2$}
\put(102, 30){$g_3$}
\put(40,66){$h_1$}
\put(84,21){$h_2$}
\put(85,-05){$M$}

\qbezier(0,25)(2,64)(35,65)
\qbezier(35,65)(35,67)(35,75)
\qbezier(35,65)(54,64)(55,51)
\qbezier(55,51)(60,26)(80,25)
\qbezier(80,25)(80,22)(80,15)
\qbezier(80,25)(95,27)(100,40)
\qbezier(100,40)(105,57)(120,60)
\qbezier(-15,0)(-10,67)(10,85)
\qbezier(10,85)(75,120)(140,80)
\qbezier(140,80)(122,47)(120,-8)
\qbezier(120,-8)(75,20)(-15,0) 

\qbezier(-3,26)(0,25)(3,24)
\qbezier(32,75)(35,75)(38,75)
\qbezier(77,15)(80,15)(83,15)
\qbezier(119,57)(120,60)(121,63)

\qbezier(-2.4,25.8)(-2.2,24.4)(-2.0,23)
\qbezier(-0.0,24.8)(0.2,23.4)(0.4,22)
\qbezier(2.4,23.8)(2.6,22.4)(2.8,21)
\qbezier(32.4,76)(32.6,77)(33,78)
\qbezier(34.4,76)(34.6,77)(35,78)
\qbezier(36.6,76)(36.8,77)(37.2,78)
\qbezier(77.8,14)(77.6,13)(77.4,12)
\qbezier(79.8,14)(79.6,13)(79.4,12)
\qbezier(81.8,14)(81.6,13)(81.4,12)
\qbezier(120.5,57.7)(121,57.4)(121.5,57.1)
\qbezier(120.7,59.3)(121.2,59.0)(121.7,58.7)
\qbezier(121.5,61.5)(122,61.2)(122.5,60.9)

{\color{blue}
\qbezier(13,61)(16,61)(19,61)
\qbezier(17,55)(18,58)(19,61)
\qbezier(60,41)(58,44)(56,47) 
\qbezier(56,40)(56,43)(56,47) 
\qbezier(100,34)(100,37)(100,40)
\qbezier(96,38)(98,39)(100,40)
\qbezier(35,69)(34,70)(33,71)
\qbezier(35,69)(36,70)(37,71)
\qbezier(80,19)(79,20)(78,21)
\qbezier(80,19)(81,20)(82,21)
}

\end{picture}
\vspace{0.3cm}
\caption{The orientation of $g$ and $h$ as $q=3$.} 
\end{center}
\end{figure}


\bigskip 
{\bf $\bullet$ When $q\ge 2$ is even.} 
\bigskip

Let  
\begin{align}
\label{eq:Gamma_(g,h)}
\Gamma=(\Gamma_{1}, \ldots, \Gamma_{n(2q-1)}):=(g_{1}, h_{1}, g_{2}, h_{2}, \ldots, g_{q}) \in\mathbb{R}^{n (2q-1)} 
\text{,}
\end{align}
where 
\begin{align}
\label{eq:(g,h)}
g_{\nu}=(g^{1}_{\nu}, \ldots, g^{n}_{\nu}), ~~ \nu=1, \ldots, q ,
~~~~ 
\text{ and }  
~~~~ 
h_{\nu}=(h^{1}_{\nu}, \ldots, h^{n}_{\nu}), ~~ 
\nu=1, \ldots, q-1.
\end{align} 
The boundary conditions in \eqref{BC-q-even-fitting} can be formally written as 
\begin{align*} 
\sum\limits_{j=1}^{n(2q-1)}\mathcal{B}_{ij}\left( y^\ast, \partial_{y} \right) \Gamma_{j}(t, y^{\ast})=\Phi_{ i }( y^{\ast}), 
~~  i \in \{1, \ldots, kn(2q-1)\}, 
~~ y^{\ast} \in \{0,1\}, 
\end{align*} 
or simply as 
\begin{align*} 
\mathcal{B}(y^{\ast}, \partial_{y}) \Gamma^{T}(t, y^{\ast})=\Phi(y^{\ast}) , ~~ y^{\ast} \in \{0,1\}, 
\end{align*} 
where the boundary operators $\mathcal{B}$, represented as diagonal block matrices, are denoted by $\mathcal{B}(y^{\ast}, \partial_{y})$ to highlight their dependence solely on 
$y^{\ast}$ and the differential operator $\partial_y$:  
\begin{equation*}
\label{matrix-B-bdry-q-even-fitting}
\begin{cases}
\mathcal{B}(y^\ast=0, \partial_{y})=
\text{diag}(B_0, \overbrace{Id_{n\times n}, B_2, Id_{n\times n}, B_2, \ldots, Id_{n\times n}}^{\frac{q}{2} \text{ many of }  Id_{n\times n} 
\text{ and } \frac{q}{2}-1 \text{ many of } B_{2}}, B_0), 
\\
\mathcal{B}(y^\ast=1, \partial_{y})
=\text{diag}(\overbrace{B_2, Id_{n\times n}, B_2, Id_{n\times n}, \ldots, Id_{n\times n}, B_2}^{\frac{q}{2} \text{ many of } B_{2}  \text{ and } \frac{q}{2}-1 \text{ many of } Id_{n\times n}}), 
\end{cases}
\end{equation*} 
with 
\begin{align}
\label{matrix-B_0}
B_{0}
=& B_0(\partial_{y}) 
=\begin{pmatrix}
Id_{n \times n}\\ 
\partial_{y} \cdot Id_{n \times n} \\  
\vdots  \\
\partial_{y} ^{k-1} \cdot Id_{n \times n}\\ 
\end{pmatrix}\text{,}
\end{align} 
\begin{align}  
\label{matrix-B_2^y-ast-fitting}
B_2=B_2(\partial_{y})
=\begin{pmatrix}   
Id_{n \times n}& -Id_{n \times n} & 0 \\ 
0 &Id_{n \times n}&-Id_{n \times n} \\ 
\partial_{y} \cdot Id_{n \times n}& 0 & \partial_{y}\cdot Id_{n \times n} \\ 
\partial_{y}^{2}\cdot Id_{n \times n}& 0 & -\partial_{y}^{2} \cdot Id_{n \times n} \\ 
\vdots  &\vdots  &\vdots \\ 
 \partial_{y}^{2k-2} \cdot Id_{n \times n}& 0 & - \partial_{y}^{2k-2} \cdot Id_{n \times n} \\ 
\partial_{y}^{2k-1}\cdot Id_{n \times n}& 
\frac{(-1)^{k+1}}{\sigma^2} \partial_{y} \cdot Id_{n \times n} &\partial_{y}^{2k-1} \cdot Id_{n \times n} \\ 
\end{pmatrix}; 
\end{align} 
and $\Phi=(\Phi_1, \ldots, \Phi_{nk(2q-1)})$ is a constant vector, given by
\begin{equation}
\label{definition-Phi-1-fitting}
\begin{cases}
\big( \Phi_1(0), \ldots, \Phi_{nk}(0) \big) 
=\big(p_0, b^{1}_{1,x_0}, \ldots, b^{k-1}_{1,x_0} \big) \in\mathbb{R}^{nk},
\\
\big( \Phi_{(2q-1)nk-nk+1}(0), \ldots, \Phi_{(2q-1)nk}(0) \big) 
\\ 
~~~~~~~~~~~~~~~~~~~~~~~~~~~ 
=\bigg(p_q, (-1)^{(q-1) \cdot 1} b^{1}_{q,x_q}, \ldots, (-1)^{(q-1) \cdot (k-1)} b^{k-1}_{q,x_q} \bigg) \in\mathbb{R}^{nk},
\\
\big( \Phi_{nk+(2nk+n)\nu+1}(0), \ldots, \Phi_{k+(2nk+n)\nu+n}(0) \big)=p_{\nu+1}\in\mathbb{R}^{n}, ~~~~~~~ \nu \leq q-1, 
\\
\big( \Phi_{2\nu nk+n(\nu-1)+1}(1), \ldots, \Phi_{2\nu nk+n(\nu-1)+n}(1) \big)=p_{\nu}\in\mathbb{R}^{n}, ~~~~~~~~~ \nu \leq q-1, 
\\
\Phi_i(y^{\ast})=0, ~~~~~~~~~~~~~~~~~~~~~~~~~~~~~~~~~~~~~~~~~~~~~~~~~~~~~~~~ \text{ for the rest of } i, ~ y^{\ast}.
\end{cases}
\end{equation} 
The notation $\Phi_i(y^{\ast})$ serves as a simplified form of $\Phi_i(t, y^{\ast})$, since each $\Phi_i(t, y^{\ast})$ is independent of $t$. 
Similarly, $B_2(y^\ast,\partial_y)=B_2(\partial_y)$ for any $y^\ast\in\{0,1\}$, since each 
$B_2(y^\ast,\partial_y)$ is independent of $y^\ast$.    
It's notable that the dimensions of the matrices  
$B_0$, $B_2$, and $\mathcal{B}$, introduced above, are $kn \times n$, $2nk \times 3n$, and $2nkq \times 2nq$, respectively.


\vspace{0.1  cm}

\begin{figure}[h]
\centering
\begin{tikzpicture}[scale=0.8]
  \matrix (magic) [matrix of nodes,left delimiter=(,right delimiter=)]
  {
\node{$B_0$};\draw(-0.5,-0.25) rectangle(0.5,0.5); &  & 0 \\
 & \node{$Id_{n\times n}$};\draw(-0.5,-0.25) rectangle(0.5,0.5); & \\
    0& & \node{$B_0$};\draw(-0.5,-0.25) rectangle(0.5,0.5); \\
  };
\end{tikzpicture}
\begin{tikzpicture}[scale=0.8]
  \matrix (magic) [matrix of nodes,left delimiter=(,right delimiter=)]
  {
\node{$B_0$};\draw(-0.5,-0.25) rectangle(0.5,0.5); &  & &  & 0 \\
&\node{$Id_{n\times n}$};\draw(-0.5,-0.25) rectangle(0.5,0.5); & & &  \\
  &  &\node{$B_2$};\draw(-1.5,-0.25) rectangle(1.5,0.5);
\draw[dashed] (0.5,0.5)--(0.5,0.25);  \draw[dashed] (0.5, 0.15)--(0.5,0.03);     \draw[dashed] (0.5,-0.25)--(0.5,-0.2);  
\draw[dashed] (-0.5,0.5)--(-0.5,0.25);  \draw[dashed] (-0.5, 0.15)--(-0.5,0.03);     \draw[dashed] (-0.5,-0.25)--(-0.5,-0.2);  
& & \\
    & & & \node{$Id_{n\times n}$};\draw(-0.5,-0.25) rectangle(0.5,0.5); & \\
    0& & && \node{$B_0$};\draw(-0.5,-0.25) rectangle(0.5,0.5);  \\
  };
\end{tikzpicture}
\caption{The matrix $\mathcal{B}(y^\ast=0, \partial_{y})$} when $q=2$ or $q=4$. 
\end{figure}


\vspace{1.5 cm}

\begin{figure}[h]
\centering
\begin{tikzpicture}[scale=0.8]
  \matrix (magic) [matrix of nodes,left delimiter=(,right delimiter=)]
  {
\node{$B_2$};\draw(-1.5,-0.25) rectangle(1.5,0.5);
\draw[dashed] (0.5,0.5)--(0.5,0.25);  \draw[dashed] (0.5, 0.15)--(0.5,0.03);     \draw[dashed] (0.5,-0.25)--(0.5,-0.2);  
\draw[dashed] (-0.5,0.5)--(-0.5,0.25);  \draw[dashed] (-0.5, 0.15)--(-0.5,0.03);     \draw[dashed] (-0.5,-0.25)--(-0.5,-0.2);    \\
  };
\end{tikzpicture}
\begin{tikzpicture}[scale=0.8]
  \matrix (magic) [matrix of nodes,left delimiter=(,right delimiter=)]
  {
\node{$B_2$};\draw(-1.5,-0.25) rectangle(1.5,0.5);
\draw[dashed] (0.5,0.5)--(0.5,0.25);  \draw[dashed] (0.5, 0.15)--(0.5,0.03);     \draw[dashed] (0.5,-0.25)--(0.5,-0.2);  
\draw[dashed] (-0.5,0.5)--(-0.5,0.25);  \draw[dashed] (-0.5, 0.15)--(-0.5,0.03);     \draw[dashed] (-0.5,-0.25)--(-0.5,-0.2);  
& & 0\\
    & \node{$Id_{n\times n}$};\draw(-0.5,-0.25) rectangle(0.5,0.5); & \\
    0& &\node{$B_2$};\draw(-1.5,-0.25) rectangle(1.5,0.5);
\draw[dashed] (0.5,0.5)--(0.5,0.25);  \draw[dashed] (0.5, 0.15)--(0.5,0.03);     \draw[dashed] (0.5,-0.25)--(0.5,-0.2);  
\draw[dashed] (-0.5,0.5)--(-0.5,0.25);  \draw[dashed] (-0.5, 0.15)--(-0.5,0.03);     \draw[dashed] (-0.5,-0.25)--(-0.5,-0.2);   \\
  };
\end{tikzpicture}
\caption{The matrix $\mathcal{B}(y^\ast=1, \partial_{y})$} when $q=2$ or $q=4$. 
\end{figure}


\bigskip 

{\bf $\bullet$ When $q\ge3$ is odd.} 
\bigskip 

As $q\ge3$ is odd, the boundary conditions in 
\eqref{BC-q-even-fitting} can be similarly written as 
\begin{align*} 
\mathcal{B}(y^{\ast}, \partial_{y})  \Gamma^{T}(t, y^{\ast})=\Phi(y^{\ast}), ~~ y^{\ast} \in \{0,1\}, 
\end{align*} 
where the diagonal block matrices $\mathcal{B}(y^{\ast}, \partial_{y})$ 
are given by 
\begin{align*} 
\begin{cases} 
\mathcal{B}(y^\ast=0, \partial_{y})=&\text{diag}(B_0, \overbrace{B_2, Id_{n \times n}, B_2, Id_{n \times n}, \ldots, B_2, Id_{n \times n}}
^{\frac{q-1}{2}  \text{ many of }  B_2 \text{ and } \frac{q-1}{2}  \text{ many of }  Id_{n \times n} }), \\
\mathcal{B}(y^\ast=1, \partial_{y})=&\text{diag}(\overbrace{Id_{n \times n}, B_2 , Id_{n \times n}, 
B_2,  \ldots, Id_{n \times n}, B_2}^{\frac{q-1}{2}  \text{ many of }  Id_{n \times n} \text{ and } \frac{q-1}{2}  \text{ many of }  B_2 }, B_0),
\end{cases}
\end{align*} 
with $B_0$ and $B_2$ given in \eqref{matrix-B_0} and \eqref{matrix-B_2^y-ast-fitting} respectively,
and $\Phi=(\Phi_1,\ldots,\Phi_{kn(2q-1)})$ 
are constant vectors given by 
\begin{equation}
\label{definition-Phi-2-fitting}
\begin{cases}
\big( \Phi_1(0), \ldots, \Phi_{nk}(0) \big)=\big(p_0, 
b^{1}_{1,x_0}, \ldots, b^{k-1}_{1,x_0} \big)\in \mathbb{R}^{nk}, 
\\
\big( \Phi_{kn(2q-1)-nk+1}(1), \ldots, \Phi_{kn(2q-1)}(1) \big)
\\ 
~~~~~~~~~~~~~~~~~~~~~~~~~~~~~~~ 
=\bigg(p_q, (-1)^{q-1}  b^{1}_{q,x_q}, \ldots, (-1)^{(q-1)(k-1)} b^{k-1}_{q,x_q} \bigg) \in \mathbb{R}^{nk}, \\
\big( \Phi_{n+2nk\nu+n(\nu-1)+1}(0), \ldots, \Phi_{n+2nk\nu+n(\nu-1)+n}(0) \big)=p_{\nu}\in \mathbb{R}^{n}, ~~ \nu \leq q-1, 
\\
\big( \Phi_{(2nk+n)\nu+1}(1), \ldots, \Phi_{(2nk+n)\nu+n}(1) \big)=p_{\nu+1}\in \mathbb{R}^{n}, ~~~~~~~~~~~~~~ \nu \leq q-1, 
\\
\Phi_i(y^{\ast})=0, ~~~~~~~~~~~~~~~~~~~~~~~~~~~~~~~~~~~~~~~~~~~~~~~~~~~~~~~~ \text{ for the rest of } i, ~ y^{\ast}. 
\end{cases}
\end{equation}


\vspace{1.5 cm}

\begin{figure}[h]
\centering
\begin{tikzpicture}[scale=0.4]
  \matrix (magic) [matrix of nodes,left delimiter=(,right delimiter=)]
  {
\node{$B_0$};\draw(-0.5,-0.25) rectangle(0.5,0.5); &  & &  & 0 \\
&  \node{$B_2$};\draw(-1.5,-0.25) rectangle(1.5,0.5);
\draw[dashed] (0.5,0.5)--(0.5,0.25);  \draw[dashed] (0.5, 0.15)--(0.5,0.03);     \draw[dashed] (0.5,-0.25)--(0.5,-0.2);  
\draw[dashed] (-0.5,0.5)--(-0.5,0.25);  \draw[dashed] (-0.5, 0.15)--(-0.5,0.03);     \draw[dashed] (-0.5,-0.25)--(-0.5,-0.2);  & &  &  \\
    &  & \node{$Id_{n \times n}$};\draw(-0.5,-0.25) rectangle(0.5,0.5);   &  &     \\
&  &  & \node{$B_2$};\draw(-1.5,-0.25) rectangle(1.5,0.5);
\draw[dashed] (0.5,0.5)--(0.5,0.25);  \draw[dashed] (0.5, 0.15)--(0.5,0.03);     \draw[dashed] (0.5,-0.25)--(0.5,-0.2);  
\draw[dashed] (-0.5,0.5)--(-0.5,0.25);  \draw[dashed] (-0.5, 0.15)--(-0.5,0.03);     \draw[dashed] (-0.5,-0.25)--(-0.5,-0.2);  &  \\
    0 &  &   &  & \node{$Id_{n \times n}$};\draw(-0.5,-0.25) rectangle(0.5,0.5);   \\
  };
\end{tikzpicture}
\caption{The matrix $\mathcal{B}(y^\ast=0, t, \partial_{y}, \partial_{t})$ when $q=5$.} 
\end{figure}

\vspace{1.2  cm}

\begin{figure}[h]
\centering
\begin{tikzpicture}[scale=0.4]
  \matrix (magic) [matrix of nodes,left delimiter=(,right delimiter=)]
  {
\node{$Id_{n \times n}$};\draw(-0.5,-0.25) rectangle(0.5,0.5);  & &  &  & 0 \\
   &\node{$B_2$};\draw(-1.5,-0.25) rectangle(1.5,0.5);
\draw[dashed] (0.5,0.5)--(0.5,0.25);  \draw[dashed] (0.5, 0.15)--(0.5,0.03);     \draw[dashed] (0.5,-0.25)--(0.5,-0.2);  
\draw[dashed] (-0.5,0.5)--(-0.5,0.25);  \draw[dashed] (-0.5, 0.15)--(-0.5,0.03);     \draw[dashed] (-0.5,-0.25)--(-0.5,-0.2);  &  &     &   \\
 &  & \node{$Id_{n \times n}$};\draw(-0.5,-0.25) rectangle(0.5,0.5);  &  & \\
      &   &  & \node{$B_2$};\draw(-1.5,-0.25) rectangle(1.5,0.5);
\draw[dashed] (0.5,0.5)--(0.5,0.25);  \draw[dashed] (0.5, 0.15)--(0.5,0.03);     \draw[dashed] (0.5,-0.25)--(0.5,-0.2);  
\draw[dashed] (-0.5,0.5)--(-0.5,0.25);  \draw[dashed] (-0.5, 0.15)--(-0.5,0.03);     \draw[dashed] (-0.5,-0.25)--(-0.5,-0.2);& 
\\
0 &  & &  & \node{$B_0$};\draw(-0.5,-0.25) rectangle(0.5,0.5); 
\\
  };
\end{tikzpicture}
\caption{The matrix $\mathcal{B}(y^\ast=1, t, \partial_{y}, \partial_{t})$ when $q=5$.} 
\end{figure}


\begin{defi}
[The Compatibility Conditions of Order $\ell$
to the Parabolic System for Networks in 
\eqref{AP-q-even-fitting} and \eqref{BC-q-even-fitting}]
\label{def:cc-order-p-fitting} 
\text{ } 

Let $\ell \in \mathbb{N}_0$, $\alpha_j\in(0,1)$, and $\varrho_j=\ell+\alpha_j$, $j\in\{1,2\}$. 
We say that the initial datum $(g_0,h_0)$ with the regularity 
$g_{l,0} \in C^{2k+\varrho_1}([0,1])$ and 
$h_{l,0} \in C^{2+\varrho_2}([0,1])$  
fulfills the compatibility conditions of order $\ell$ to the nonlinear parabolic system 
\eqref{AP-q-even-fitting} and \eqref{BC-q-even-fitting}  
if $(g_0,h_0)$ satisfies 
all of the following differential equations at the boundaries 
$\{t=0\}\times \{0,1\}$, 
\begin{align} 
\label{compatibility-condition-fitting}
\sum\limits_{j=1}^{n(2q-1)}\mathcal{B}_{ij}\left( y^\ast, \partial_{y}\right) \partial_t^p  \Gamma_{j}(t, y^{\ast}) \bigg{| }_{t=0}= 
\begin{cases}
\Phi_i(y^{\ast}), ~ p=0, 
\\
0, ~~~~~~~~ p\ge 1, 
\end{cases}
~ i\in\{1, \ldots, kn(2q-1)\}  
\text{,}
\end{align} 
which are characterized as below: 
\begin{itemize} 
\item[(i)] 
Each component $\Gamma_{j}$ of the vector $\Gamma=(\Gamma_1,\ldots,\Gamma_{n(2q-1)})$   
is a real-valued function that belongs to either the set $G_{\ast}$ or the set $H_{\ast}$. 
For each fixed $i\in\{1, \ldots, kn(2q-1)\}$ 
appearing in \eqref{compatibility-condition-fitting}, the associated integer $p$ must satisfy the condition 
\begin{align}
\label{highest-order-diferential-operator}
\max_{j} \{t_j\cdot p+\sigma_i(j) \} \le \ell
\text{.}
\end{align}
Here, $\sigma_{i}(j):=\beta_{ij}- t_j$, 
where 
\begin{align*}
\begin{cases}
t_j=2k, ~~~ \text{ if } ~~~ \Gamma_j\in G_{\ast}, 
\\
t_j=2, ~~~~ \text{ if } ~~~ \Gamma_j\in H_{\ast},
\end{cases}
\end{align*} 
and $\beta_{ij}$ denotes the highest order of the spatial derivative $\partial_y$ appearing in the $(i,j)$-entry of the boundary operator matrix $\mathcal{B}$. 
\item[(ii)] 
The equations in \eqref{compatibility-condition-fitting} should be understood as differential equations in the spatial variable $y$. They are derived by replacing each time derivative $\partial_t$ with a spatial derivative $\partial_y$, using the relations provided in \eqref{AP-q-even-fitting}, or equivalently, by applying the formal identities in \eqref{eq:d_t(Gamma_j)=E(Gamma)}.
\end{itemize}
\end{defi}

\begin{defi}
[The Geometric Compatibility Conditions of Order $\ell$ to 
the Intrinsic Geometrical Parabolic System for Networks 
in \eqref{eq:E_s-flow-fitting}$\sim$
\eqref{eq:BC-higher-order-5-fitting} ] 
\label{def:geom-cc-order-p-fitting}
Let $x_{l}=l$, $l \in \{0, \ldots, q\}\subset\mathbb{Z}$. 
Suppose $\ell \in \mathbb{N}_0$, $\alpha_j\in(0,1)$, and $\varrho_j=\ell+\alpha_j$, $j\in\{1,2\}$. 
We say that the initial datum 
$(\gamma_0, \chi_{1,0}, \ldots, \chi_{q-1,0})$ with the regularity $\gamma_{l,0} \in C^{2k+\varrho_1}([x_{l-1},x_l])$ and 
$\chi_{l, 0} \in C^{2+\varrho_2}([x_{l-1},x_l])$ 
fulfills the geometric compatibility conditions of order $\ell$ to the parabolic system \eqref{eq:E_s-flow-fitting}$\sim$\eqref{eq:BC-higher-order-5-fitting} if the initial datum $(g_0, h_0)$, derived from converting $(\gamma_0, \chi_{1,0}, \ldots, \chi_{q-1,0})$ by using \eqref{eq:parametrization-gamma} and \eqref{eq:parametrization-chi}, satisfies the compatibility conditions of order $\ell$ in Definition \ref{def:cc-order-p-fitting}. 
\end{defi}

\begin{rem} 
[The Geometric Compatibility Conditions of Order Zero to 
\eqref{eq:E_s-flow-fitting}$\sim$\eqref{eq:BC-higher-order-5-fitting}] 
\label{remark-def:geo.cc.order-0-fitting} 
 
Suppose $x_{l}=l$, $l \in \{0, \ldots, q\}\subset\mathbb{Z}$, $\alpha_j\in(0,1)$, $j\in\{1,2\}$, as given previously,  
and $(\gamma_0, \chi_{1,0}, \ldots, \chi_{q-1,0})$ is the initial datum with the smoothness $\gamma_0\in C^{2k-2}([x_0, x_q])$, $\gamma_{l,0} \in C^{2k+\alpha_1}([x_{l-1}, x_{l}])$, $\chi_{l, 0} \in C^{2+\alpha_2}([x_{l-1}, x_{l}])$. 
\begin{enumerate}
\item[(i)] 
The initial datum $(\gamma_0, \chi_{1,0}, \ldots, \chi_{q-1,0})$ is said to fulfill the compatibility conditions of order zero to the parabolic system 
\eqref{eq:E_s-flow-fitting}$\sim$\eqref{eq:BC-higher-order-5-fitting} if the following holds: 
\begin{equation*} 
\begin{cases} 
\mathcal{L}_x^{2k}(\gamma_{0})(x_{l}) =0, ~~~~~~~~~~~~~~~~~~~~~~~~~~~~~~~~~~~~~~~~~~~~~~~~~~~~~~~~~~~ 
l \in \{0, q\}, 
\\
\sigma^2 D_{x}\partial_{x}\chi_{l,0}(x_{l-1})=0
~~~~~~~~~~~~~~~~~~~~~~~~~~~~~~~~~~~~~~~~~~~  
l \in \{1, \ldots, q-1\}, 
\\
\mathcal{L}_x^{2k}(\gamma_{l,0})(x_{l})
= \sigma^2 D_{x}\partial_{x}\chi_{l,0}(x_{l})
=\mathcal{L}_x^{2k}(\gamma_{l+1,0})(x_{l}), 
~~~~~~~~~ 
l \in \{1, \ldots, q-1\},
\\
\gamma_{0}(x_{l})=p_{l}, ~~~~~~~~~~~~~~~~~~~~~~~~~~~~~~~~~~~~~~~~~~~~~~~~~~~~~~~~~~~~~~~~~ 
l \in \{0,q\}, 
\\
\chi_{l, 0}(x_{l-1})=p_{l}, 
~~~~~~~~~~~~~~~~~~~~~~~~~~~~~~~~~~~~~~~~~~~~~~~~~~~ 
l \in \{1,\ldots, q-1\}, 
\\ 
D_{x}^{\mu-1}\partial_x\gamma_{l,0} (x^{\ast})=v^{\mu}_{x^{\ast}} ~~~~~~~~~~~~~~~~ 
1\leq \mu \leq k-1, ~ (l, x^{\ast}) \in \{(1, x_0), (q, x_q)\}, 
\\
\gamma_{l,0}(x_{l})=\chi_{l, 0}(x_{l})=\gamma_{l+1,0}(x_{l})
~~~~~~~~~~~~~~~~~~~~~~~~~~~~~~~~~~ 
l \in \{1, \ldots, q-1\},
\\
[\Delta_{l} D_{x}^{\mu-1}\partial_x\gamma_{0}]=0, 
~~~~~~~~~~~~~~~~~~~~~~~~~~~ 
1 \leq \mu \leq 2k-2, ~ l \in \{1, \ldots, q-1\},
\\ 
(-1)^{k} [\Delta_{l}  D_{x}^{2k-2}\partial_x\gamma_{0}]
+\frac{1}{\sigma^2}\partial_{x}\chi_{l,0}(x_{l})=0, 
~~~~~~~~~~~~~~~~~~~~~ 
l \in \{1, \ldots, q-1\}, 
\end{cases}
\end{equation*}
where the boundary datum $\{v^{\mu}_{x^{\ast}}\}$ are constant vectors. 

\item[(ii)] 
For the boundary conditions 
\eqref{eq:BC-higher-order-1-fitting}, \eqref{eq:BC-2-order-1-fitting}, and \eqref{eq:BC-higher-order-2-fitting}, 
the highest order of differential operator $\partial_y$ in $\mathcal{B}_{i j}$ in \eqref{compatibility-condition-fitting} is equal to zero, i.e., $\beta_{i j}=0$. 
Thus, in this case, $\sigma_i(j)=-t_j$, and the compatibility conditions of order zero in \eqref{highest-order-diferential-operator} holds when $p=0$ or $p=1$ in \eqref{compatibility-condition-fitting}. 
For the remaining boundary conditions, 
\eqref{eq:BC-higher-order-3-fitting}, \eqref{eq:BC-higher-order-4-fitting}, and \eqref{eq:BC-higher-order-5-fitting}, 
the highest order of differential operator $\partial_y$ in $\mathcal{B}_{i j}$ in \eqref{compatibility-condition-fitting} is greater than or equal to $1$, i.e., $\beta_{i j} \geq 1$. 
From the definition, the order $p$ of the time derivative $\partial_t$ in \eqref{compatibility-condition-fitting} must satisfy 
\begin{align*}
0 \geq \max_{j} \{t_j\cdot p+\sigma_i(j) \} \geq \max_{j} \{t_j\cdot (p-1)+1 \}
\end{align*} 
in order for the compatibility condition of order zero in \eqref{highest-order-diferential-operator} to be satisfied.  
This implies that $p=0$, since $t_j$ is equal to either $2k>0$ or $2$. 
\end{enumerate}

\end{rem}

\subsection{The Induced Linear Problem}
\label{Sec:linear} 

To establish local existence for the nonlinear parabolic system \eqref{eq:E_s-flow-fitting}$\sim$\eqref{eq:BC-higher-order-5-fitting}, we begin by analyzing the corresponding linearized system and verifying the conditions required for applying Solonnikov’s theory of linear parabolic systems (see Theorem \ref{Solonnikov-theorem} in \S\ref{subsec:Solonnikov}). 
These requirements include Assumption \ref{assumption:complementary-condition} and the compatibility conditions specified in Definition \ref{def:compatibility-conditions-general-equa}.
We recall that compatibility conditions of order $\ell$ are required to ensure that solutions to the linear parabolic system attain the corresponding level of regularity up to the initial time on compact domains.

Let 
\begin{align*}
Z^{T}:=
\left(C^{\frac{2k+\alpha_1}{2k}, 2k+\alpha_1} [0,T] \times [0,1]\right)^{q} 
\times \left(C^{\frac{2+\alpha_2}{2}, 2+\alpha_2} [0,T] \times [0,1]\right)^{q-1}   
\end{align*}
be the Banach space associated with the norm 
\begin{align*}
\norm{(g,h)}_{Z^{T} }:=\sum\limits^{q}_{l=1}\norm{g_{l}}_{C^{\frac{2k+\alpha_1}{2k},2k+\alpha_1}([0,T] \times [0,1])}
+\sum\limits^{q-1}_{l=1}\norm{h_{l}}_{C^{\frac{2+\alpha_2}{2},2+\alpha_2}([0,T] \times [0,1])} 
\end{align*} 
(see \cite[Theorems 1.8.2 and 1.8.6]{M98}). 
Let $(g_0,h_0)$ be obtained from modifying the orientation of 
$(\gamma_0, \chi_{1, 0}, \ldots, \chi_{q-1, 0})\in \Theta_{\mathcal{P}}$, as defined in \eqref{eq:parametrization-gamma} and 
\eqref{eq:parametrization-chi}. 
For any $T>0$, 
define 
$X^{T}_{(g_0,h_0)}\subset Z^{T}$ by  
\begin{align*}
X^{T}_{(g_0,h_0)}:=&\big{ \{ } (g,h)\in Z^{T}: 
g_{l}(0, \cdot)=g_{l,0}(\cdot), h_{l}(0, \cdot)=h_{l, 0}(\cdot), \forall\, l\big{ \} }, 
\end{align*} 
and endow $X^{T}_{(g_0,h_0)}$ with the norm  $\|\cdot\|_{X^{T}_{(g_0,h_0)}}=\|\cdot\|_{Z^{T} }$.  
Let the subset $B_K\subset X^{T}_{(g_0,h_0)}$ be  
$$
B_K=\{(g,h) \in X^{T}_{(g_0,h_0)}: \norm{(g,h)}_{X^{T}_{(g_0,h_0)}} \leq K\}, 
$$ 
which is a bounded, closed, and convex subset of the Banach space $Z^{T}$.

Since the orders of parabolicity for $g_l$ and $h_l$ of the nonlinear parabolic system \eqref{AP-q-even-fitting}$\sim$\eqref{BC-q-even-fitting} are different, we decompose the linearized system into two subsystems, allowing Solonnikov's framework (see Theorem \ref{Solonnikov-theorem}) to be applied appropriately, as detailed below:

\begin{equation}
\label{eq:higher-order-linear-g_l}
\begin{cases}
\partial_t g_{l}+ (-1)^{k}\cdot \partial_y^{2k} g_{l}
=G_l(\partial_y^{2k-1} \bar{g}_{l}, \ldots, \partial_y \bar{g}_{l}, \bar{g}_{l}), 
~~~~~~~~~~~~~~~~~~ 
l \in\{1, \ldots, q\}, 
\\ 
g_{l}(0,y)= g_{l,0}(y), 
~~~~~~~~~~~~~~~~~~~~~~~~~~~~~~~~~~~~~~~~~~~~~~~~~~~~~~~~~ 
l\in \{1, \ldots, q\}, 
\\
g_{1}(t,0)= p_{0}, 
~~~ 
g_{q}(t, 2\{\frac{q}{2}\})=p_{q},  
\\ 
\partial_{y}^{\mu}g_{1}(t,0)= 
b^{\mu}_{1, x_0},   
~~ 
\partial_{y}^{\mu}g_{q}(t, 2\{\frac{q}{2}\})= 
(-1)^{(q-1)\cdot\mu} b^{\mu}_{q, x_q}, 
~~~~~~ 
\mu \in \{1, \ldots, k-1\},  
\\ 
g_{l}(t, y^{\ast}_{l})
=g_{l+1}(t, y^{\ast}_{l}), 
~~~~~~~~~~~~~~~~~~~~~~~~
l \in \{1, \ldots, q-1\}, 
~
 y^{\ast}_l= 1- 2\{\frac{l+1}{2}\},  
\\ 
\partial_{y}^{\mu-1}g_{l}(t,  y^{\ast}_{l})
+ (-1)^{\mu-1} \partial_{y}^{\mu} g_{l+1}(t, y^{\ast}_{l})=0, 
\\ 
~~~~~~~~~~~~~~~~~~~~~~~~~~~
\mu \in \{1, \ldots, 2k-2\}, 
~ 
l \in \{1, \ldots, q-1\}, y^{\ast}_{l}= 1- 2\{\frac{l+1}{2}\},  
\\ 
\partial_{y}^{2k-1} g_{l}(t, y^{\ast}_{l})
+\partial_{y}^{2k-1} g_{l+1}(t, y^{\ast}_{l})=\frac{(-1)^{k}}{\sigma^2}\partial_{y} \bar{h}_{l}(t, y^{\ast}_l), 
\\ 
~~~~~~~~~~~~~~~~~~~~~~~~~~~~~~~~~~~~~~~~~~~~~~~~~~~~~ 
l \in \{1, \ldots, q-1\}, ~
 y^{\ast}_{l}= 1- 2\{\frac{l+1}{2}\}, 
\end{cases}
\end{equation} 
and 
\begin{equation}
\label{eq:second-order-linear-h_l}
\begin{cases}
\partial_t h_{l}= \sigma^2 \partial_y^{2} h_{l} 
+ \sigma^2 W_1(\partial_y \bar{h}_l, \bar{h}_l), 
\\ 
h_{l}(0,y)= h_{l,0}(y), 
\\
h_{l}(t, y^{\ast}_l)= p_{l}, 
~~~~~~~~~~~~~~~~~~~~~~~~~~~~~~~~~~~~~~~~~~~~~~~~~~~~~~~~~~~~~ 
y^{\ast}_{l}= 1- 2\{\frac{l}{2}\},   
\\
 h_{l}(t, y^{\ast}_l)=g_{l}(t, y^{\ast}_{l}),
~~~~~~~~~~~~~~~~~~~~~~~~~~~~~~~~~~~~~~~~~~~~~~~~~~~~ 
y^{\ast}_{l}= 1- 2\{\frac{l+1}{2}\}, 
\end{cases}
\end{equation} 
where $l \in \{1, \ldots, q-1\}$, $t\in[0,T]$, $y\in[0,1]$, $\{\frac{b}{a}\}:=\frac{b}{a}-[\frac{b}{a}]\in[0,1)$, $(\bar{g}, \bar{h}) \in X^{T}_{(g_0,h_0)}$. 

\bigskip

These two linear parabolic systems \eqref{eq:higher-order-linear-g_l} and \eqref{eq:second-order-linear-h_l} are coupled through their boundary data. In particular, the boundary datum in \eqref{eq:second-order-linear-h_l} involves $g_l(\cdot,y^\ast_l)$, obtained by restricting the solution of \eqref{eq:higher-order-linear-g_l} to the boundary $[0,T]\times \{y^\ast_l\}$, and is treated as a prescribed function. Despite this coupling, which imparts a nonlinear character to the overall system when considered as a whole, 
we still refer to \eqref{eq:higher-order-linear-g_l}$\sim$\eqref{eq:second-order-linear-h_l} as a coupled linear parabolic system, since the underlying parabolic equations themselves are linear.

To prove the existence of classical solutions to the coupled linear parabolic system, \eqref{eq:higher-order-linear-g_l}$\sim$\eqref{eq:second-order-linear-h_l} in this subsection, we first apply Solonnikov's theory to establish the existence of a solution $g_l$ to the linear parabolic system \eqref{eq:higher-order-linear-g_l} (see Lemma \ref{lem:linear-system-g-l}). With $g_l$ obtained, we then apply Solonnikov's theory again to prove the existence of a solution $h_l$ to the linear parabolic system \eqref{eq:second-order-linear-h_l} (see Lemma \ref{lem:linear-system-h-l}). 
A significant technical obstacle in this analysis in this subsection concerns the validation of complementarity conditions—a crucial requirement for ensuring well-posedness. The fulfillment of these conditions is fundamentally dependent upon the formulation of the prescribed boundary conditions.

Note that, in this approach, we do not obtain full higher regularity for $g_l$ and $h_l$; specifically, we only achieve 
$g_{l}\in C^{\frac{2k+\varrho_1}{2k}, 2k+\varrho_1}([0,T]\times [0,1])$, $\forall\, l$, for the linear equation \eqref{eq:higher-order-linear-g_l}, and $h_{l} \in C^{\frac{2+\varrho_2}{2}, 2+\varrho_2}([0,T]\times [0,1])$, $\forall\, l$, for the linear equation \eqref{eq:second-order-linear-h_l}, where $\varrho_1, \varrho_2\in (0,1)$. The existence of solutions to \eqref{eq:higher-order-linear-g_l} relies on verifying the compatibility conditions of order $\ell=[\varrho_1]=[\varrho_2]$. However, when $\ell$ is large, these sufficient conditions may fail to hold. This difficulty stems from the presence of the boundary term $\partial_y \bar{h}_l$ in the linearized parabolic system \eqref{eq:higher-order-linear-g_l}. 
For instance, when $\ell$ is large, verifying the compatibility conditions of order $\ell$ requires differentiating the boundary conditions in time, which leads to time-derivatives of $\partial_y \bar{h}_l$. Since $\bar{h}_l$ is a prescribed map, these time derivatives cannot be converted into the spatial-derivatives of $\partial_y \bar{h}_l$, as required by the compatibility conditions of order $\ell$ described in Definition \ref{def:cc-order-p-fitting}.

The following two lemmas aim to provide the existence of the coupled linear parabolic systems, \eqref{eq:higher-order-linear-g_l} and \eqref{eq:second-order-linear-h_l}.

\begin{lem}[Solutions to the linear parabolic system \eqref{eq:higher-order-linear-g_l}]
\label{lem:linear-system-g-l} 
Let $\lambda\in [0,\infty)$, $\sigma \in (0,\infty)$, $\alpha_j\in(0,1)$, $j \in \{1, 2\}$ such that $\alpha_1>k\alpha_2$. 
Suppose that $(g_0,h_0)$ is the initial datum for the nonlinear parabolic system, \eqref{AP-q-even-fitting}$\sim$\eqref{BC-q-even-fitting}, where  
$g_0=(g_{1,0},\cdots,g_{q,0})$,   
$h_0=(h_{1,0}, \ldots, h_{q-1,0})$, 
and each $g_{l,0}$ and $h_{l,0}: [0,1] \to \mathbb{R}^n$ satisfies the regularity conditions $g_{l,0}\in C^{2k+\alpha_1}([0,1])$ and $h_{l,0} \in C^{2+\alpha_2}([0,1])$ for all $l$. 
Moreover, assume that $(g_0, h_0)$ satisfies the  compatibility conditions of order zero, as defined in Definition \ref{def:cc-order-p-fitting}, for the nonlinear parabolic system \eqref{AP-q-even-fitting}$\sim$\eqref{BC-q-even-fitting}. 
Then, for any $T>0$ and $(\bar{g}, \bar{h}) \in X^{T}_{(g_0,h_0)}$, there exists a unique solution $g_l$ 
to the linear parabolic system \eqref{eq:higher-order-linear-g_l} such that 
$g_{l}\in C^{\frac{2k+\alpha_1}{2k}, 2k+\alpha_1}([0,T]\times [0,1])$,  
$\forall\, l \in \{1, \ldots, q\}$, and  
\begin{align}
\label{estimate-for-g_l}
\nonumber
\sum\limits^{q}_{l=1} & \norm{g_{l}}_{C^{\frac{2k+\alpha_1}{2k},2k+\alpha_1}([0,T] \times [0,1])} \leq \bar{C}_0\cdot \bigg{(} \sum^{q}_{l=1}
\norm{G_l(\partial_y^{2k-1} \bar{g}_{l},
\ldots,\bar{g}_{l})}_{C^{\frac{\alpha_1}{2k},\alpha_1}([0,T] \times [0,1])} \\ \nonumber
&~~~~~~~~~~~~~~~~~~~~~~~~~~~~~+\frac{1}{\sigma^2} \sum_{\substack{l=1, \cdots, q-1 \\ y^{\ast}_l =1-2 \big\{ \frac{l+1}{2} \big\} }} \|\partial_y \bar{h}_l (\cdot, y^{\ast}_l)\|_{C^{\frac{1+\alpha_1}{2k}}([0,T])} +|p_0|+|p_q|\\
&~~~~~~~~~~~~~~~~~~~~~~~~~~+\sum\limits_{\mu=1}^{k-1} \left( |b^{\mu}_{1,0}|+ |b^{\mu}_{q, 2 \{ \frac{q}{2} \}}|   \right)+\sum^{q-1}_{l=1}\Vert g_{l,0} \Vert_{C^{2k+\alpha_1}([0,1])}\bigg{ )},
\end{align}
where the constant $\bar{C}_0>0$ is independent of the non-homogeneous terms and the initial-boundary datum in the linear parabolic system, \eqref{eq:higher-order-linear-g_l}. 
\end{lem}

\begin{proof} 

The proof relies on the application of Solonnikov's theory, which requires verifying the sufficient conditions stated in Assumption \ref{assumption:complementary-condition} and Theorem \ref{Solonnikov-theorem}. 

Observe that the left-hand side of the linear parabolic system 
\eqref{eq:higher-order-linear-g_l} 
can be written as a matrix multiplication 
$\mathcal{L}(y, t, \partial_{y}, \partial_{t})g^{T},$ 
where $g=(g_{1}, \ldots, g_{q}) \in\mathbb{R}^{nq}$, $g_{\nu}=(g^{1}_{\nu}, \ldots, g^{n}_{\nu})$, 
$\forall\, \nu\in\{1, \ldots, q \}$,  
$\mathcal{L}(y, t, \partial_{y}, \partial_{t})
= \text{diag} \big(\mathcal{L}_{1 1}, \ldots, 
\mathcal{L}_{nq \, nq}\big)$ 
is a diagonal $nq\times nq$ matrix, 
and 
\begin{equation*}
\mathcal{L}_{\ell \ell}(y,t,\partial_{y},\partial_{t})=
\partial_{t}+(-1)^{k}\partial_{y}^{2k}, 
~~~~~~\ell=(l-1)n+j, ~~ j \in \{1, \ldots, n\},~~l \in \{1, \ldots, q\}.
\end{equation*}  
Thus, we obtain  
\begin{equation*}
\mathcal{L}_{\ell \ell}(y, t, \mathrm{i} \xi, p)=
p+\xi^{2k}, ~~~~~~~~~~~~~~~~~~~~~~~
\ell=(l-1)n+j, ~~ j \in \{1, \ldots, n\},~~l \in \{1, \ldots, q\}. 
\end{equation*} 
Note that as following the notation in Solonnikov theory \cite{Solonnikov65}, the principal part $\mathcal{L}_0$ of the matrix $\mathcal{L}$ coincides with $\mathcal{L}$, i.e., $\mathcal{L}_0=\mathcal{L}$, in this article. 
Hence, we let  
\begin{align*}
L(y, t, \mathrm{i}\xi, p):= \det\mathcal{L}_0(y, t, \mathrm{i} \xi, p)
=
\left(p+ \xi^{2k}\right)^{nq}
\text{,}
\end{align*} 
and   
\begin{align*}
\mathcal{\hat{L}}_0 (y, t, \mathrm{i} \xi, p)
:=L(y, t, \mathrm{i}\xi, p)\mathcal{L}_0^{-1}(y, t, \mathrm{i}\xi, p) 
= \text{diag} \big(\widehat{\mathcal{L}}_{1 1}^0, \ldots, \widehat{\mathcal{L}}_{nq \, nq}^0 \big) 
\text{,}
\end{align*} 
where 
\begin{equation*}
\mathcal{\hat{L}}_{\ell \ell}^0=
(p+ \xi^{2k})^{nq-1}, 
~~~~~~ \ell=(l-1)n+j, j \in \{1, \ldots, n\}, l \in \{1, \ldots, q\}.
\end{equation*}

\smallskip 

{\bf $\bullet$ The parabolicity conditions.} 
(see \cite[Page 8]{Solonnikov65} or Definition \ref{def:parabolicity-condition}.)

\smallskip 
 
From the definition, the roots of the polynomial 
$L(x, t, \mathrm{i}\xi, p)$ 
satisfy 
$$ p=-\xi^{2k} 
~~\text{ with multiplicity } nq, 
~~~~~~~~\xi \in \mathbb{R}. 
$$ 
Therefore, by definition, the system is uniformly parabolic.

\smallskip 

{\bf $\bullet$ The compatibility conditions of the initial datum at the boundaries.} 
(see Definition \ref{def:cc-order-p-fitting}; for general parabolic systems, see  
\cite[Page 98]{Solonnikov65} or Definition \ref{def:compatibility-conditions-general-equa}.)

\smallskip 

Since the boundary operators for $g$ in \eqref{eq:higher-order-linear-g_l} are linear and 
$g_{l}(0,\cdot)=\bar{g}_{l}(0,\cdot)=g_{l,0}$, $\bar{h}_{l}(0,\cdot)=h_{l,0}$, $\forall\, l$, 
we may conclude that $g_0$ also fulfills the compatibility conditions of order zero to the linear equation \eqref{eq:higher-order-linear-g_l}, if $(g_0,h_0)$ also fulfills the compatibility conditions of order zero to the nonlinear parabolic system in Definition 
\ref{def:cc-order-p-fitting}.

\smallskip 

{\bf $\bullet$ The polynomial $M^{+}$.} 
(see \cite[Page 11]{Solonnikov65} or Definition \ref{def:complementary-condition} and Remark \ref{Remark:Complementary-condition}.) 

\smallskip 

As $\text{Re} \text{ } p \geq 0$, and $p \ne 0$, we write 
\begin{align}\label{defintion:p}
p=|p|e^{\mathrm{i} \theta_{p}}, 
~~~~~~~~\text{ with }  
-\frac{1}{2}\pi \leq \theta_{p} \leq \frac{1}{2}\pi, 
\end{align}
and define 
\begin{align*} 
&\xi_{\mu}(y^{\ast}, p)
= d_p \cdot e^{\mathrm{i} \frac{\mu}{k}\pi},  
~~~~~~~~~~~~~~ \mu \in \{1, \ldots, 2k\}, 
\end{align*} 
where 
\begin{align} 
\label{definition:d_p}
&d_p= |p|^{\frac{1}{2k}}\cdot e^{\mathrm{i} \frac{\theta_p-\pi}{2k}}.
\end{align} 
Observe from its definition that the polynomial 
$L=L(y, t, \mathrm{i}\xi, p)$ in variable $\xi$ has 
$knq$ roots with  positive real parts 
and  $knq$ roots with negative real parts. 
The roots with positive real parts are 
$\{  \xi_{\mu}(y^{\ast}, p): \mu=1, \ldots, k \}$  
with multiplicity $nq$ for each fixed $\mu$.  
The roots with negative real parts are 
$\{ \xi_{\mu}(y^{\ast}, p): \mu=k+1, \ldots, 2k \}$  
with multiplicity $nq$ for each fixed $\mu$.  
Thus, we obtain  
\begin{align*}
L=\prod_{\mu=1}^{2k}\left(\xi-\xi_{\mu}(y^{\ast}, p)\right)^{nq} 
\text{.}
\end{align*}
From the definition of polynomial $M^{+}(y^{\ast}, \xi, p)$, 
we derive   
\begin{align*} 
M^{+}(y^{\ast}, \xi, p)
=\prod\limits_{\mu=1}^{k} \left(\xi-\xi_{\mu}(y^{\ast}, p)\right) ^{nq} . 
\end{align*}

\smallskip

{\bf $\bullet$ The complementary conditions along the boundaries $[0,T]\times \partial{I}$.} 
(see \cite[Page 11]{Solonnikov65} or Definition \ref{def:complementary-condition}.)

\smallskip 

The boundary conditions in \eqref{eq:higher-order-linear-g_l} can be rewritten as 
\begin{align*}
\mathcal{B}(y^{\ast}, \partial_{y}) 
g^{T}(t, y^{\ast})=\Phi^{\bar{h}}(t, y^{\ast}), 
~~~y^{\ast} \in \{0,1\}. 
\end{align*} 
The matrix $\mathcal{B}(y^{\ast}, \partial_{y})$ represents a simplified notation of 
$\mathcal{B}_0(y^{\ast}, t, \partial_{y}, \partial_{t})$ in Solonnikov's theory (e.g., see Definition \ref{def:complementary-condition}), when $\mathcal{B}_0$ is independent of $t$ and $\partial_t$, 
and is given by 
\begin{align*}
\begin{cases}
(i) \text{ when $q$ is even:} 
\\ 
\mathcal{B}(y^\ast=0, \partial_{y})=
\text{diag}(B_0, \overbrace{B_1, \ldots, B_1}^{\frac{q}{2}-1}, B_0), 
~~~
\mathcal{B}(y^\ast=1, \partial_{y})=\text{diag}(\overbrace{B_1, \ldots, B_1}^{\frac{q}{2}}), 
\\ 
\\ 
(ii) \text{ when $q$ is odd:}  
\\
\mathcal{B}(y^{\ast}=0, \partial_{y})=\text{diag}(B_0, \overbrace{B_1, \ldots, B_1}^{\frac{q-1}{2}}),
 ~~~~
\mathcal{B}( y^{\ast}=1, \partial_{y})=\text{diag}(\overbrace{B_1, \ldots, B_1}^{\frac{q-1}{2}}, B_0), 
\end{cases}
\end{align*}  
where   
\begin{align*} 
B_{0}
=& B_{0}(\partial_{y})
=\begin{pmatrix}
Id_{n \times n}
\\ 
\partial_{y} \cdot Id_{n \times n} 
\\  
\vdots  
\\
\partial_{y} ^{k-1} \cdot Id_{n \times n}
\\ 
\end{pmatrix}\text{,}
\\ 
B_1
=&B_{1}(\partial_{y})
=\begin{pmatrix}
Id_{n \times n}&-Id_{n \times n}
\\ 
\partial_{y}\cdot Id_{n \times n}&  \partial_{y} \cdot Id_{n \times n}
\\ 
\vdots & \vdots
\\ 
 \partial_{y}^{2k-3}\cdot Id_{n \times n}
 & \partial_{y}^{2k-3} \cdot Id_{n \times n}
 \\ 
 \partial_{y}^{2k-2}\cdot Id_{n \times n}& -\partial_{y}^{2k-2} \cdot Id_{n \times n}
 \\ 
 \partial_{y}^{2k-1}\cdot Id_{n \times n}& \partial_{y}^{2k-1} \cdot Id_{n \times n}
 \\ 
\end{pmatrix}.
\end{align*} 
Hence, $\mathcal{B}(y^{\ast}, \partial_{y})$ are diagonal block matrices. Moreover, $\Phi^{\bar{h}}$ is defined by
\begin{equation}\label{definition:Phi-bar-h-q-even}
\begin{cases}
\big( \Phi^{\bar{h}}_1(0), \ldots, \Phi^{\bar{h}}_{nk}(0) \big)=\big(p_0, 
b^{1}_{1,x_0}, \ldots, b^{k-1}_{1,x_0} \big)\in \mathbb{R}^{nk}, 
\\
\big( \Phi^{\bar{h}}_{kn(q-1)-nk+1}(0), \ldots, \Phi^{\bar{h}}_{kn(q-1)}(0) \big)
\\ 
~~~~~~~~~~~~~~~~~~~~~~~~~~~~~~~~~~~~~~ 
=\bigg(p_q, (-1)^{q-1}  b^{1}_{q,x_q}, \ldots, (-1)^{(q-1)(k-1)} b^{k-1}_{q,x_q} \bigg) \in \mathbb{R}^{nk}, 
\\
\big( \Phi^{\bar{h}}_{nk+n(2k-1)\nu+1}(0), \ldots, \Phi^{\bar{h}}_{nk+n(2k-1)\nu+n}(0) \big)=\frac{(-1)^k}{\sigma^2} \partial_y \bar{h}_l(t,0) \in \mathbb{R}^{n}, 
\\
~~~~~~~~~~~~~~~~~~~~~~~~~~~~~~~~~~~~~~~~~~~~~~~~~~~~~~~~~~~~~~~~~~~~~~~~~~~~
\nu \leq \frac{q}{2}-1, 1 \leq l \leq q-1, 
\\
\big( \Phi^{\bar{h}}_{n(2k-1)\nu+1}(1), \ldots, \Phi^{\bar{h}}_{n(2k-1)\nu+n}(1) \big)=\frac{(-1)^k}{\sigma^2} \partial_y \bar{h}_l(t,1) \in \mathbb{R}^{n}, 
~ 
\nu \leq \frac{q}{2}, 1 \leq l \leq q-1 ,
\\
\Phi^{\bar{h}}_i(y^{\ast})=0, 
~~~~~~~~~~~~~~~~~~~~~~~~~~~~~~~~~~~~~~~~~~~~~~~~~~~~~~~~~~~~~~~~~ 
\text{ for the rest of } i, ~ y^{\ast}, 
\end{cases}
\end{equation}
as $q$ is even, and 
\begin{equation}\label{definition:Phi-bar-h-q-odd}
    \begin{cases}
        \big( \Phi^{\bar{h}}_1(0), \ldots, \Phi^{\bar{h}}_{nk}(0) \big)=\big(p_0, 
b^{1}_{1,x_0}, \ldots, b^{k-1}_{1,x_0} \big)\in \mathbb{R}^{nk}, 
\\
\big( \Phi^{\bar{h}}_{kn(q-1)-nk+1}(1), \ldots, \Phi^{\bar{h}}_{kn(q-1)}(1) \big)
\\ 
~~~~~~~~~~~~~~~~~~~~~~~~~~~~~~~~~~~~~~~~~ 
=\bigg(p_q, (-1)^{q-1}  b^{1}_{q,x_q}, \ldots, (-1)^{(q-1)(k-1)} b^{k-1}_{q,x_q} \bigg) \in \mathbb{R}^{nk}, 
\\
\big( \Phi^{\bar{h}}_{nk+n(2k-1)\nu+1}(0), \ldots, \Phi^{\bar{h}}_{nk+n(2k-1)\nu+n}(0) \big)=\frac{(-1)^k}{\sigma^2} \partial_y \bar{h}_l(t,0) \in \mathbb{R}^{n}, 
\\
~~~~~~~~~~~~~~~~~~~~~~~~~~~~~~~~~~~~~~~~~~~~~~~~~~~~~~~~~~~~~~~~~~~~~~~~~~~~~~~~~ 
\nu \leq \frac{q-1}{2}, 1 \leq l \leq q-1 , 
\\
\big( \Phi^{\bar{h}}_{n(2k-1)\nu+1}(1), \ldots, \Phi^{\bar{h}}_{n(2k-1)\nu+n}(1) \big)=\frac{(-1)^k}{\sigma^2} \partial_y \bar{h}_l(t,1) \in \mathbb{R}^{n}, ~ \nu \leq \frac{q-1}{2}, 1 \leq l \leq q-1 ,
\\
\Phi^{\bar{h}}_i(y^{\ast})=0, 
\Phi^{\bar{h}}_i(y^{\ast})=0, 
~~~~~~~~~~~~~~~~~~~~~~~~~~~~~~~~~~~~~~~~~~~~~~~~~~~~~ 
\text{ for the rest of } i, ~ y^{\ast},  
    \end{cases}
\end{equation}
as $q$ is odd.

To verify the complementary conditions as defined in Solonnikov's theory, it suffices to show that the rows of the matrix  
\begin{align*}
\mathcal{A}(y^{\ast}, t, \mathrm{i}\xi, p)=\mathcal{B}_0(y^{\ast}, \mathrm{i} \xi) \hat{\mathcal{L}}_0(y^{\ast}, t, \mathrm{i} \xi, p)
\end{align*} 
are linearly independent modulo $M^{+}(y^{\ast}, \xi, p)$, where $\text{Re}\{p\} \geq 0$, and $p \ne 0$. 
Here, $\mathcal{B}_0(y^{\ast}, \mathrm{i} \xi)$ denotes a simplified notation for $\mathcal{B}_0(y^{\ast}, t, \mathrm{i} \xi, p)$, since $\mathcal{B}_0$ is independent of both $t$ and $p$ throughout this article.  
Moreover, since the matrix $\mathcal{B}_0(y^{\ast}, \mathrm{i} \xi)$ is composed of block matrices regardless of whether $q$ is even or odd, it suffices to verify the complementary conditions for each individual block at the boundary. 
Note that $\mathcal{A}=\mathcal{A}(y^{\ast}, t, \mathrm{i}\xi, p)$ is itself a diagonal block matrix, as $\hat{\mathcal{L}}_0=\hat{\mathcal{L}}_0(y^{\ast}, t, \mathrm{i}\xi, p)$ is diagonal, and $\mathcal{B}_0=\mathcal{B}_0(y^{\ast}, \mathrm{i}\xi)$ consists of two types of block matrices: 
$B_0(\mathrm{i}\xi)$ and $B_1(\mathrm{i}\xi)$.  
Therefore, it suffices to verify the complementary conditions at $y^{\ast}=0$ (or at $y^\ast=1$) separately for the following block matrices: 
\begin{align*}
\mathcal{A}_1 :=B_1 \cdot \begin{pmatrix}
\mathcal{\hat{L}}_{11}^0 Id_{n \times n} 
\\ 
\mathcal{\hat{L}}_{11}^0 Id_{n \times n} 
\\ 
\end{pmatrix}.
\end{align*}

Observe that 
\begin{align*}
\mathcal{A}_0 
= B_{0}(y^\ast=0, \mathrm{i}\xi) \cdot \mathcal{\hat{L}}_{11}^0(y^\ast=0,t,\mathrm{i} \xi , p) 
Id_{n \times n}
=\begin{pmatrix}
\mathcal{\hat{L}}_{11}^0 \cdot Id_{n\times n} 
\\ 
\mathcal{\hat{L}}_{11}^0 \cdot \mathrm{i} \xi \cdot Id_{n \times n} 
\\ 
\mathcal{\hat{L}}_{11}^0 \cdot (\mathrm{i} \xi)^{2} \cdot Id_{n \times n} 
\\
\vdots 
\\
\mathcal{\hat{L}}_{11}^0 \cdot (\mathrm{i} \xi)^{k-1} \cdot Id_{n \times n} 
\\
\end{pmatrix}.
\end{align*} 
To verify the linear independence of the rows of matrix $\mathcal{A}_0$ modulo $M^{+}(y^\ast=0, \xi, p)$, we need to show that 
\begin{align}
\label{comple-xi-higher-order}
\omega \cdot \mathcal{A}_0(y^\ast=0, t, \mathrm{i}\xi, p)=0 
~\text{  mod }  M^{+}(y^\ast=0, \xi, p) 
~~ \Longrightarrow ~~ \omega=(\omega_{1}, \ldots, \omega_{kn})=0. 
\end{align} 
The algebraic equation \eqref{comple-xi-higher-order} 
can be written as 
\begin{align*} 
\left(\sum\limits^{k-1}_{\mu=0}(\mathrm{i}  \xi)^{\mu} \cdot \omega_{\mu n+j}\right) (p+ \xi^{2k})^{nq-1} =0 
~ \text{ mod }  M^{+}(y^\ast=0, \xi, p), 
~~ \forall\, j \in \{1, \ldots, n\}. 
\end{align*} 
By dividing both sides of the above equation by  $\prod\limits_{\mu=1}^{k} \left(\xi-\xi_{\mu}(y^{\ast}=0, p)\right)^{nq-1}$, we obtain 
\begin{align*}
\left(\sum\limits^{k-1}_{\mu=0}(\mathrm{i}  \xi)^{\mu} \cdot \omega_{\mu n+j}\right) \prod\limits_{\mu=k+1}^{2k} \left(\xi-\xi_{\mu}(y^{\ast}=0, p)\right)^{nq-1}  =0 ~ \text{ mod }  \prod\limits_{\mu=1}^{k} \left(\xi-\xi_{\mu}(y^{\ast}=0, p)\right),
\end{align*}
$\forall\, j \in \{1, \ldots, n\}$. 
This equation is equivalent to 
\begin{align}
\label{eq-first-m-column}
\omega_{j}+\mathrm{i}  \xi  \omega_{n+j}+(\mathrm{i}  \xi)^{2}  \omega_{2n+j}
+\cdots +(\mathrm{i}  \xi)^{k-1}  \omega_{(k-1)n+j}=0 
~ \text{ mod } \prod\limits^{k}_{\mu=1}\left(\xi-\xi_{\mu}(y^\ast=0,p)\right)
\text{,}
\end{align}
$\forall\, j \in \{1, \ldots, n\}$, 
since the polynomial $\prod\limits_{\mu=k+1}^{2k} \left(\xi-\xi_{\mu}(y^{\ast}, p)\right)^{nq-1}$ is not divisible by the polynomial 
$\prod\limits_{\mu=1}^{k} \left(\xi-\xi_{\mu}(y^{\ast}, p)\right)$.  
By substituting $\xi=\xi_{\mu}$ into \eqref{eq-first-m-column}, 
we obtain 
\begin{align}
\label{system-first-m-column} 
\bar{\omega}_{j} \cdot C
=0, 
\end{align} 
where 
\begin{align*} 
\bar{\omega}_{j}=
(\omega_{j}, \omega_{n+j}, \omega_{2n+j},\cdots, \omega_{(k-1)n+j}), 
\end{align*} 
and 
\begin{align*} 
C= \begin{pmatrix}
1&\mathrm{i} \xi_{1}&(\mathrm{i} \xi_{1})^{2}&\cdots&(\mathrm{i} \xi_{1})^{k-1} \\  
1&\mathrm{i} \xi_{2}&(\mathrm{i} \xi_{2})^{2}&\cdots&(\mathrm{i} \xi_{2})^{k-1} \\  
1&\mathrm{i} \xi_{3}&(\mathrm{i} \xi_{3})^{2}&\cdots&(\mathrm{i} \xi_{3})^{k-1} \\   
\vdots&\vdots&\vdots&\ddots&\vdots \\   
1&\mathrm{i} \xi_{k}&(\mathrm{i} \xi_{k})^{2}&\cdots&(\mathrm{i} \xi_{k})^{k-1} \\   
\end{pmatrix}. 
\end{align*} 
Notice that the matrix $C$ is the so-called Vandermonde matrix with the property, 
\begin{align*}
\det C
=\prod\limits_{1 \leq \mu< \nu \leq k}(\mathrm{i} \xi_{\nu}-\mathrm{i} \xi_{\mu}) \ne 0.
\end{align*}
This implies that any solution to 
\eqref{system-first-m-column} is trivial, i.e., $\bar{\omega}_{j}=0$, $\forall\, j \in \{1, \ldots, n\}$. 
Thus, \eqref{comple-xi-higher-order} holds.

Next, we show that 
\begin{align}
\label{comple-xi-y-1-fitting} 
\omega \cdot \mathcal{A}_1 =0  
~ \text{ mod }  M^{+}(y^\ast=0, \xi, p)
~ \Longrightarrow ~ \omega=(\omega_{1}, \ldots,\omega_{2kn})=0,  
\end{align} 
where 
\begin{align*}
\mathcal{A}_1 
=
\begin{pmatrix}
Id_{n \times n} \cdot \mathcal{\hat{L}}_{11}^0& -Id_{n \times n}   \cdot \mathcal{\hat{L}}_{11}^0 
\\ 
\mathrm{i} \xi \cdot Id_{n \times n} \cdot \mathcal{\hat{L}}_{11}^0 & \mathrm{i} \xi \cdot Id_{n \times n} \cdot \mathcal{\hat{L}}_{11}^0
\\ 
(\mathrm{i} \xi)^{2} \cdot  Id_{n \times n} \cdot \mathcal{\hat{L}}_{11}^0& -(\mathrm{i} \xi)^{2} \cdot Id_{n \times n} \cdot \mathcal{\hat{L}}_{11}^0 
\\ 
\vdots  &\vdots 
\\ 
 (\mathrm{i} \xi)^{2k-2} \cdot Id_{n \times n} \cdot  \mathcal{\hat{L}}_{11}^0&  - (\mathrm{i} \xi)^{2k-2} \cdot  Id_{n \times n} \cdot \mathcal{\hat{L}}_{11}^0
 \\ 
(\mathrm{i} \xi)^{2k-1}\cdot Id_{n \times n} \cdot \mathcal{\hat{L}}_{11}^0 &(\mathrm{i} \xi)^{2k-1} \cdot  Id_{n \times n} \cdot \mathcal{\hat{L}}_{11}^0
\\ 
\end{pmatrix}. 
\end{align*} 
From a direct computation like above, 
the sufficient condition in 
\eqref{comple-xi-y-1-fitting} is equivalent to 
\begin{align*} 
\begin{cases}
\sum\limits_{\mu=0}^{2k-1} (\mathrm{i} 
\xi )^{\mu}\cdot \omega_{\mu n+j} =0 
~~~~~~~~~~~~~~~~~~~~~~~ \text{mod } \prod\limits^{k}_{\mu=1}\left(\xi-\xi_{\mu}(y^\ast=0,p)\right),  
\\
\sum\limits_{\mu=0}^{2k-1}
(-1)^{\mu-1}(\mathrm{i}\xi )^{\mu}\cdot\omega_{\mu n+j}=0 
~~~~~~~~~~~~ \text{mod } \prod\limits^{k}_{\mu=1}\left(\xi-\xi_{\mu}(y^\ast=0,p)\right),
\end{cases} 
\Rightarrow 
\omega=0. 
\end{align*} 
By substituting $\xi=\xi_r$, for all $r\in\{1, \ldots, k\}$, 
in these two equations, we obtain 
\begin{equation*}
\begin{cases}
\sum\limits_{\mu=0}^{2k-1} (\mathrm{i} \xi_{r})^{\mu} \cdot \omega_{ \mu n+j} =0, 
\\
\sum\limits_{\mu=0}^{2k-1} (-1)^{\mu-1} (\mathrm{i} \xi_{r})^{\mu} \cdot  \omega_{\mu n+j} =0 ,
\end{cases} 
~~ \forall\, r\in \{1, \ldots, k\}, ~~ \forall\, j \in \{1, \ldots, n\}. 
\end{equation*} 
By subtraction and addition between these two equations, we obtain 
\begin{equation}\label{D+E-1-fitting}
\begin{cases}
2\sum\limits_{\mu=0}^{k-1} (\mathrm{i}\cdot  \xi_{r})^{2\mu} \cdot \omega_{2\mu n+j} =0, 
\\ 
2\sum\limits_{\mu=1}^{k} (\mathrm{i}\cdot  \xi_{r})^{2\mu-1} \cdot \omega_{(2\mu-1) n+j} =0, 
\end{cases} 
~~~~~~~~~ \forall\, r\in \{1, \ldots, k\}, ~~ \forall\, j \in \{1, \ldots, n\}. 
\end{equation} 
The first equation in \eqref{D+E-1-fitting} can be written as 
\begin{align}
\label{eq:Dw=0}
\bar{\omega}_{j} \cdot D=0 
\end{align} 
where 
$\bar{\omega}_{j}$ is defined as 
\begin{align*}
\bar{\omega}_{j}=(\omega_{j}, ~~ \mathrm{i}^{2} \cdot \omega_{2n+j}, ~~ \mathrm{i}^{4} \cdot \omega_{4n+j}, \ldots, ~~ \mathrm{i}^{2(k-1)} \cdot \omega_{(2k-2)n+j}), 
\end{align*}
and 
\begin{align}
\label{def:D} 
D=\begin{pmatrix}
1&\xi_{1}^2&\xi_{1}^4&\cdots&\xi_{1}^{2k-2} 
\\ 
1&\xi_{2}^2&\xi_{2}^4&\cdots&\xi_{2}^{2k-2} 
\\ 
1&\xi_{3}^2&\xi_{3}^4&\cdots&\xi_{3}^{2k-2} 
\\ 
\vdots&\vdots&\vdots&\ddots&\vdots 
\\ 
1&\xi_{k}^2&\xi_{k}^4&\cdots&\xi_{k}^{2k-2} 
\\ 
\end{pmatrix} 
\end{align} 
is the so-called Vandermonde matrix with the property, 
\begin{align}
\label{eq:det_D}
\det{D}=\prod_{1 \leq \mu <\nu \leq k}\left(\xi_{\nu}^2-\xi_{\mu}^2\right) 
\ne 0, 
\end{align} 
where $\xi_{\mu} \ne \xi_{\nu}, \forall\, \mu\ne \nu$.  
Hence, any solution to \eqref{eq:Dw=0} is trivial. 
It follows that the first equation in \eqref{D+E-1-fitting} can be readily solved, yielding     
\begin{align*} 
 \omega_{2\mu n+j}=0, 
~~~~~~ \forall\, \mu \in \{0, 1, \ldots, k-1\}, \, j\in \{1, \ldots, n\}. 
\end{align*}
The second equation in \eqref{D+E-1-fitting} can be rewritten as  
\begin{equation*} 
\tilde{\omega} \cdot E=0, 
\end{equation*} 
where 
\begin{equation*}
\tilde{\omega} =\begin{pmatrix}
\mathrm{i} \cdot \omega_{n+j}, &\mathrm{i}^3 \omega_{3n+j}, &\mathrm{i}^5 \omega_{5n+j}, &\ldots, & \mathrm{i}^{2k-1} \omega_{(2k-1)n+j}\\ 
\end{pmatrix},
\end{equation*}
and 
\begin{align*}
E=\begin{pmatrix} 
\xi_{1}& \xi_{1}^3&\cdots& \xi_{1}^{2k-3}&\xi_{1}^{2k-1} 
\\ 
\xi_{2}&\xi_{2}^3&\cdots&\xi_{2}^{2k-3}& \xi_{2}^{2k-1} 
\\ 
\xi_{3}&\xi_{3}^3&\cdots&\xi_{3}^{2k-3}& \xi_{3}^{2k-1} 
\\ 
\vdots&\vdots&\ddots&\vdots&\vdots
\\ 
\xi_{k}& \xi_{k}^3&\cdots&\xi_{k}^{2k-3}&\xi_{k}^{2k-1} 
\end{pmatrix}.
\end{align*} 
Note that 
$$
\det{E}=\underset{1 \leq j \leq k}{\prod} \xi_j \cdot \det{D} \ne 0. 
$$ 
This implies that 
\begin{align*} 
 \omega_{(2\mu-1)n+j}=0, 
~~~~~~ \forall\, \mu \in \{1, \ldots, k\}, \, j\in \{1, \ldots, n\}. 
\end{align*} 
Combining the above, we conclude that 
$\omega=(\omega_1,\cdots,\omega_{2kn})=0$.
Therefore, the complementary conditions at the boundary are satisfied.

\smallskip 

{\bf $\bullet$ The complementary conditions of the initial datum.} 
(see \cite[Page 12]{Solonnikov65} or Definition \ref{def:complementary-condition}.) 

\smallskip 
  
The matrix associated with the initial conditions has the form 
\begin{align*}
\mathcal{C}_0(y,\partial_{y},\partial_{t})=Id_{nq \times nq}.
\end{align*} 
It follows from the definitions of $\mathcal{C}_0(y, 0, p)$ and $\hat{\mathcal{L}}_0(y, t, \mathrm{i}\xi, p)$
that the matrix 
\begin{align*}
\mathcal{D}(y,p)=\mathcal{C}_0(y, 0, p)  \mathcal{\hat{L}}_0(y, 0, 0, p)= \text{diag}(\underbrace{p^{nq-1}, \ldots, p^{nq-1}}_{nq  \text{ elements } }) 
\end{align*} 
has rows that are linearly independent modulo $p^{nq}$.
Therefore, the complementary conditions for the initial data $g_0$ are satisfied.

\smallskip 

{\bf $\bullet$ The smoothness of coefficients.} 

\smallskip 

From applying Lemmas \ref{lem:RemarkB1} and \ref{lem:tech2} to the assumption,  
$\bar{g}_l \in C^{\frac{2k+\alpha_1}{2k}, 2k+\alpha_1}([0,T] \times [0,1])$,  
the coefficients in the linear parabolic system \eqref{eq:higher-order-linear-g_l} fulfill 
\begin{equation*}
G_l(\partial_y^{2k-1} \bar{g}_{l}, \ldots, \bar{g}_{l}) 
\in C^{\frac{1+\alpha_1}{2k}, 1+\alpha_1}([0,T] \times [0,1]) \subset C^{\frac{\alpha_1}{2k}, \alpha_1}([0,T] \times [0,1]), ~~~~\forall\, l \in \{1, \ldots, q\},
\end{equation*} 
and 
\begin{equation*}
\bar{h}_l (\cdot, y^{\ast}_l) 
\in C^{\frac{2+\alpha_2}{2}}([0,T]) \subset C^{\frac{1+\alpha_1}{2k}}([0,T]), ~~~~\forall\, y^{\ast}_l=1-2 \bigg\{ \frac{l+1}{2} \bigg\}, ~ l \in \{1, \ldots, q-1\}.
\end{equation*} 
By assumption we have $g_{l,0} \in C^{2k+\alpha_1}([0,1])$, $\forall\, l\in \{1, \ldots, q-1\}$. 
Hence, the required smoothness of coefficients in applying 
Theorem \ref{Solonnikov-theorem} in the H\"{o}lder space  $C^{\frac{2k+\alpha_1}{2k}, 2k+\alpha_1}([0,T] \times [0,1])$ is fulfilled. 

\bigskip 

We have now verified all the required sufficient conditions stated in Assumption \ref{assumption:complementary-condition} and 
Theorem \ref{Solonnikov-theorem}. 
Consequently, the results including the estimate \eqref{estimate-for-g_l} follow directly by applying 
the inequality 
\begin{align}\label{cauchy-inequality}
|a_1| + \cdots + |a_n| \leq \sqrt{n(a_1^2 + \cdots + a_n^2)}
\text{,} 
\end{align}
and Theorem \ref{Solonnikov-theorem} with the specific choice of data as follows:

Initial datum: $\phi=g_0$,

Boundary datum: 
$\Phi=\Phi^{\bar{h}}$, as given by \eqref{definition:Phi-bar-h-q-even} and \eqref{definition:Phi-bar-h-q-odd},

Inhomogeneous lower-order term: 
\begin{align*}
f=\bigg( G_1(\partial_y^{2k-1} \bar{g}_{1}, \ldots, \bar{g}_{1}), \ldots, G_q(\partial_y^{2k-1} \bar{g}_{q}, \ldots, \bar{g}_{q})  \bigg)
\end{align*} 
in the linear parabolic system \eqref{eq:Solonnikov}. 
\end{proof}

\begin{lem}[Solutions to the linear parabolic system \eqref{eq:second-order-linear-h_l}]
\label{lem:linear-system-h-l}
Suppose all the assumptions in Lemma \ref{lem:linear-system-g-l}  are satisfied, and $g_l$ is a solution to \eqref{eq:higher-order-linear-g_l}. 
Then, for any $T>0$, there exists a unique solution $h_l$ to the linear parabolic system \eqref{eq:second-order-linear-h_l} such that $h_{l}\in C^{\frac{2+\alpha_2}{2}, 2+\alpha_2}([0,T]\times [0,1])$, $\forall\, l \in \{1, \ldots, q-1\}$, and 
\begin{align}\label{estimate-for-h_l}
\sum\limits^{q}_{l=1} 
& \norm{h_{l}}_{C^{\frac{2+\alpha_2}{2},2+\alpha_2}([0,T] \times [0,1])} \leq 
\bar{C}_0
\cdot \bigg{(} \sigma^2 \sum^{q-1}_{l=1}\norm{ W_1(\partial_y \bar{h}_l, \bar{h}_l)}_{C^{\frac{\alpha_2}{2},\alpha_2}([0,T] \times [0,1])}) 
\\ \nonumber
&~~~~~~~~~~~~~~~
+\sum^{q-1}_{l=1}\, |p_{l}|
+\sum_{\substack{l=1, \cdots, q-1 \\ y^{\ast}_l =1-2 \big\{ \frac{l+1}{2} \big\} }} 
\|g_l (\cdot, y^{\ast}_l)\|_{C^{\frac{2+\alpha_2}{2}}([0,T])}  +\sum^{q-1}_{l=1}\Vert h_{l,0} \Vert_{C^{2+\alpha_2}([0,1])}\bigg{ )},
\end{align}
where the constant $\bar{C}_0>0$ is independent of the non-homogeneous terms and the initial-boundary datum in the linear parabolic system, \eqref{eq:second-order-linear-h_l}. 
\end{lem}

\begin{proof} 

Observe that the left-hand side of the linear parabolic system 
\eqref{eq:second-order-linear-h_l} 
can be expressed in the form of a matrix multiplication:  
$\mathcal{L}(y, t, \partial_{y}, \partial_{t})h^{T},$ 
where $h=(h_1, \ldots, h_{q-1}) \in \mathbb{R}^{n(q-1)}$, with $h_\nu=(h_\nu^1, \ldots, h_\nu^{n})$,  
and 
$\mathcal{L}(y, t, \partial_{y}, \partial_{t})
=\text{diag} \big(\mathcal{L}_{1 1}, \ldots, \mathcal{L}_{n(q-1) \, n(q-1)} \big) 
$ 
is a diagonal matrix of size $n(q-1)\times n(q-1)$ matrix. Each diagonal entry is given by 
\begin{equation*}
\mathcal{L}_{\ell \ell}(y,t,\partial_{y},\partial_{t})=
\partial_{t}- \sigma^2 \partial_{y}^{2}, 
~~~~~~~~~~ 
\ell=(l-1) n+j, j \in \{1, \ldots, n\}, ~ l \in \{1, \ldots, q-1\}.
\end{equation*}  
Thus, we obtain  
\begin{equation*}
\mathcal{L}_{\ell \ell}(y, t, \mathrm{i} \xi, p)=
p+ \sigma^2 \xi^{2}, 
~~~~~~~~~~~~~~~~
\ell=(l-1)n+j, ~ j \in \{1, \ldots, n\}, ~ l \in \{1, \ldots, q-1\}.
\end{equation*} 
Note that, in this article, the principal part $\mathcal{L}_0$ of the matrix $\mathcal{L}$ coincides, i.e., $\mathcal{L}_0=\mathcal{L}$. 
Hence, we let  
\begin{align*}
L(y, t, \mathrm{i}\xi, p):= \det\mathcal{L}_0(y, t, \mathrm{i} \xi, p)
=\left(p+ \sigma^2 \xi^{2}\right)^{n(q-1)}\text{,}
\end{align*} 
and   
\begin{align*}
\mathcal{\hat{L}}_0 (y, t, \mathrm{i} \xi, p)
:=L(y, t, \mathrm{i}\xi, p)\mathcal{L}_0^{-1}(y, t, \mathrm{i}\xi, p) 
= \text{diag} \big(\widehat{\mathcal{L}}_{1 1}^0, \ldots, \widehat{\mathcal{L}}_{n(q-1) \, n(q-1)}^0 \big) 
\text{,}
\end{align*} 
where 
\begin{equation}\label{eq:L_ll^0}
\mathcal{\hat{L}}_{\ell \ell}^0=
(p+ \sigma^2 \xi^{2})^{n(q-1)-1}, 
~\ell=(l-1)n+j, ~ j \in \{1, \ldots, n\}, ~l \in \{1, \ldots, q-1\}.
\end{equation}

\smallskip 

{\bf $\bullet$ The parabolicity condition.} 
(\cite[Page 8]{Solonnikov65} or Definition \ref{def:parabolicity-condition}.)

\smallskip 

From the definition, the roots of the polynomial 
$L(x, t, \mathrm{i}\xi, p)$ satisfy
$$ p=-\sigma^2 \xi^{2} ~~\text{ with multiplicity } n(q-1), 
~~~\xi \in \mathbb{R}. 
$$
Therefore, the system is uniformly parabolic by the definition.

\smallskip 

{\bf $\bullet$ The compatibility conditions of the initial datum at the boundaries.} 
(see Definition \ref{def:cc-order-p-fitting}; for general parabolic systems, see 
\cite[Page 98]{Solonnikov65} or Definition \ref{def:compatibility-conditions-general-equa}.

\smallskip 

Since the boundary operators for $h$ in \eqref{eq:second-order-linear-h_l} are linear, the compatibility conditions of order zero for the linear parabolic system  \eqref{eq:second-order-linear-h_l} (See Definition \eqref{def:compatibility-conditions-general-equa}) are satisfied if 
\begin{equation}
\label{eq:cc_for_h_l}
\begin{cases}
h_{l,0}(y^{\ast}_l)= p_{l}, 
~~~~~~~~~~~~~~~~~~~~~~~~~~~~~~~~~~~~~~~~~~~~~~~~~~~~~~~~~~~~~~~ 
y^{\ast}_l= 1- 2\{\frac{l}{2}\},   
\\
 h_{l,0}(y^{\ast}_l)=g_{l}(0, y^{\ast}_l), 
~~~~~~~~~~~~~~~~~~~~~~~~~~~~~~~~~~~~~~~~~~~~~~~~~~~~
y^{\ast}_l= 1- 2\{\frac{l+1}{2}\},
\\
\sigma^2 \partial_y^{2} h_{l,0}(y^{\ast}_l) 
+ \sigma^2 W_1(\partial_y \bar{h}_{l}(0, y^{\ast}_l), \bar{h}_{l}(0, y^{\ast}_l))=0, 
~~~~~~~~~~~~~~~~~~~ 
y^{\ast}_l= 1- 2\{\frac{l}{2}\},  
\\
\sigma^2 \partial_y^{2} h_{l,0}(y^{\ast}_l) 
+ \sigma^2 W_1(\partial_y \bar{h}_{l}(0, y^{\ast}_l), \bar{h}_{l}(0, y^{\ast}_l) )=\partial_t g_l(t,y^{\ast}_l)_{|_{t=0}}, 
~
y^{\ast}_l= 1- 2\{\frac{l+1}{2}\},
\end{cases} 
\end{equation} 
for all $l \in \{1, \ldots, q-1\}$. 
Below, we only verify the last equation in \eqref{eq:cc_for_h_l}, since the assumption,  
$(g_0, h_0)$ satisfies the  compatibility conditions of order zero for the nonlinear parabolic system \eqref{AP-q-even-fitting}$\sim$\eqref{BC-q-even-fitting}, 
trivially assures the validity of the first three equations in \eqref{eq:cc_for_h_l}. 

Since $g_l$ is assumed to be a solution to \eqref{eq:higher-order-linear-g_l}, it satisfies
\begin{align*}
\partial_t g_l=(-1)^{k+1}\cdot \partial_y^{2k} g_{l}
+G_l(\partial_y^{2k-1} \bar{g}_{l}, \ldots, \partial_y \bar{g}_{l}, \bar{g}_{l}). 
\end{align*} 
In view of the initial conditions 
$g_{l}(0, \cdot)= g_{l,0}(\cdot )$ and $\bar{g}_{l}(0, \cdot)= g_{l,0}(\cdot )$, together with the smoothness of $g_l$ and $\bar{g}_l$,  we obtain 
\begin{align}
\label{eq:d_t(g_l)}
\partial_t g_l(t,y^{\ast}_l)_{|_{t=0}}=(-1)^{k+1}\cdot \partial_y^{2k} g_{l,0}(y^{\ast}_l)+G_l(\partial_y^{2k-1} g_{l,0}(y^{\ast}_l), \ldots, \partial_y g_{l,0}(y^{\ast}_l), g_{l,0}(y^{\ast}_l)).
\end{align} 
Since the pair $(g_0, h_0)$ satisfies the compatibility conditions of order zero for the nonlinear parabolic system \eqref{AP-q-even-fitting}$\sim$\eqref{BC-q-even-fitting}, equation \eqref{eq:d_t(g_l)} confirms that the final condition in \eqref{eq:cc_for_h_l} holds.
We therefore conclude that $h_0$ satisfies the compatibility conditions of order zero for the linear parabolic system \eqref{eq:second-order-linear-h_l}. 

\smallskip 

{\bf $\bullet$ The polynomial $M^{+}$.} 
(\cite[Page 11]{Solonnikov65} or see Definition \ref{def:complementary-condition} and Remark \ref{Remark:Complementary-condition}.) 

\smallskip 

As $\text{Re } p \geq 0$, and $p \ne 0$, we write 
\begin{align*}
p=|p| e^{\mathrm{i} \theta_p}, ~~~ -\frac{\pi}{2} \leq \theta_p \leq \frac{\pi}{2}, 
\end{align*}
and define 
\begin{align*}
\zeta_\nu(y^{\ast},p)=\frac{\sqrt{|p|}}{|\sigma|} e^{\mathrm{i} \bigg( \frac{\theta_p}{2}+\frac{(2\nu-1)\pi}{2} \bigg)}, 
~~~~\nu \in \{1, 2\}.
\end{align*}

Observe from the definition that the polynomial 
$L=L(y, t, \mathrm{i}\xi, p)$, viewed as a function of $\xi$ has $2n(q-1)$ roots. 
Among them, those with positive real parts are denoted by 
$\zeta_{1}(y^{\ast}, p)$ with multiplicity $n(q-1)$,
while those with negative real parts are denoted by 
$\zeta_{2}(y^{\ast}, p)$, also with multiplicity $n(q-1)$. 
Thus, we obtain  
$$ L=\prod_{\nu=1}^{2}\left(\xi-\zeta_{\nu}(y^{\ast}, p)\right)^{n(q-1)}. $$ 
From the definition of polynomial $M^{+}(y^{\ast}, \xi, p)$, 
we conclude   
\begin{align*} 
M^{+}(y^{\ast}, \xi, p)
=\left(\xi-\zeta_{1}(y^{\ast}, p)\right)^{n(q-1)} . 
\end{align*}

\smallskip 

{\bf $\bullet$ The complementary conditions along the boundaries $[0,T]\times \partial{I}$.} 
(\cite[Page 11]{Solonnikov65} or Definition \ref{def:complementary-condition}.)

\smallskip

The boundary conditions in \eqref{eq:second-order-linear-h_l} can be rewritten as 
\begin{align*}
\mathcal{B}_0(y^\ast,\partial_y) h^{T}(t,y^{\ast})= Id_{n(q-1) \times n(q-1)} h^{T}(t,y^{\ast})=\Phi^{g}(t,y^{\ast}), 
~~~ y^\ast\in\{0,1\},    
\end{align*}
where 
\begin{equation}\label{definition-Phi-bar-g}
\Phi^{g}(t,y^{\ast})=
\begin{cases}
\big(p_1, g_2(t,0), p_3, \ldots, g_{q-2}(t,0), p_{q-1}   \big), 
~~~~~~~~~\text{ as } y^{\ast}=0, ~~ q \text{ odd},
\\
\big(p_1, g_2(t,0), p_3, \ldots, p_{q-2}, g_{q-1}(t,0)   \big), 
~~~~~~~~\text{ as } y^{\ast}=0, ~~ q \text{ even},
\\
\big(g_1(t,1), p_2, \ldots, g_{q-2}(t,1),  p_{q-1} \big), ~~~~~~~~~~~~~\text{ as } y^{\ast}=1, ~~ q \text{ odd},
\\ 
\big( g_1(t,1), p_2, \ldots, p_{q-2}, g_{q-1}(t,1)   \big), 
~~~~~~~~~~~~~\text{ as } y^{\ast}=1, ~~ q \text{ even}.
\end{cases}
\end{equation} 
To verify the complementary conditions, it suffices to show that the rows of the matrix 
\begin{align*}
\mathcal{A}(y^{\ast}, t, \mathrm{i}\xi, p)=&\mathcal{B}_0(y^{\ast},  \mathrm{i}\xi) \cdot \widehat{\mathcal{L}}_0(y^{\ast}, t, \mathrm{i}\xi, p)=Id_{n(q-1) \times n(q-1)} \widehat{\mathcal{L}}_0(y^{\ast}, t, \mathrm{i}\xi, p) \\
=&\text{diag} \big(\widehat{\mathcal{L}}_{1 1}^0, \ldots, \widehat{\mathcal{L}}_{n(q-1) \, n(q-1)}^0 \big) 
\end{align*} 
are linearly independent modulo $M^{+}(y^{\ast}, \xi, p)$, 
where $\widehat{\mathcal{L}}_{\ell \ell}^0$ is defined in \eqref{eq:L_ll^0}, $\text{Re}\{p\} \geq 0$, and $p \ne 0$. 
To verify the linear independence of the rows of matrix $\mathcal{A}$ modulo $M^{+}(y^\ast, \xi, p)$, it suffices to show that   
\begin{align}
\label{comple-zeta-second-order}
\omega \cdot \mathcal{A}(y^\ast, t, \mathrm{i}\xi, p)=0 
~\text{  mod }  M^{+}(y^\ast, \xi, p) 
~~ \Longrightarrow ~~ \omega=(\omega_{1}, \ldots, \omega_{n(q-1)})=0. 
\end{align} 
Note that the algebraic equation \eqref{comple-zeta-second-order} 
can be written as 
\begin{align*} 
(p+ \sigma^2 \xi^{2})^{n(q-1)-1} \omega_{j} =0 
~ \text{ mod }  M^{+}(y^\ast, \xi, p), 
~~ \forall\, j \in \{1, \ldots, n(q-1)\}.  
\end{align*}  
This is equivalent to 
\begin{align}\label{eq:xi-zeta}
    \left(\xi-\zeta_{2}(y^{\ast}, p)\right)^{n(q-1)-1} \omega_j =0 ~~~~ \text{ mod }  \xi-\zeta_{1}(y^{\ast}, p), 
~~ \forall\, j \in \{1, \ldots, n(q-1)\},
\end{align} 
since the polynomial $\left(\xi-\zeta_{2}(y^{\ast}, p)\right)^{n(q-1)-1}$ is not divisible by the term, $\xi-\zeta_{1}(y^{\ast}, p)$. 
By substituting $\xi=\zeta_{1}(y^{\ast}, p)$ into both sides of \eqref{eq:xi-zeta}, 
we obtain 
\begin{align*}
\omega_j = 0, 
~~~~\forall\, j  \in \{1, \ldots, n(q-1) \}, 
\end{align*}
which establishes \eqref{comple-zeta-second-order}.

\smallskip 

{\bf $\bullet$ The complementary conditions of the initial datum.} 
(\cite[Page 12]{Solonnikov65} or Definition \ref{def:complementary-condition}.) 

\smallskip 

Since the matrix associated with the initial conditions has the form 
\begin{align*}
\mathcal{C}_0(y,\partial_{y},\partial_{t})=Id_{n(q-1) \times n(q-1)},
\end{align*} 
it follows from the definitions of $\mathcal{C}_0(y, 0, p)$ and $\mathcal{\hat{L}}_0(y,t,\mathrm{i}\xi,p)$ 
that the matrix 
\begin{align*}
\mathcal{D}(y,p)=\mathcal{C}_0(y, 0, p) \mathcal{\hat{L}}_0 (y, 0, 0, p)
= \text{diag}\big(\underbrace{p^{n(q-1)-1}, \ldots, p^{n(q-1)-1}}_{n(q-1)  \text{ elements} } \big)  
\end{align*} 
has rows that are linearly independent modulo $p^{n(q-1)}$. 
Therefore, the initial datum $h_0$ satisfies the complementary conditions.

\smallskip 

{\bf $\bullet$ The smoothness of coefficients.} 

\smallskip 

By applying Lemmas \ref{lem:RemarkB1} and \ref{lem:tech2}, and noting that $\frac{\alpha_1}{k}>\alpha_2$, 
it follows from the assumption $\bar{h}_l \in C^{\frac{2+\alpha_2}{2}, 2+\alpha_2}([0,T] \times [0,1])$ for all $l$ that the coefficients in the linear parabolic system \eqref{eq:second-order-linear-h_l} satisfy 
\begin{equation*}
W_1(\partial_y \bar{h}_{l}, \bar{h}_{l})
\in C^{\frac{1+\alpha_2}{2}, 1+\alpha_2}([0,T] \times [0,1]) \subset C^{\frac{\alpha_2}{2}, \alpha_2}([0,T] \times [0,1]), 
~~~~\forall\, l \in \{1, \ldots, q-1\}
\text{.}
\end{equation*} 
By applying Lemma \ref{lem:RemarkB1} to the smooth solution $g_l$
obtained in Lemma \ref{lem:linear-system-g-l}, we conclude that the boundary data fulfill 
\begin{align*}
    g_l(\cdot, y_l^{\ast}) \in C^{\frac{2k+\alpha_1}{2k}}([0,T]) \subset C^{\frac{2+\alpha_2}{2}}([0,T]), ~~~~~ y_l^{\ast}=1-2\bigg\{ \frac{l+1}{2} \bigg\}, l=1, \ldots, q-1.
\end{align*}
On the other hand, the assumption on the initial data ensures that  $h_{l,0} \in C^{2+\alpha_2}([0,1])$, $\forall\, l\in \{1, \ldots, q-1\}$.

\smallskip 

Since we have verified all the required sufficient conditions stated in Assumption \ref{assumption:complementary-condition} and 
Theorem \ref{Solonnikov-theorem}, 
the results including the estimate \eqref{estimate-for-h_l} follow directly by applying \eqref{cauchy-inequality}, and Theorem \ref{Solonnikov-theorem} with the specific choice of data:

Initial datum: 
\begin{align*}
\phi=\big(\phi_1, \ldots, \phi_{n(q-1)}  \big)=&\big( h_{1,0}, h_{2,0}, \ldots, h_{q-1,0} \big),  
\end{align*} 

Boundary datum: 
$\Phi^{g}$, as given by \eqref{definition-Phi-bar-g}, 

Inhomogeneous lower-order term: 
\begin{align*}
f=\bigg(  W_1(\partial_y \bar{h}_{1}, \bar{h}_{1}),  \ldots, W_{q-1}(\partial_y \bar{h}_{q-1},\bar{h}_{q-1}) \bigg).
\end{align*} 
\end{proof}

We are now ready to derive the following theorem. 

\begin{teo}[Solutions to the coupled linear parabolic systems, \eqref{eq:higher-order-linear-g_l}$\sim$\eqref{eq:second-order-linear-h_l}]
\label{thm:STE_for_linear-fitting} 
Let $\lambda\in [0,\infty)$, $\sigma \in (0,\infty)$, $\alpha_j\in(0,1)$, $j \in \{1, 2\}$ such that $\alpha_1>k\alpha_2$. 
Suppose that $(g_0,h_0)$ is the initial datum of the coupled linear parabolic systems, \eqref{eq:higher-order-linear-g_l}$\sim$\eqref{eq:second-order-linear-h_l}, 
where $g_0=(g_{1,0}, \ldots, g_{q,0})$,   
$h_0=(h_{1,0}, \ldots, h_{q-1,0})$, and  
$g_{l,0}$ and $h_{l,0}: [0,1] \to\mathbb{R}^n$ possess the regularity $g_{l,0}\in C^{2k+\alpha_1}([0,1])$ and  
$h_{l,0} \in C^{2+\alpha_2}([0,1])$, respectively, for all $l$. 
Moreover, assume that $(g_0, h_0)$ fulfills the compatibility conditions of order zero, as defined in Definition \ref{def:cc-order-p-fitting}, for the nonlinear parabolic system, \eqref{AP-q-even-fitting}$\sim$\eqref{BC-q-even-fitting}.   
Then, for any $T\in(0,1)$ and $(\bar{g}, \bar{h}) \in X^{T}_{(g_0,h_0)}$, there exists a unique solution $(g, h)\in X^{T}_{(g_0,h_0)}$ 
to the linear parabolic systems, \eqref{eq:higher-order-linear-g_l}$\sim$\eqref{eq:second-order-linear-h_l}, 
such that 
$g_{l}\in C^{\frac{2k+\alpha_1}{2k}, 2k+\alpha_1}([0,T]\times [0,1])$, 
$h_{l} \in C^{\frac{2+\alpha_2}{2}, 2+\alpha_2}([0,T]\times [0,1])$, 
$\forall\, l$, and  
\begin{align}\label{estimate:(g,h)}
\nonumber
&\norm{(g, h)}_{ X^{T}_{(g_0, h_0)} } 
\leq  C_0\cdot \bigg{(} \sum^{q}_{l=1}
\norm{G_l(\partial_y^{2k-1} \bar{g}_{l},
\ldots, \bar{g}_{l})}_{C^{\frac{\alpha_1}{2k},\alpha_1}([0,T] \times [0,1])} 
\\ \nonumber
&~~~~~~~~~~~~~~~~~~~~~~~~~~~~~   
+\sigma^2 \sum^{q-1}_{l=1}\norm{ W_1(\partial_y \bar{h}_l, \bar{h}_l)}_{C^{\frac{\alpha_2}{2},\alpha_2}([0,T] \times [0,1])}+\sum^{q}_{l=0}| p_{l}|
\\ \nonumber
&~~~~~~~~~~~~~~~~~~~~~~ 
+ \frac{1}{\sigma^2} 
\sum_{\substack{l=1, \cdots, q-1 \\  y^{\ast}_l =1-2 \big\{ \frac{l+1}{2} \big\} }} 
\|\partial_y \bar{h}_l (\cdot, y^{\ast}_l)\|_{C^{\frac{1+\alpha_1}{2k}}([0,T])}   +\sum\limits_{\mu=1}^{k-1} \left( |b^{\mu}_{1,0}|+ |b^{\mu}_{q, 2 \{ \frac{q}{2} \}}|   \right)   
\\
&~~~~~~~~~~~~~~~~~~~~~~~~~~~~ 
+\sum^{q}_{l=1}\Vert g_{l,0} \Vert_{C^{2k+\alpha_1}([0,1])}+\sum^{q-1}_{l=1}\Vert h_{l,0} \Vert_{C^{2+\alpha_2}([0,1])}\bigg{ )},
\end{align} 
where $C_0>0$ is independent of the non-homogeneous terms and the initial-boundary datum in the coupled linear parabolic system, 
\eqref{eq:higher-order-linear-g_l}$\sim$\eqref{eq:second-order-linear-h_l}. 
\end{teo} 

\begin{proof} 
The existence of a solution $(g, h)\in X^{T}_{(g_0,h_0)}$  to the coupled linear parabolic systems, \eqref{eq:higher-order-linear-g_l}$\sim$\eqref{eq:second-order-linear-h_l} follows from Lemmas \ref{lem:linear-system-g-l} and \ref{lem:linear-system-h-l}. 
Moreover,  
\begin{align*}
\nonumber
&\norm{(g, h)}_{ X^{T}_{(g_0, h_0)} } 
\leq  \bar{C}_0\cdot \bigg{(} \sum^{q}_{l=1}
\norm{G_l(\partial_y^{2k-1} \bar{g}_{l},
\ldots, \bar{g}_{l})}_{C^{\frac{\alpha_1}{2k},\alpha_1}([0,T] \times [0,1])} 
\\  \nonumber
&~~~~~~~~ 
+\sigma^2 \sum^{q-1}_{l=1}\norm{ W_1(\partial_y \bar{h}_l, \bar{h}_l)
 }_{C^{\frac{\alpha_2}{2},\alpha_2}([0,T] \times [0,1])}+\sum^{q}_{l=0}| p_{l}|
\\  \nonumber
&~~~~~~~~
+ \sum_{\substack{l=1, \cdots, q-1 \\ y^{\ast}_l =1-2 \big\{ \frac{l+1}{2} \big\} }}
\|g_l(\cdot, y^{\ast}_l)\|_{C^{\frac{2+\alpha_2}{2}}([0,T])} 
+\frac{1}{\sigma^2} 
\sum_{\substack{l=1, \cdots, q-1 \\  y^{\ast}_l =1-2 \big\{ \frac{l+1}{2} \big\} }} 
\|\partial_y \bar{h}_l (\cdot, y^{\ast}_l)\|_{C^{\frac{1+\alpha_1}{2k}}([0,T])}      
\\
&~~~~~~~~ +\sum\limits_{\mu=1}^{k-1} \left( |b^{\mu}_{1,0}|+ |b^{\mu}_{q, 2 \{ \frac{q}{2} \}}|   \right)+\sum^{q}_{l=1}\Vert g_{l,0} \Vert_{C^{2k+\alpha_1}([0,1])}+\sum^{q-1}_{l=1}\Vert h_{l,0} \Vert_{C^{2+\alpha_2}([0,1])}\bigg{ )}.
\end{align*} 
From the assumption,  
$\alpha_2<\frac{\alpha_1}{k}$, $T\in(0,1)$, 
and the definition of parabolic H\"{o}lder space, we have 
\begin{align*}
\nonumber
& \| g_l (\cdot, y^{\ast}_l) \|_{C^{\frac{2+\alpha_2}{2} }([0,T])} 
\\  \nonumber
=& \| g_l (\cdot, y^{\ast}_l) \|_{C^{0}([0,T])}+\| \partial_t g_l (\cdot, y^{\ast}_l) \|_{C^{0}([0,T])}+[\partial_t g_l(\cdot, y^{\ast}_l) ]_{\frac{\alpha_2}{2}, t } 
\\ 
\leq &  \| g_l (\cdot, y^{\ast}_l) \|_{C^{0}([0,T])}+\| \partial_t g_l (\cdot, y^{\ast}_l) \|_{C^{0}([0,T])}+[\partial_t g_l(\cdot, y^{\ast}_l) ]_{\frac{\alpha_1}{2k}, t } T^{\frac{\alpha_1}{2k}-\frac{\alpha_2}{2}} 
\\
\leq & \norm{g_{l}}_{C^{\frac{2k+\alpha_1}{2k},2k+\alpha_1}([0,T] \times [0,1])}.
\end{align*} 
The proof of \eqref{estimate:(g,h)} follows directly from the two inequalities above by setting $C_0=\max \{2 \bar{C}_0^2, \bar{C}_0 \}$. 
\end{proof}

\begin{rem} 
To keep the proof of Theorem \ref{thm:STE_for_linear-fitting} concise, we establish the existence of solutions to the coupled linear parabolic system, \eqref{eq:higher-order-linear-g_l}$\sim$\eqref{eq:second-order-linear-h_l} only for $T \in (0,1)$, even though $T$ may be taken arbitrarily in $(0,\infty)$ 
in Lemmas \ref{lem:linear-system-g-l} and \ref{lem:linear-system-h-l}.
\end{rem}

\subsection{Local Solutions of Theorem \ref{thm:Main_Thm-fitting}} 
\label{Sec:STE}

To establish the existence of local solutions in Theorem \ref{thm:Main_Thm-fitting}, we apply the Banach Fixed Point Theorem, making use of the a priori estimates provided in Theorem \ref{thm:STE_for_linear-fitting} in \S\ref{Sec:linear}. 
It is important to note that the regularity of solutions up to the initial time at the boundary points is closely tied to the so-called geometric compatibility conditions imposed on the initial data. These conditions play a crucial role in ensuring well-posedness and are discussed in detail below. 

\begin{teo}[The short-time existence in Theorem \ref{thm:Main_Thm-fitting}] 
\label{thm:existence-RS-even-fitting} 
Let $\lambda\in [0,\infty)$, $\sigma \in (0,\infty)$ and 
$\varrho_j\in (0,\infty)\setminus \mathbb{N}$, $\forall\, j\in \{1, 2\}$, with $\ell:=[\varrho_1]=[\varrho_2]\in\mathbb{N}_0$ and fractional parts satisfying $\{\varrho_1\}>k\{\varrho_2\}$. 
Assume that 
$(g_0,h_0)$ is the initial datum of the nonlinear parabolic system, 
\eqref{AP-q-even-fitting}$\sim$\eqref{BC-q-even-fitting}. 
Moreover, suppose that the initial datum, 
$g_{l,0}\in C^{2k+\varrho_1} \left([0,1]\right)$ and 
$h_{l,0}\in C^{2+\varrho_2}([0,1])$, 
satisfy the compatibility conditions of order $\ell$,  
as defined in Definition \ref{def:cc-order-p-fitting}, for all $l$. 
Then, the following statements hold. 
\begin{itemize}
\item[(i)] 
There exists a positive number 
$t_0=t_0\left(n, \{\varrho_1\}, \{\varrho_2\}, \lambda, \sigma, g_0, h_0, M\right)>0$ 
such that $(g, h)\in X^{t_0}_{(g_0,h_0)}$ 
is the unique solution to \eqref{AP-q-even-fitting} under the initial-boundary conditions specified in \eqref{BC-q-even-fitting}.

\item[(ii)]  
The solution $(g,h)$ fulfills the regularity 
\begin{equation*} 
\begin{cases}
g_l  \in C^{\frac{2k+\varrho_1}{2k},2k+\varrho_1}\left([0,t_0] \times [0,1]\right)\bigcap 
C^{\infty}\left((0,t_0]\times [0,1]\right), 
~~~~~~~~~ l\in\{1, \ldots, q\}, 
\\
h_l  \in C^{\frac{2+\varrho_2}{2},2+\varrho_2}\left([0,t_0] \times [0,1]\right)\bigcap 
C^{\infty}\left((0,t_0]\times [0,1]\right), 
~~~~~~ l\in\{1, \ldots, q-1\}.  
\end{cases}
\end{equation*} 

\item[(iii)] 
Moreover, 
the induced map $\gamma: [x_0, x_q]\rightarrow M \subset \mathbb{R}^n$ 
defined by \eqref{eq:parametrization-gamma} 
satisfies  
$\gamma(t,\cdot) \in C^{2k-2}([x_0, x_q]), ~ \forall\, t \in [0,t_0]$. 
\end{itemize} 
\end{teo}

\begin{rem} 
The dependence of the constant $t_0$ in Theorem \ref{thm:existence-RS-even-fitting} can be expressed more explicitly as follows: 
\begin{align*}
t_0=t_0 \bigg(n, \{\varrho_1\}, \{\varrho_2\}, \lambda, \sigma, 
\sum\limits^{q}_{l=1} \Vert g_{l,0}\Vert_{C^{2k+\{\varrho_1\}}([0,1])}, 
\sum\limits^{q-1}_{l=1} \Vert h_{l,0}\Vert_{C^{2+\{\varrho_2\}}([0,1])}, 
K_1, 
\|R_M\|_{C^{1}(M)}, 
C(M)
\bigg)
\end{align*}
where 
\begin{align*} 
K_1 &= 2 C_0 \cdot \bigg{(} \sum^{q}_{l=1} 
\Vert G_l(\partial_y^{2k-1}g_{l,0}, \ldots, g_{l,0})  \Vert_{C^{\{\varrho_1\}}([0,1])}+\sigma^2\sum^{q-1}_{l=1} 
\Vert  W_1(\partial_y h_{l,0}, h_{l,0}) \Vert_{C^{\{\varrho_2\}}([0,1])}
\\
& ~~~~~~~~~~~~~~~ 
+\sum\limits_{l=0}^{q} |p_{l}|
+\sum\limits_{\mu=1}^{k-1}(|b^{\mu}_{1,0}|+|b^{\mu}_{q,2\{\frac{q}{2}\}}|)
+\sum\limits^{q}_{l=1}\Vert g_{l,0} \Vert_{C^{2k+ \{\varrho_1\}} }([0,1]) \\
&~~~~~~~~~~~~~~~~+\big(  1+\frac{1}{\sigma^2} \big) \sum\limits^{q-1}_{l=1}\Vert h_{l,0} \Vert_{C^{2+\{\varrho_2\}}([0,1])} \bigg). 
\end{align*} 
Note that $R_{M}$ represents the Riemannian curvature tensor, as defined in \eqref{def:R_M}, the constant $C(M)$ is a universal bound for the 
$C^\infty$-norm of the derivatives of the first and second fundamental forms of $M$, as specified in \eqref{eq:C_infty-bdd_for_I+II}, 
and $C_0$ is a positive constant, which is independent of the initial-boundary data given in \eqref{BC-q-even-fitting}. 
\end{rem}

\begin{rem}
\label{rem:STE_1_universal_constant} 
In the argument for the short-time existence, the constants involved, including 
$C_0$, are independent of $m$, the dimension of the Riemannian manifold $M$. This independence stems from the fact that $M$ can be isometrically embedded into $\mathbb{R}^n$ with $m<n$, allowing the use of algebraic identities and triangle inequalities in the analysis. However, for the  argument for the long-time existence, the situation differs. The constants in the estimates that lead to Gr\"{o}nwall inequalities depend on $m$ because this formulation is intrinsic to $M$. Consequently, the isometric embedding of $M$ into $\mathbb{R}^n$ with $m<n$ does not influence these constants in the argument for the long-time existence.  
\end{rem}

\begin{proof} 

To keep notation clean, the H\"{o}lder exponents $\{\varrho_1\}$ and $\{\varrho_2\}$ will be denoted by $\alpha_1$ and $\alpha_2$, respectively, throughout the rest of this article, where $\alpha_1=\{\varrho_1 \}$ and $\alpha_2=\{\varrho_2 \}$.  

\smallskip 

The proof is structured as follows: 

\smallskip 

{\bf Existence of Local Solutions}: In Part (i), we apply Solonnikov's theory in conjunction with the Banach Fixed Point Theorem to establish the existence of local solutions to \eqref{AP-q-even-fitting} within the space $X_{(g_0,h_0)}^{t_0}$ for some $t_0>0$,  
subject to the initial-boundary conditions specified in \eqref{BC-q-even-fitting}.

{\bf Higher-Order Smoothness}: In Part (ii), we employ a bootstrapping argument to demonstrate the higher-order smoothness of $(g,h)$, expressed as: 
\begin{equation*} 
\begin{cases}
g_l \in C^{\frac{2k+\varrho_1}{2k},2k+\varrho_1} \left([0,t_0] \times [0,1]\right)\bigcap 
C^{\infty}\left((0,t_0]\times [0,1]\right), 
~~~~~~~~~~~~~~~ l\in\{1, \ldots, q\}, 
\\
h_l \in C^{\frac{2+\varrho_2}{2},2+\varrho_2} \left([0,t_0] \times [0,1]\right)\bigcap 
C^{\infty}\left((0,t_0]\times [0,1]\right), ~~~~~~~~~~~~ l\in\{1, \ldots, q-1\}.
\end{cases}
\end{equation*} 
 
{\bf Regularity of $\gamma$}: In Part (iii), we conclude the global regularity properties of $\gamma$. We show that $\gamma$ maintains regularity throughout the entire space-time domain.


\begin{itemize}

\item[\bf (i).] {\bf Existence of Local Solutions}: 
The existence of solution to \eqref{AP-q-even-fitting} in the space $X^{t_0}_{(g_0,h_0)}$ with the initial-boundary conditions \eqref{BC-q-even-fitting}.

\bigskip 

{\bf $\bullet$ Self-maps, i.e., $\exists$ $t_1 >0$ so that 
$\mathcal{G}\left( X^{T}_{(g_0,h_0)} \cap B_{K_1}\right) \subset X^{T}_{(g_0,h_0)} \cap B_{K_1}$, 
$\forall\, T\in(0, t_1)$.}

By applying Theorem \ref{thm:STE_for_linear-fitting} to the linear parabolic systems, \eqref{eq:higher-order-linear-g_l} and \eqref{eq:second-order-linear-h_l}, we introduce the operator
\begin{align*} 
\nonumber
\mathcal{G}:  X^{T}_{(g_0,h_0)} \cap B_K 
&\to 
X^{T}_{(g_0,h_0)} \cap B_K 
\\
(\bar{g}, \bar{h})
&\mapsto  (g,h)
\text{,}
\end{align*}
where $(g,h)$ is a solution to \eqref{eq:higher-order-linear-g_l} and \eqref{eq:second-order-linear-h_l}  
for any sufficiently small $T>0$. 
In the subsequent analysis, we will demonstrate that 
$\mathcal{G}$ acts as a self-map on $B_{K_1}$ by selecting an appropriate $K_1>0$. 
By applying the triangle inequalities in the 
H\"{o}lder spaces, \eqref{eq:C_infty-bdd_for_I+II}, 
\eqref{est:W_i-2nd-fund-form}, 
and the technical lemmas, Lemmas \ref{lem:RemarkB1}$\sim$\ref{lem:T^beta-norm0}, with $v=\bar{g}_l-g_{l,0}$ or $v=\bar{h}_l-h_{l,0}$ therein as appropriate, we obtain 
\begin{align}
\label{estimate-Riemmanian-curv-tensor} 
\nonumber
\| R_{M}( \bar{g}_l) - R_{M}(g_{l,0}) \|_{ C^{\frac{\alpha_1}{2k}, \alpha_1} ([0,T]\times [0,1])} \leq &  \|R_M\|_{C^{1}(M)}  \| \bar{g}_l-g_{l,0}\|_{ C^{\frac{\alpha_1}{2k}, \alpha_1} ([0,T] \times [0,1])} \\ 
\leq & C  \left(n, K_1, \|R_{M}\|_{C^{1}(M)} \right) \cdot T^{\beta_1} 
\text{,}
\end{align} 
and 
\begin{align}
\label{estimate-W_i} 
\nonumber 
&\|W_i(\partial_y^{i}\bar{g}_l, \ldots, \bar{g}_l)- W_i(\partial_y^{i}g_{l,0}, \ldots, g_{l,0})\|_{ C^{\frac{\alpha_1}{2k}, \alpha_1} ([0,T] \times [0,1])} 
\\ \nonumber
&\leq C(n, K_1) \cdot 
\|W_i \|_{C_{\text{loc}}^{1}(\mathbb{R}^n \times \cdots \times \mathbb{R}^n \times M) } \cdot \sum_{j \leq i} \| \partial_y^{j}\bar{g}_l-\partial_y^{j}g_{l,0}\|_{ C^{\frac{\alpha_1}{2k}, \alpha_1} ([0,T] \times [0,1])} 
\\
& \leq  C(n, K_1, C(M))\cdot 
T^{\beta_1}, 
\end{align} 
for any $i \in\{1, \cdots, 2k-1\}$, and 
$\beta_1=\min \big{ \{ } \frac{\alpha_1}{2k},\frac{1-\alpha_1}{2k} \big{ \} } $. 
From \eqref{estimate-Riemmanian-curv-tensor}, \eqref{estimate-W_i}, \eqref{eq:C_infty-bdd_for_I+II}, \eqref{est:W_i-2nd-fund-form}, 
and \eqref{eq:G_l-F_l-gamma+chi}, 
one can apply the triangle inequalities in the 
H\"{o}lder spaces, the algebraic identity
\begin{align}
\label{identity-a^n-b^n}
a^p-b^p=(a-b)(a^{p-1}+a^{p-2} b+\cdots+a b^{p-2}+b^{p-1}), 
~~~~~~ p \in \{2,3,\cdots\},  
\end{align}
and the technical lemmas, 
Lemmas \ref{lem:RemarkB1}$\sim$\ref{lem:T^beta-norm0}, 
to obtain 
\begin{align*} 
& \sum\limits^{q}_{l=1}\Vert G_l(\partial_y^{2k-1} \bar{g}_{l}, \ldots, \bar{g}_{l})
-G_l(\partial_y^{2k-1} g_{l,0}, \ldots, g_{l,0}) 
\Vert_{ C^{\frac{\alpha_1}{2k}, \alpha_1}([0,T] \times [0,1])} 
\\  \nonumber 
\leq & 
C\left(g_0 , K_1, C(M) 
\right) \sum\limits^{q}_{l=1} \| R_{M}( \bar{g}_l) - R_{M}(g_{l,0}) \|_{ C^{\frac{\alpha_1}{2k}, \alpha_1} ([0,T]\times [0,1])} 
\\  \nonumber
&~ +C \left(g_0, K_1,  \|R_M\|_{C^{1}(M)}  \right) \cdot 
\\  \nonumber
&~~~~~~~~~~~~~~~~ \sum\limits_{i \leq 2k-1}\sum\limits^{q}_{l=1}  \|W_i(\partial_y^{i}\bar{g}_l, \ldots, \bar{g}_l)- W_i(\partial_y^{i}g_{l,0}, \ldots, g_{l,0})\|_{ C^{\frac{\alpha_1}{2k}, \alpha_1} ([0,T] \times [0,1])}  
\\  \nonumber
&~~ + C \left(\lambda, g_0, K_1, C(M), \|R_M\|_{C^{0}(M)}\right) 
 \sum\limits_{j \leq 2k-1} \sum\limits_{l=1}^{q} \| \partial_y^{j}\bar{g}_l - \partial_y^{j}g_{l,0}\|_{ C^{\frac{\alpha_1}{2k}, \alpha_1}([0,T] \times [0,1])} 
\\ 
\leq & C\left(n, \lambda, g_0, K_1, M \right) \cdot T^{\beta_1} 
\text{. }
\end{align*} 
Note that we have used the simplified notation $C \left(g_0 , K_1,  M \right)$ 
to indicate that the constant $C$ depends on 
$\sum\limits^{q}_{l=1}\Vert g_{l,0} \Vert_{ C^{2k+\alpha_1} }([0,1])$, $\|R_M\|_{C^{1}(M)}$, and $C(M)$. 
Throughout the following analysis, we adopt this same notational convention to streamline the presentation of constants. 

By applying the triangle inequalities in the 
H\"{o}lder spaces, \eqref{est:W_i-2nd-fund-form}, 
and the technical lemmas, Lemmas \ref{lem:RemarkB1}$\sim$\ref{lem:T^beta-norm0}, with $v=\bar{g}_l-g_{l,0}$ or $v=\bar{h}_l-h_{l,0}$ therein, 
we can derive the estimate:    
\begin{align*}
\nonumber
&\|W_1(\partial_y\bar{h}_l, \bar{h}_l)- W_1(\partial_y h_{l,0}, h_{l,0})\|_{ C^{\frac{\alpha_2}{2}, \alpha_2} ([0,T] \times [0,1])} 
\\ \nonumber
&~~~~~ \leq C(n, K_1) \cdot 
\|W_1 \|_{C_{\text{loc}}^{1}(\mathbb{R}^n \times \cdots \times \mathbb{R}^n \times M) } \cdot \sum_{j \leq 1} \| \partial_y^j \bar{h}_l-\partial_y^j h_{l,0}\|_{ C^{\frac{\alpha_2}{2}, \alpha_2} ([0,T] \times [0,1])} 
\\ 
& ~~~~~
\leq  C(n, K_1, C(M)) \cdot T^{\beta_2} 
\text{,}
\end{align*}
where 
$\beta_2=\min \big{ \{ } \frac{\alpha_2}{2}, \frac{1-\alpha_2}{2} \big{ \} }$. 

By applying the inequality \eqref{interpolation-inequality-x-1} and the estimate \eqref{eq:normapp1} in Lemma \ref{lem:T^beta-norm0} with parameters $\ell=0$, $\alpha'=\frac{1+\alpha_1}{k}$, and $\alpha=\alpha_2$, we obtain 
\begin{align*}
\| \partial_y\bar{h}_l (\cdot, y^{\ast}_l)- \partial_y h_{l,0} (y^{\ast}_l) \|_{C^{\frac{1+\alpha_1}{2k} }([0,T])} 
&\leq \| \partial_y \bar{h}_l - \partial_y h_{l,0} \|_{C^{\frac{1+\alpha_2}{2}, 1+\alpha_2 }([0,T] \times [0,1])} 
\\ 
&\leq 2CK_1 T^{\beta_2}, 
\end{align*}
where $C>0$ is a constant.

Drawing on the inequality established in Theorem \ref{thm:STE_for_linear-fitting} from the theory of linear parabolic systems, and applying the triangle inequality in the H\"{o}lder spaces, we obtain the following estimate: 
\begin{align*} 
&\norm{(g,h)}_{X^{T}_{(g_0,h_0)}} 
\\  \nonumber 
&\leq  C_{0}\cdot \sum\limits^{q}_{l=1} 
\Vert G_l(\partial_y^{2k-1} \bar{g}_{l}, \ldots, \bar{g}_{l})
-G_l(\partial_y^{2k-1} g_{l,0}, \ldots, g_{l,0}) 
\Vert_{C^{\frac{\alpha_1}{2k}, \alpha_1}([0,T] \times [0,1])} 
\\ \nonumber 
& ~~+C_{0}\cdot \sigma^2 \sum\limits^{q-1}_{l=1} 
\Vert W_1(\partial_y \bar{h}_{l}, \bar{h}_{l})
-W_1(\partial_y h_{l,0}, h_{l,0}) 
\Vert_{C^{\frac{\alpha_2}{2}, \alpha_2} ([0,T] \times [0,1]) } 
\\  \nonumber
&~~ +C_0 \cdot 
\sum_{\substack{l=1, \cdots, q-1 \\ y^{\ast}_l =1-2 \big\{ \frac{l+1}{2} \big\} }}  
\frac{1}{\sigma^2} \cdot \| \partial_y\bar{h}_l (\cdot, y^{\ast}_l)- \partial_y h_{l,0} (y^{\ast}_l) \|_{C^{\frac{1+\alpha_1}{2k} }([0,T])} 
\\
&~~+C_0 \cdot  
\bigg( 
\sum^{q}_{l=1}\norm{  G_l(\partial_y^{2k-1} g_{l,0}, \ldots, g_{l,0})}_{ C^{\alpha_1} ([0,1]) } 
 + \sigma^2 \sum^{q-1}_{l=1}\norm{ W_1(\partial_y h_{l,0}, h_{l,0})}_{C^{\alpha_2}( [0,1])} 
\\ \nonumber   
& 
~~~~~~~~~~~~~~+ \sum\limits_{l=0}^{q} |p_{l}| 
 +\sum\limits_{\mu=1}^{k-1} \left( |b^{\mu}_{1,0}|+ |b^{\mu}_{q, 2 \{ \frac{q}{2} \}}|   \right)
+\sum^{q}_{l=1}
\Vert g_{l,0}\Vert_{C^{2k+\alpha_1}([0,1])} 
\\
&~~~~~~~~~~~~~~+ \big(1+\frac{1}{\sigma^2} \big) \sum^{q-1}_{l=1}\Vert h_{l,0} \Vert_{ C^{2+\alpha_2} ([0,1])} \bigg) 
\\  \nonumber 
&\leq  C_1 T^{\beta} 
+C_0 \cdot  
\bigg( 
\sum^{q}_{l=1}\norm{  G_l(\partial_y^{2k-1} g_{l,0}, \ldots, g_{l,0})}_{C^{\alpha_1}([0,1])} 
\\ \nonumber   
& 
~~~~~~~~~~~~~
+\sigma^2 \sum^{q-1}_{l=1}\norm{W_1(\partial_y h_{l,0}, h_{l,0})}_{C^{\alpha_2}([0,1])}
+ \sum\limits_{l=0}^{q} |p_{l}| 
 +\sum\limits_{\mu=1}^{k-1} \left( |b^{\mu}_{1,0}|
 + |b^{\mu}_{q, 2 \{ \frac{q}{2} \}}| \right)  
\\
&~~~~~~~~~~~~~~~~~~~~~~~~~~~~~~~~
+\sum^{q}_{l=1}\Vert g_{l,0} \Vert_{C^{2k+\alpha_1}([0,1])} 
+ \big(1+\frac{1}{\sigma^2} \big) \sum^{q-1}_{l=1}\Vert h_{l,0} \Vert_{C^{2+\alpha_2}([0,1])} \bigg), 
\end{align*} 
where $T\in(0,1)$ is sufficiently small, 
$C_1 =C_1 \left(\lambda, n, \sigma, g_0, h_0, K_1, M \right)$, 
and $\beta:=\min\{\beta_1, \beta_2\}$.  
By letting 
$K_1 \in (0, \infty)$ fulfilling 
\begin{align*} 
\frac{K_1}{2}=& C_0 \cdot \bigg{(} \sum^{q}_{l=1} 
\Vert G_l(\partial_y^{2k-1}g_{l,0}, \ldots, g_{l,0})  \Vert_{ C^{\alpha_1} ([0,1])}
+ \sigma^2 \sum^{q-1}_{l=1} 
\Vert  W_1(\partial_y h_{l,0}, h_{l,0}) \Vert_{C^{\alpha_2}([0,1])}
\\  
&~~~~~~~~~~~+ \sum\limits_{l=0}^{q} |p_{l}|
+\sum\limits_{\mu=1}^{k-1}(|b^{\mu}_{1,0}|+|b^{\mu}_{q,2\{\frac{q}{2}\}}|)
+\sum^{q}_{l=1}\Vert g_{l,0} \Vert_{ C^{2k+\alpha_1}([0,1])} 
\\
& ~~~~~~~~~+\big(1+\frac{1}{\sigma^2} \big) \sum^{q-1}_{l=1}\Vert h_{l,0} \Vert_{C^{2+\alpha_2}([0,1])} \bigg) 
\end{align*}
and by choosing $t_1>0$ fulfilling 
\begin{align*}
C_1 t_1^{\beta} \leq \frac{K_1}{2} , 
\end{align*} 
we conclude that $\norm{(g,h)}_{X^{t_1}_{(g_0,h_0)}}\leq K_1$. 
Thus, 
$$\mathcal{G}\left(X^{T}_{(g_0,h_0)} 
\cap B_{K_1}\right) \subset X^{T}_{(g_0,h_0)} \cap B_{K_1}, 
~~~~~~\forall\, T \in (0,t_1]. $$

\bigskip


{\bf $\bullet$ Contraction maps.} 

We now prove that the mapping
$$
\mathcal{G}: X^{T}_{(g_0,h_0)} \cap B_{K_1} \to X^{T}_{(g_0,h_0)} \cap B_{K_1}
$$ 
is a contraction map for sufficiently small $T>0$. Specifically, we show that there exists 
$t_0 \in (0,t_1]$ such that for all $T\in(0,t_0]$ 
and all 
$$
(\bar{g}, \bar{h}), (\bar{e}, \bar{f})\in X^{T}_{(g_0,h_0)} \cap B_{K_1}, 
$$
the following estimate holds: 
\begin{align}
\label{est:contraction-map-gamma+chi}
\norm{(g,h)-(e,f)}_{X^{T}_{(g_0,h_0)}} 
\leq C_{2} \, T^{\beta} \,  \norm{(\bar{g}, \bar{h})-(\bar{e}, \bar{f})}_{X^{T}_{g_0}}
<\norm{(\bar{g}, \bar{h})-(\bar{e}, \bar{f})}_{X^{T}_{g_0}},  
\end{align} 
where 
$(g,h)=\mathcal{G} \big( (\bar{g}, \bar{h}) \big)$, $(e,f)=\mathcal{G} \big( (\bar{e}, \bar{f}) \big)$, and the exponent $\beta=\min\{\beta_1, \beta_2\}$. 
The constant $C_2 = C_2 \left( \lambda, n, \sigma, g_0, h_0, 
K_1, \|R_M\|_{C^{1}(M)}, C(M) \right)$ is positive and depends only on the specified data.

Observe from 
\eqref{eq:co-deri}, \eqref{eq:co-deri-1}, \eqref{BC-q-even-fitting}, 
\eqref{eq:higher-order-linear-g_l}, and \eqref{eq:second-order-linear-h_l} that the difference $g-e$ fulfills 
the linear parabolic system:  
\begin{equation}
\label{eq:linear-g-e}
\begin{cases} 
\partial_t ( g_{l}-e_{l})+(-1)^{k} \partial_y^{2k}( g_{l}-e_{l})=
G_l(\partial_y^{2k-1} \bar{g}_{l}, \ldots, \bar{g}_{l})-G_l(\partial_y^{2k-1} \bar{e}_{l}, \ldots, \bar{e}_{l}), 
\\
~~~~~~~~~~~~~~~~~~~~~~~~~~~~~~~~~~~~~~~~~~~~~~~~~~~ 
\text{ in } (0,T) \times (0,1), ~ l \in\{1, \ldots, q\}, 
\\
(g_{l}-e_{l})(0,y)= 0, 
~~~~~~~~~~~~~~~~~~~~~~~~~~~~~~~~~~~~~~~~~~~~~~~~~~~~ 
l \in \{1, \ldots, q\}, 
\\
(g_{1}-e_{1})(t, 0)=0, 
~~~ 
(g_{q}-e_{q}) \big{( } t, 2\big{ \{ } \frac{q}{2} \big{ \} } \big{ ) }=0, 
\\
\partial_{y}^{\mu}(g_{1}-e_{1})(t, 0)=0, 
~~ \partial_{y}^{\mu}(g_{q}-e_{q}) \big{( } t, 2\big{ \{ } \frac{q}{2} \big{ \} } \big{ ) }=0, 
~~~~~~ 
\mu \in \{1, \ldots, k-1\}, 
\\
(g_{l}-e_{l})(t, y^{\ast}_{l})= (g_{l+1}-e_{l+1})(t, y^{\ast}_{l}),   
~ l \in \{1, \ldots, q-1\}, 
y^{\ast}_{l}= 1- 2\{\frac{l+1}{2}\}, 
\\
\partial_{y}^{\mu}\left(g_{l}-e_{l}\right) (t,y^{\ast}_{l})
+(-1)^{\mu-1}\partial_{y}^{\mu} \left(g_{l+1}-e_{l+1}\right)(t,y^{\ast}_{l})=0,  
\\
~~~~~~~~~~~~~~~~~~~~~
\mu \in \{1, \ldots, 2k-2\}, 
~ l \in \{1, \ldots, q-1\}, ~
y^{\ast}_{l}= 1- 2\{\frac{l+1}{2}\},  
\\ 
\partial_{y}^{2k-1}(g_{l}-e_{l})(t,y^{\ast}_l)+\partial_{y}^{2k-1}( g_{l+1}-e_{l+1})(t,y^{\ast}_{l})= \frac{(-1)^{k}}{\sigma^2}  \partial_{y} (\bar{h}_{l}-\bar{f}_{l})(t,y^{\ast}_{l}), 
\\
~~~~~~~~~~~~~~~~~~~~~~~~~~~~~~~~~~~~~~~~~~~~~~~~~ 
l \in \{1, \ldots, q-1\}, 
y^{\ast}_{l}= 1- 2\{\frac{l+1}{2}\},  
\end{cases}
\end{equation} 
and the difference $h-f$ fulfills the linear parabolic system:   
\begin{equation}
\label{eq:linear-h-f}
\begin{cases} 
\partial_t ( h_{l}-f_{l})- \sigma^2  \partial_y^{2}( h_{l}-f_{l})
= \sigma^2 W_1(\partial_y \bar{h}_l, \bar{h}_l)
- \sigma^2 W_1(\partial_y \bar{f}_l, \bar{f}_l),  
\\
~~~~~~~~~~~~~~~~~~~~~~~~~~~~~~~~~~~~~~~~~~~~ 
\text{ in } (0,T) \times (0,1), l \in\{1, \ldots, q-1\}, 
\\
(h_{l}-f_{l})(0,y)= 0,  
~~~~~~~~~~~~~~~~~~~~~~~~~~~~~~~~~~~~~~~~~~~ 
l \in \{1, \ldots, q-1\}, 
\\
(h_{l}-f_{l})(t, y^{\ast}_l)= 0, 
~~~~~~~~~~~~~~~~~~~~~ 
l \in \{1, \ldots, q-1 \}, 
~ y^{\ast}_l=1-2 \big{ \{ } \frac{l}{2} \big{ \} }, 
\\
(h_{l}-f_{l})(t, y^{\ast}_l)=(g_{l}-e_{l})(t, y^{\ast}_l),  
~~ 
l \in \{1, \ldots, q-1\}, ~
y^{\ast}_l= 1- 2\{\frac{l+1}{2}\},
\end{cases}
\end{equation} 
where $t \in [0,T]$ and $y \in [0,1]$.

Given that 
$(g,h)=\mathcal{G}\big( (\bar{g}, \bar{h}) \big)$ and $(e,f)=\mathcal{G}\big( (\bar{e}, \bar{f}) \big) $, 
with the initial data satisfying 
$\bar{g}_{l}(0,\cdot)=g_{l,0}(\cdot)$, $\bar{e}_{l}(0,\cdot)=g_{l,0}(\cdot)$, $\bar{h}_{l}(0,\cdot)=h_{l,0}(\cdot)$, $\bar{f}_{l}(0,\cdot)=h_{l,0}(\cdot)$, 
and assuming these data fulfill the compatibility conditions of order zero for \eqref{eq:higher-order-linear-g_l} and \eqref{eq:second-order-linear-h_l},
it follows that the initial data of $g-e$ and $h-f$ are identically zero and also satisfy the compatibility conditions of order zero for the linear parabolic systems \eqref{eq:linear-g-e} and \eqref{eq:linear-h-f}, respectively.
Therefore, by applying the theory of linear parabolic systems as stated in Theorem \ref{thm:STE_for_linear-fitting}, we conclude that 
$g-e$ is the unique solution to \eqref{eq:linear-g-e} and 
$h-f$ is the unique solution to \eqref{eq:linear-h-f}. Moreover, we obtain the following estimate: 
\begin{align}
\label{est:(g,h)-(e,f)}
&\norm{(g,h)-(e,f)}_{X^{T}_{(g_0,h_0)}} 
\\ \nonumber 
&\leq C_0 \cdot \left(\sum\limits^{q}_{l=1} 
\norm{ G_l(\partial_y^{2k-1} \bar{g}_{l}, \ldots, \bar{g}_{l})-G_l(\partial_y^{2k-1} \bar{e}_{l}, \ldots, \bar{e}_{l})}_{C^{\frac{\alpha}{2k}, \alpha}([0,T] \times [0,1])}
\right) 
\\ \nonumber 
& ~~~~~~~~~~~~~~~ 
+C_0 \cdot \left( \sigma^2 \sum\limits^{q}_{l=1} 
\norm{ W_1(\partial_y \bar{h}_l, \bar{h}_l)-W_1(\partial_y \bar{f}_l, \bar{f}_l) }_{C^{\frac{\alpha}{2}, \alpha}([0,T] \times [0,1])}
\right) 
\\ \nonumber
&~~~~~~~~~~~~~~+\frac{C_0 }{\sigma^2}  \sum_{\substack{l=1, \cdots, q-1 \\ y^{\ast}_l =1-2 \big\{ \frac{l+1}{2} \big\} }} 
\|\partial_y \bar{h}_l(\cdot, y^{\ast}_l) -\partial_y\bar{f}_l(\cdot, y^{\ast}_l)\|_{C^{\frac{1+\alpha_1}{2k}}([0,T])} 
\text{.}
\end{align} 
Thus, by applying the algebraic identity 
\eqref{identity-a^n-b^n}, the triangle inequality in the H\"{o}lder spaces, and leveraging \eqref{eq:C_infty-bdd_for_I+II} and \eqref{est:W_i-2nd-fund-form} along with the technical lemmas, 
Lemmas \ref{lem:RemarkB1}$\sim$\ref{lem:T^beta-norm0}, where we set $v=\bar{g}_l-\bar{e}_l$ or $v=\Bar{h}_l-\Bar{f}_l$ therein 
as appropriate, we can establish the following desired estimates: 
\begin{align*}
\nonumber
& \| R_{M}( \bar{g}_l) - R_{M} (\bar{e}_l) \|_{ C^{\frac{\alpha_1}{2k}, \alpha_1} ([0,T]\times [0,1])} 
\\
& ~~~~~~~~~~~ \leq  
C \left(n, K_1,  \|R_M\|_{C^{1}(M)} \right) 
\cdot  T^{\beta_1} 
\cdot \| \bar{g}_l-\bar{e}_l \|_{ C^{\frac{2k+\alpha_1}{2k}, 2k+\alpha_1}([0,T]\times [0,1])}, 
\\ \nonumber 
& \|W_i(\partial_y^{i}\bar{g}_l, \ldots, \bar{g}_l)- W_i(\partial_y^{i}\bar{e}_l, \ldots, \bar{e}_l)\|_{ C^{\frac{\alpha_1}{2k}, \alpha_1}([0,T] \times [0,1])} 
\\ \nonumber 
& ~~~~~~~~~~~~ \leq  
C\left(n, K_1, C(M) \right) \cdot T^{\beta_1} \cdot \| \bar{g}_l-\bar{e}_l \|_{ C^{\frac{2k+\alpha_1}{2k}, 2k+\alpha_1}([0,T]\times [0,1])}
\text{.} 
\end{align*} 
By applying 
\eqref{eq:F_l}, \eqref{eq:C_infty-bdd_for_I+II}, \eqref{est:W_i-2nd-fund-form}, 
and the technical lemmas, Lemmas \ref{lem:RemarkB1}$\sim$\ref{lem:T^beta-norm0}, we obtain 
\begin{align} 
\label{eq:G-estimate-3}
&\norm{ G_l(\partial_y^{2k-1} \bar{g}_{l}, \ldots, \bar{g}_{l})-G_l(\partial_y^{2k-1} \bar{e}_{l}, \ldots, \bar{e}_{l})}_{  C^{\frac{\alpha_1}{2k}, \alpha_1}([0,T] \times [0,1])} 
\\ \nonumber
& \leq C \left( K_1, C(M)  
\right) \sum\limits^{q}_{l=1} \| R_{M}( \bar{g}_l) - R_{M}(\bar{e}_{l}) \|_{ C^{\frac{\alpha_1}{2k}, \alpha_1}([0,T]\times [0,1])} 
\\ \nonumber
& ~~  +C \left( K_1, \|R_M\|_{C^{1}(M)}\right) \cdot 
\\ \nonumber 
& ~~~~~~~~~~~~~~~~~~~~~ 
\sum_{\substack{l\in\{1, \ldots, q\} \\ i\in\{0, \ldots, 2k-1\} }}
\|W_i(\partial_y^{i}\bar{g}_l, \ldots, \bar{g}_l)- W_i(\partial_y^{i}\bar{e}_{l}, \ldots, \bar{e}_{l})\|_{ C^{\frac{\alpha_1}{2k}, \alpha_1}([0,T] \times [0,1])} 
\\ \nonumber
& ~~ +C \left(\lambda,\|R_M\|_{C^{0}(M)}, C(M) \right) 
\sum_{\substack{l\in\{1, \ldots, q\} \\ j\in\{0, \ldots, 2k-1\} }} 
\|\partial_y^{j}\bar{g}_l-\partial_y^{j}\bar{e}_{l}\|_{C^{\frac{\alpha_1}{2k}, \alpha_1}([0,T] \times [0,1])} 
\\ \nonumber
& \leq  C\left(n, \lambda, K_1, \|R_M\|_{C^{1}(M)}, C(M) 
\right) 
\cdot 
\|\bar{g}_l- \bar{e}_l\|_{ C^{\frac{2k+\alpha_1}{2k}, 2k+\alpha_1}([0,T] \times [0,1])} 
\cdot T^{\beta_1} 
\text{.}
\end{align} 
Similarly, for the estimates of terms involving $W_1(\partial_y \bar{h}_l, \bar{h}_l)- W_1(\partial_y \bar{f}_l, \bar{f}_l)$, 
we establish the estimate:  
\begin{align}
\label{est:W_1-(h-f)}
\nonumber
&\|W_1(\partial_y \bar{h}_l, \bar{h}_l)- W_1(\partial_y \bar{f}_l, \bar{f}_l)\|_{ C^{\frac{\alpha_2}{2}, \alpha_2} ([0,T] \times [0,1])} 
\\
& ~~~~~~~~~~~~~~~~~~~~ \leq 
C\left(n, K_1, C(M) 
\right) \cdot T^{\beta_2}  \cdot \| \bar{h}_l-\bar{f}_l \|_{  C^{\frac{2+\alpha_2}{2}, 2+\alpha_2}([0,T]\times [0,1])}
\text{.} 
\end{align} 
By applying the inequalities \eqref{interpolation-inequality-x-1} and \eqref{eq:normapp1} in Lemma \ref{lem:T^beta-norm0} with the choices, $\ell=0$, $\alpha'=\frac{1+\alpha_1}{k}$, and $\alpha=\alpha_2$, we obtain 
\begin{align*}
\nonumber
\|\partial_y \bar{h}_l(\cdot, y^{\ast}_l) - \partial_y\bar{f}_l(\cdot, y^{\ast}_l)\|_{C^{\frac{1+\alpha_1}{2k}}([0,T])} \leq &\|\partial_y \bar{h}_l - \partial_y \bar{f}_l \|_{C^{\frac{2+\alpha_2}{2}, 2+\alpha_2}([0,T] \times [0,1])} \\
\leq & C T^{\beta_2} \|\bar{h}_l -\bar{f}_l \|_{C^{\frac{2+\alpha_2}{2}, 2+\alpha_2}([0,T] \times [0,1])},
\end{align*}
where $C>0$ is a constant. 
Together with \eqref{est:(g,h)-(e,f)}, \eqref{eq:G-estimate-3}, and \eqref{est:W_1-(h-f)}, we may choose a sufficiently small   
$t_0 = t_0 \left(\lambda, n, \sigma, \alpha_1, \alpha_2, g_0, h_0, K_1, M \right)$ 
such that $t_0\in (0, t_1)$ and the term 
$C_2 T^{\beta}$ in \eqref{est:contraction-map-gamma+chi} fulfills 
$$ 
C_2 T^{\beta}<1,\, \forall\,  T\in(0, t_0]. 
$$  
With the existence of such $t_0>0$, 
the operator 
$$
\mathcal{G} :X^{T}_{(g_0,h_0)} \cap B_{K_1} \to X^{T}_{(g_0,h_0)} \cap B_{K_1}
$$ 
not only represents a self-map but also serves as a strict contraction map for any $T\in(0,t_0]$.

By applying the Banach Fixed Point Theorem to the set $X^{t_{0}}_{(g_0,h_0)} \cap B_{K_1}$, we conclude that there exists a \emph{unique} fixed point
$(g,h) \in X^{t_{0}}_{(g_0,h_0)} \cap B_{K_1}$, which is also a solution to
\eqref{AP-q-even-fitting} with the initial-boundary conditions
\eqref{BC-q-even-fitting}.

\item[\bf (ii).] {\bf Higher-Order Smoothness}: 
The higher-order smoothness of local solutions to \eqref{AP-q-even-fitting} with the initial-boundary conditions \eqref{BC-q-even-fitting}. 

{\bf $\bullet$ Smoothing.} 

From Part (i), we have established the existence of solutions to 
\eqref{AP-q-even-fitting} along with the initial-boundary conditions specified in \eqref{BC-q-even-fitting}. With these solutions in place, we can interpret the nonlinear parabolic system formally as a linear parabolic system, and then apply Solonnikov's theory (Theorem \ref{Solonnikov-theorem}) to deduce higher regularity, as the following. 

Denote by $c_l=c_l(t, y)=G_l(\partial_y^{2k-1} g_{l}, \ldots, \partial_y g_{l}, g_{l})$, $l\in\{1, \ldots, q\}$, 
and $d_l=d_l(t, y)=\sigma^2 W_1(\partial_y h_l, \partial_y h_l)$, 
$l \in\{1, \ldots, q-1\}$. 
Then, we rewrite the nonlinear parabolic system, \eqref{AP-q-even-fitting} and \eqref{BC-q-even-fitting} as: 
\begin{equation}
\label{higher-regularity-g-l}
\begin{cases}
\partial_t g_{l}+(-1)^{k}\cdot \partial_y^{2k} g_{l}=c_l, 
~~~~~~~~~~~~~~~~~~~~~~~~~~~~~~~~~~~~~~~~~~ l \in\{1, \ldots, q\}, 
\\ 
g_{l}(0,y)= g_{l,0}(y), 
~~~~~~~~~~~~~~~~~~~~~~~~~~~~~~~~~~~~~~~~~~~~~~~~~~~~ 
l\in \{1, \ldots, q\}, 
\\
g_{1}(t,0)= p_{0}, 
~~~~~~~~~ 
g_{q}(t, 2\{\frac{q}{2}\})=p_{q},  
\\ 
\partial_{y}^{\mu}g_{1}(t,0)= 
b^{\mu}_{1, x_0},   
~~ 
\partial_{y}^{\mu}g_{q}(t, 2\{\frac{q}{2}\})= 
(-1)^{(q-1)\cdot\mu} b^{\mu}_{q, x_q}, 
\mu \in \{1, \ldots, k-1\},  
\\ 
g_{l}(t, y^{\ast}_l)  
=g_{l+1}(t, y^{\ast}_{l}), 
~~~~~~~~~~~~~~~~~~ 
l \in \{1, \ldots, q-1\}, ~
y^{\ast}_l= 1- 2\{\frac{l+1}{2}\}, 
\\ 
\partial_{y}^{\mu-1}g_{l}(t,y^{\ast}_l)
+ (-1)^{\mu-1} \partial_{y}^{\mu} g_{l+1}(t,y^{\ast}_{l})=0, 
\\ 
~~~~~~~~~~~~~~~~~~~
\mu \in \{1, \ldots, 2k-2\}, 
~ l \in \{1, \ldots, q-1\}, 
~ y^{\ast}_l= 1- 2\{\frac{l+1}{2}\},  
\\ 
\partial_{y}^{2k-1} g_{l}(t,y^{\ast}_l) 
+\partial_{y}^{2k-1} g_{l+1}(t,y^{\ast}_{l})=\frac{(-1)^{k}}{\sigma^2}\partial_{y}h_{l}(t,y^{\ast}_l), 
\\ 
~~~~~~~~~~~~~~~~~~~~~~~~~~~~~~~~~~~~~~~~~~~~~~ 
l \in \{1, \ldots, q-1\}, ~
y^{\ast}_l= 1- 2\{\frac{l+1}{2}\},  
\end{cases} 
\end{equation} 
and 
\begin{equation}
\label{higher-regularity-h-l}
\begin{cases}
\partial_t h_{l}-\sigma^2 \partial_y^{2} h_{l}=d_l,  
~~~~~~~~~~~~~~~~~~~~~~~~~~~~~~~~~~~~~~~~~ 
l \in\{1, \ldots, q-1\}, 
\\
h_{l}(0,y)= h_{l,0}(y), 
~~~~~~~~~~~~~~~~~~~~~~~~~~~~~~~~~~~~~~~~~~~~ 
l\in \{1, \ldots, q-1\}, 
\\
h_{l}(t, y^{\ast}_l)
=g_{l}(t, y^{\ast}_l) 
~~~~~~~~~~~~~~~~~~~~ 
l \in \{1, \ldots, q-1\}, ~
y^{\ast}_l= 1- 2\{\frac{l+1}{2}\}, 
\end{cases} 
\end{equation} 
where $t \in [0,T]$, $y \in [0,1]$. 
According to Part (i), 
we have established the existence of solutions to \eqref{higher-regularity-g-l} and \eqref{higher-regularity-h-l}, denoted by $(g,h)$, possessing the regularity: 
$g_{l} \in C^{\frac{2k+\alpha_1}{2k}, 2k+\alpha_1 }\left( [0,t_0] \times [0,1] \right)$,\\
$h_{l} \in C^{\frac{2+\alpha_2}{2}, 2+\alpha_2}\left( [0,t_0] \times [0,1] \right)$, $\forall\,l$. 
Given this, we can infer that: 
$$
c_l \in C^{\frac{1+\alpha_1}{2k}, 1+\alpha_1}\left( [0,t_0] \times [0,1] \right), 
~~~ 
d_l \in C^{\frac{1+\alpha_2}{2}, 1+\alpha_2}\left( [0,t_0] \times [0,1] \right), ~~~ \forall\, l.
$$

Considering \eqref{higher-regularity-g-l} and \eqref{higher-regularity-h-l} as a linear parabolic system, it satisfies the parabolicity conditions, compatibility conditions, and complementary boundary conditions. Since the prerequisites for Theorem \ref{Solonnikov-theorem} are met, we apply this theorem to deduce that the solutions to \eqref{higher-regularity-g-l} and \eqref{higher-regularity-h-l} exhibit enhanced regularity: 
$$
g_{l} \in C^{\frac{2k+1+\alpha_1}{2k}, 2k+1+\alpha_1}\left( [0,t_0] \times [0,1] \right), 
~~ 
h_{l} \in C^{\frac{2+1+\alpha_2}{2}, 2+1+\alpha_2}\left( [0,t_0] \times [0,1] \right), 
~~~~ \forall\, l. 
$$
Note that when $(g, h) \in X^{t_0}_{(g_0,h_0)}$ possesses the regularity: 
$$
g_l \in C^{\frac{2k+\varrho_1}{2k},2k+\varrho_1}\left([0,t_0] \times [0,1]\right), 
~~~~ 
h_l\in C^{\frac{2+\varrho_2}{2},2+\varrho_2}\left([0,t_0] \times [0,1]\right), 
~~~~ \forall\, l, 
$$ 
we have: 
\begin{align*}
\begin{cases}
\partial_y^j g_{l}(0,y^{\ast})=\partial_y^j g_{l,0}(y^{\ast}), 
~~~~~\forall\, 0 \leq j \leq 2k+\ell, 
\\ 
\partial_y^j h_{l}(0,y^{\ast})=\partial_y^j h_{l,0}(y^{\ast}), 
~~~~~~ \forall\, 0 \leq j \leq 2 +\ell. 
\end{cases}
\end{align*} 
Given the regularity of the coefficients in the linear parabolic system 
\eqref{higher-regularity-g-l} and \eqref{higher-regularity-h-l}, 
and assuming that the compatibility conditions of order $\ell$ for \eqref{eq:higher-order-linear-g_l} and \eqref{eq:second-order-linear-h_l} are satisfied by the initial data $(g_0, h_0)$, we apply a bootstrapping argument to improve the regularity of the solution. 
This iterative process, repeated for 
\eqref{higher-regularity-g-l} and \eqref{higher-regularity-h-l}, 
successively increases the regularity of $c_l$ and $d_l$. After $\ell$ iterations, we ultimately achieve: 
\begin{align*}
\begin{cases}
g_{l} \in C^{\frac{2k+\varrho_1}{2k}, 2k+\varrho_1}\left( [0,t_0] \times [0,1] \right), ~~~~~~\forall\, l\in\{1,\cdots,q\},  
\\ 
h_{l} \in C^{\frac{2+\varrho_2}{2}, 2+\varrho_2 }\left( [0,t_0] \times [0,1] \right), ~~~ \forall\, l\in\{1,\cdots,q-1\} .
\end{cases}
\end{align*}

The approach to proving $g_{l} \in C^{\infty}\left( (0,t_0] \times [0,1] \right)$, $h_{l} \in C^{\infty}\left( (0,t_0] \times [0,1] \right)$, $\forall\, l$, 
closely parallels the method outlined in the previous section for smoothing solutions. This involves constructing smooth cut-off functions with null initial data, verifying the compatibility conditions of higher-order linear parabolic systems, and applying bootstrapping techniques to achieve enhanced regularity. Specifically, the regularity of $g$ and $h$ can be elevated by showing that: 
\begin{equation*}
\begin{cases}
g_{l} \in \bigcap\limits_{\ell=1}^{\infty}C^{\frac{2k+\ell+\alpha_1}{2k},2k+\ell+\alpha_1}\left(\left[ \frac{2\ell-1}{2\ell}\varepsilon, t_0\right] \times [0,1]\right),\\
h_{l} \in \bigcap\limits_{\ell=1}^{\infty}C^{\frac{2+\ell+\alpha_2}{2},2+\ell+\alpha_2}\left(\left[ \frac{2\ell-1}{2\ell}\varepsilon, t_0\right] \times [0,1]\right),
\end{cases}
\end{equation*} 
for arbitrary $\varepsilon \in (0,t_0)$. 
This process has already been carried out in detail in the previous section, so we omit the repetition here.

\item[\bf (iii).]{\bf Regularity of $\gamma$}: 
From \eqref{eq:parametrization-chi}, we derive that  
\begin{align*}
\gamma_{l}(t,x)=&g_{l}\left(t, (-1)^{l+1} 
\left(x-x_{2\left[\frac{l}{2}\right]} \right)\right) , 
~~~~~~~~~~~~~~~ x_{l-1} \leq x \leq x_{l}, ~~ l \in \{1, \ldots, q\},
\\
\chi_{l}(t,x)=&h_{l}\left(t, (-1)^{l+1} 
\left(x-x_{2\left[\frac{l}{2}\right]} \right)\right) , 
~~~~~~~~~~ x_{l-1} \leq x \leq x_{l}, ~~ l \in \{1, \ldots, q-1\}.
\end{align*} 
The proof simply follows from applying (i). 

\end{itemize}

\end{proof}


\subsection{Global Solutions and Proof of Theorem \ref{thm:Main_Thm-fitting} } 
\label{Sec:LTE}

In this subsection, we establish the long-time existence of solutions and the asymptotics by deriving uniform bounds for curvature integrals of any order and using a contradiction argument. Based on the uniform bounds, we further analyze the asymptotic behaviour of solutions. The proof of the existence of global solutions builds upon the results developed in \S\ref{Sec:NPSN}, \S\ref{Sec:linear}, \S\ref{Sec:STE}, and a contradiction argument. 
There are mainly two steps in the rest of this subsection. 
(1) Uniform Bounds for Curvature Integrals and Contradiction Argument: 
We employ the Gagliardo–Nirenberg interpolation inequalities together with Gr\"{o}nwall’s inequality to derive uniform bounds on the higher-order derivatives of $\gamma_{l}$ and $\chi_{l}$, valid for all $t$ and any fixed $l$. These estimates enable us to carry out a contradiction argument, thereby concluding the proof of long-time existence; 
(2) Asymptotic Limits: 
Based on the uniform bounds, we proceed to analyze the asymptotic behaviour of the solutions. In particular, we establish the existence of asymptotic limits of $\gamma(t_j,\cdot)$ along certain convergent subsequences.

\bigskip

{\bf The uniform bounds for curvature integrals of any order.} 
Based on the short-time existence result in Theorem \ref{thm:existence-RS-even-fitting}, it remains to prove 
the long-time existence of classical solutions in Theorem \ref{thm:Main_Thm-fitting}, which is based on a contradiction argument and uniform bounds on curvature integrals for any higher order. 
Namely, assume on the contrary that any solution to 
\eqref{eq:E_s-flow-fitting}$\sim$\eqref{eq:BC-higher-order-5-fitting} only exists up to $t_{max}\in (0, \infty)$, i.e., 
either $\gamma_{l}(t_{max},\cdot)$ or $\chi_{l}(t_{max},\cdot)$ fails to be smooth for the first time at $t=t_{max}$ for some $l$. 
Then, we show that 
\begin{align}
\label{eq:Z_mu}
\mathcal{Z}_{\mu}^{\sigma}(t)
:=\sum\limits_{l=1}^{q}\| D_t^{\mu-1} \partial_{t} \gamma_{l}(t,\cdot)\|^2_{L^2(I_{l})} 
+ \frac{1}{\sigma^4} 
\sum\limits_{l=1}^{q-1}\| D_t^{\mu-1} \partial_{t} \chi_{l}(t,\cdot)\|^2_{L^2(I_{l})} 
\end{align} 
remains uniformly bounded up to $t=t_{max}$ for any $\mu \in \mathbb{N}$. 
This gives the contradiction.

From differentiating the first term on the right-hand side of \eqref{eq:Z_mu} with respect to $\partial_t$ and applying Lemmas \ref{lem:D_t-D_x-gamma} and 
\ref{D_x^iD_t^j-f-fitting}, we obtain 
\begin{align*} 
\frac{d}{dt}& \frac{1}{2}\sum\limits_{l=1}^{q}\| D_t^{\mu-1} \partial_{t} \gamma_{l}\|^2_{L^2(I_{i})}
=J_1+ J_2+\sum_{l=1}^{q}\int_{I_{l}} Q^{4k\mu+2k-4, 2k\mu+k-1}_{4}(\gamma_{l,x}) \text{ }dx, 
\end{align*}
where $Q^{a,c}_b(\gamma_{l,x})$ defined in \S\ref{subsection-Notation-terminology} represents a polynomial of $\gamma_{l,x}$, 
\begin{equation} 
\label{definition-I_1-I_2}
\begin{cases}
J_1=(-1)^{k+1} \sum\limits_{l=1}^{q}\int\limits_{I_{l}}\big{\langle} D_t^{\mu}  D_{x}^{2k-1}\gamma_{l, x}, D_t^{\mu-1} \partial_{t} \gamma_{l}\big{\rangle}\text{ }dx,\\
J_2= \sum\limits_{l=1}^{q} \sum\limits^{k}_{r=2}  (-1)^{r+k+1} \int\limits_{I_{l}}\big{\langle} 
   R\left(D_t^{\mu}D^{2k-r-1}_{x}\gamma_{l, x} , D^{r-2}_{x}\gamma_{l, x}\right) \gamma_{l, x}, D_t^{\mu-1}\partial_{t} \gamma_{l} \big{\rangle}\text{ }dx, 
\end{cases}
\end{equation} 
are terms involving the highest-order derivatives. 
We will apply integration by parts to these terms so that  Gr\"{o}nwall inequalities can be derived using interpolation inequalities.  
By letting $V=D_{x}^{2k-2}\gamma_{l,x}$ in \eqref{eq:D_t^j-D_x->D_x-D_t^j} of Lemma \ref{lem:D_t-D_x-gamma}, 
and applying Lemma \ref{D_x^iD_t^j-f-fitting}, along with integration by parts to the term  
$J_1$ in \eqref{definition-I_1-I_2}, 
we obtain     
\begin{align}\label{I_1-first}
J_1=& (-1)^{k+1}\sum_{l=1}^{q}\big{\langle} D_t^{\mu}  D_{x}^{2k-2}\gamma_{l, x}, D_t^{\mu-1} \partial_{t} \gamma_{l} \big{\rangle} \bigg{|}^{x_{l}}_{x_{l-1}} \\
-& (-1)^{k+1}\left(\sum_{l=1}^{q}\int_{I_{l}} \big{\langle} D_t^{\mu}  D_{x}^{2k-2}\gamma_{l, x} , D_{x} D_t^{\mu-1} \partial_{t} \gamma_{l} \big{\rangle}\text{ }dx  
-\sum_{l=1}^{q} \int_{I_{l}} Q^{4k\mu+2k-4, 2k\mu-1 }_{4}(\gamma_{l,x}) \text{ }dx \right).
\nonumber
\end{align} 
Similarly, by letting $V=D_{x}^{2k-3}\gamma_{l, x}$ in \eqref{eq:D_t^j-D_x->D_x-D_t^j} of Lemma \ref{lem:D_t-D_x-gamma}, together with applying  
Lemma \ref{D_x^iD_t^j-f-fitting} and using integration by parts, 
to the terms $\sum\limits_{l=1}^{q}\int\limits_{I_{l}} \big{\langle} D_t^{\mu}  D_{x}^{2k-2}\gamma_{l, x} , D_{x} D_t^{\mu-1} \partial_{t} \gamma_{l} \big{\rangle}\text{ }dx$
in \eqref{I_1-first}, 
we obtain 
\begin{align*} 
&  
\sum_{l=1}^{q}\int_{I_{l}} \big{\langle} D_t^{\mu}  D_{x}^{2k-2}\gamma_{l, x}, D_{x}D_t^{\mu-1} \partial_{t} \gamma_{l} \big{\rangle}\text{ }dx 
\\ 
= & -\sum_{l=1}^{q}\int_{I_{l}} 
\big{\langle} D_t^{\mu}  D_{x}^{2k-3}\gamma_{l, x}, D_{x}^2 D_t^{\mu-1} \partial_{t} \gamma_{l} \big{\rangle}\text{ }dx 
+\sum_{l=1}^{q} \int_{I_{l}} Q^{ 4k\mu+2k-4, 2k\mu }_{4}(\gamma_{l,x}) \text{ }dx 
\text{,}
\end{align*} 
where the boundary terms vanish by applying Lemma \ref{lem:bdry_vanishing}, 
i.e.,  
\begin{align*}
\big{\langle} D_t^{\mu}[\Delta_{l} D_{x}^{2k-3} \partial_{x} \gamma](t), D_{x}D_t^{\mu-1}\partial_{t} \gamma_{l}(t,x_{l}^{-}) \big{\rangle} =0, 
~~~~~~ \forall\, l=1,\cdots, q-1. 
\end{align*} 
Hence,   
\begin{align*} 
J_1 
&=(-1)^{k+1} \sum_{l=1}^{q}\int_{I_{l}} \big{\langle} D_t^{\mu}  D_{x}^{2k-3}\gamma_{l, x}, D_{x}^2 D_t^{\mu-1} \partial_{t} \gamma_{l} \big{\rangle}\text{ }dx 
+\sum_{l=1}^{q} \int_{I_{l}} Q^{ 4k\mu+2k-4, 2k\mu }_{4}(\gamma_{l,x}) \text{ }dx \\
&+(-1)^{k+1}\sum_{l=1}^{q}\big{\langle} D_t^{\mu}  D_{x}^{2k-2}\gamma_{l, x}, D_t^{\mu-1} \partial_{t} \gamma_{l} \big{\rangle} \bigg{|}^{x_{l}}_{x_{l-1}}.
\end{align*} 
By the same procedure as above, 
i.e., by applying 
Lemmas \ref{D_x^iD_t^j-f-fitting}, \ref{lem:bdry_vanishing}, 
and integration by parts to the main terms in $J_1$ 
inductively, we finally derive 
\begin{align*} 
\nonumber
J_1=& -\sum_{l=1}^{q}\int_{I_{l}} \big{\langle} D_t^{\mu}  D_{x}^{k-1}\gamma_{l, x}, D_{x}^k D_t^{\mu-1} \partial_{t} \gamma_{l} \big{\rangle} \text{ }dx 
+\sum_{l=1}^{q} \int_{I_{l}} Q^{4k\mu+2k-4, 2k\mu+k-1}_{4}(\gamma_{l,x}) \text{ }dx \\ \nonumber
&+(-1)^{k+1}\sum_{l=1}^{q}\big{\langle} D_t^{\mu}  D_{x}^{2k-2}\gamma_{l, x}, D_t^{\mu-1} \partial_{t} \gamma_{l} \big{\rangle} \bigg{|}^{x_{l}}_{x_{l-1}} \\ \nonumber
=&-\sum_{l=1}^{q}\int_{I_{l}} | D_{x}^{2k\mu+k-1}\gamma_{l, x} |^2 \text{ }dx 
+\sum_{l=1}^{q} \int_{I_{l}} Q^{4k\mu+2k-4, 2k\mu+k-1}_{4}(\gamma_{l,x}) \text{ }dx \\ 
&+(-1)^{k+1}\sum_{l=1}^{q}\big{\langle} D_t^{\mu}  D_{x}^{2k-2}\gamma_{l, x}, D_t^{\mu-1} \partial_{t} \gamma_{l} \big{\rangle} \bigg{|}^{x_{l}}_{x_{l-1}}.
\end{align*} 
By utilizing the symmetry of the Riemannian curvature tensors, applying integration by parts, and inductively using Lemmas \ref{D_x^iD_t^j-f-fitting} and  \ref{lem:bdry_vanishing}, we can express $J_2$ as
\begin{align*} 
J_2=\sum_{l=1}^{q} \int_{I_{l}} Q^{4k\mu+2k-4, 2k\mu+k-1}_{4}(\gamma_{l,x}) \text{ }dx,
\end{align*} 
which are lower-order terms.

Next, differentiating one-half of the second term on the right-hand side of \eqref{eq:Z_mu} with respect to $\partial_t$, and applying integration by parts along with the identities in Lemmas \ref{lem:D_t-D_x-gamma} and \ref{D_x^iD_t^j-f-fitting}, 
we obtain   
\begin{align*} 
J_{3}
:=& \frac{d}{dt} 
\frac{1}{2\sigma^4} \sum\limits_{l=1}^{q-1}\| D_t^{\mu-1} \partial_{t} \chi_{l}(t,\cdot)\|^2_{L^2(I_{l})}
=\frac{1}{\sigma^2}\sum_{l=1}^{q-1}\int_{I_{l}}\langle D_t^{\mu} D_x \chi_{l, x}, D_t^{\mu-1} \partial_{t} \chi_{l} \rangle \text{ }dx 
\\ \nonumber 
=&\frac{1}{\sigma^2}\sum_{l=1}^{q-1} \langle D_t^{\mu}  \chi_{l, x},  D_t^{\mu-1} \partial_{t} \chi_{l} \rangle  \bigg{|}^{x_{l}}_{x_{l-1}}-\frac{1}{\sigma^2}\sum_{l=1}^{q-1}\int_{I_{l}}\langle D_t^{\mu}  \chi_{l, x}, D_x D_t^{\mu-1} \partial_{t} \chi_{l} \rangle \text{ }dx 
\\ \nonumber
+& \sigma^{4\mu-2} \sum_{l=1}^{q-1}\int_{I_{l}} Q^{4\mu-2, 2\mu-1}_4(\chi_{l,x})\text{} dx 
\\ \nonumber
=&\frac{1}{\sigma^2}\sum_{l=1}^{q-1} \langle D_t^{\mu}  \chi_{l, x},  D_t^{\mu-1} \partial_{t} \chi_{l} \rangle  \bigg{|}^{x_{l}}_{x_{l-1}}
- \sigma^{4\mu-2} \sum_{l=1}^{q-1}\int_{I_{l}} | D_{x}^{2\mu}  \chi_{l, x}|^2 \text{ }dx 
\\ \nonumber 
&~~~~~~~~~~~~~~~~~~~~~~~~~~~~~~~~~~~~~~~~~~~~~~~~~~~~~~~~~~~~ 
+ \sigma^{4\mu-2} \sum_{l=1}^{q-1}\int_{I_{l}} Q^{4\mu-2, 2\mu-1}_4(\chi_{l,x})\text{} dx.
\end{align*} 
Note that the boundary terms in $J_1+J_3$ can be eliminated by applying the boundary conditions \eqref{eq:BC-higher-order-1-fitting}, \eqref{eq:BC-2-order-1-fitting}, \eqref{eq:BC-higher-order-2-fitting}, and \eqref{eq:BC-higher-order-5-fitting}.
Therefore, we obtain the following result: 
\begin{align*} 
\nonumber
J_1+J_3
=&- \sum_{l=1}^{q}\int_{I_{l}} | D_{x}^{2k\mu+k-1}\gamma_{l, x} |^2 \text{ }dx 
- \sigma^{4\mu-2} \sum_{l=1}^{q-1}\int_{I_{l}} | D_{x}^{2\mu}  \chi_{l, x}|^2 \text{ }dx
 \\
&+ \sum_{l=1}^{q} \int_{I_{l}} Q^{4k\mu+2k-4, 2k\mu-1 }_{4}(\gamma_{l,x}) \text{ }dx 
+ \sigma^{4\mu-2} \sum_{l=1}^{q-1}\int_{I_{l}} Q^{4\mu-2, 2\mu-1}_4(\chi_{l,x})\text{} dx. 
\end{align*}
From these identities, we derive 
\begin{align}
\label{derivative-X_mu(t)}
\frac{d}{dt}\mathcal{Z}_{\mu}^{\sigma}(t)
+&\mathcal{Z}_{\mu}^{\sigma}(t)
+2\sum_{l=1}^{q}\int_{I_{l}}|D_{x}^{2k\mu+k-1}\gamma_{l,x}|^2 \text{ }dx
+ 2 \sigma^{4\mu-2} \sum_{l=1}^{q-1}\int_{I_{l}} | D_{x}^{2\mu}\chi_{l, x} |^2 \text{ }dx 
\\
=&\sum_{l=1}^{q}\int_{I_{l}} Q^{4k\mu+2k-4, 2k\mu+k-1}_{4}(\gamma_{l,x}) \text{ }dx 
+ \sigma^{4\mu-2} \sum_{l=1}^{q-1}\int_{I_{l}} Q^{4\mu-2, 2\mu-1}_{4} (\chi_{l,x}) \text{ }dx,
\nonumber 
\end{align} 
for all $t\in (0,t_{max})$. 
Note the short-time existence of classical solutions, established in the previous subsection, ensures the continuity of 
$\mathcal{E}_{k}^{\lambda, \sigma}[(\gamma(t,\cdot), ( \chi_1, \ldots, \chi_{q-1})(t,\cdot))]$ in $t$ for $t\in[0,t_0)$. 
Together with the application of 
Lemmas  
\ref{Younggeneral}, \ref{lem:interpolation-ineq} and the estimates from the energy identity in 
\eqref{eq:L^2-d^{1}gamma} 
to the terms on the right-hand side of \eqref{derivative-X_mu(t)}, 
we find the following expression: 
\begin{align}  
\label{eq:epsilon} 
&\sum_{l=1}^{q}\int_{I_{l}} Q^{4k\mu+2k-4, 2k\mu+k-1}_{4}(\gamma_{l,x}) \text{ }dx
+ \sigma^{4\mu-2} \sum_{l=1}^{q-1}\int_{I_{l}} Q^{4\mu-2, 2\mu-1}_{4}(\chi_{l,x}) \text{ }dx 
\\ \nonumber 
& ~~~~~~~~~~~
\leq  \varepsilon \sum_{l=1}^{q}\int_{I_{l}} | D_{x}^{2k\mu+k-1}\gamma_{l, x} |^2 \text{ }dx
+ \sigma^{4\mu-2} \cdot \varepsilon  \sum_{l=1}^{q-1}\int_{I_{l}} | D_{x}^{ 2\mu} \chi_{l, x} |^2 \text{ }dx  
\\ \nonumber 
& ~~~~~~~~~~~~~~~~~~~ 
+ C_1\left(\varepsilon, \mathcal{E}_{k}^{\lambda, \sigma}(t=0), v_{x^\ast=0}^{k-1},\cdots,v_{x^\ast=0}^{1} \right) 
+ \sigma^{4\mu-2} \cdot 
C_2\left(\varepsilon, \mathcal{E}_{k}^{\lambda, \sigma}(t=0)\right), 
\end{align} 
where $C_1$ also depends on $\lambda$, $q$, 
$k$, $m$, and $\underset{\ell\in\{0, \ldots, 2k\mu+k-1\}}\max \|R_M\|_{C^\ell}$; 
$C_2$ also depends on 
$k$, $m$, and $\underset{\ell\in\{0, \ldots, 2\mu\}}\max \|R_M\|_{C^\ell}$; 
and 
$\mathcal{E}_{k}^{\lambda, \sigma}(t=0):=\mathcal{E}_{k}^{\lambda, \sigma} \big[(\gamma(0,\cdot), (\chi_1, \ldots, \chi_{q-1} ) (0,\cdot)) \big]$. 
Moreover, to have simpler notation for the constants, we only keep the essentially relevant dependencies in the constants below. 
We may choose $\varepsilon=1$ 
in \eqref{eq:epsilon} and apply \eqref{derivative-X_mu(t)} 
to obtain the differential inequality 
\begin{align}
\label{inequality-D^{j-1}_t-partial_{t}gamma-fitting}
&\frac{d}{dt}\mathcal{Z}_{\mu}^{\sigma}(t)
+\mathcal{Z}_{\mu}^{\sigma}(t) 
\leq (1+\sigma^{4\mu-2}) \cdot 
C\left( \mathcal{E}_{k}^{\lambda, \sigma}(t=0), 
v_{x^\ast=0}^{k-1},\cdots,v_{x^\ast=0}^{1} \right) 
\text{,} 
~~~~~~~~~~ \forall\, t\in(0,t_{max}).
\end{align} 
By applying Gr\"{o}nwall inequality to \eqref{inequality-D^{j-1}_t-partial_{t}gamma-fitting}, we obtain 
\begin{align}
\label{eq:mainesti} 
&\mathcal{Z}_{\mu}^{\sigma}(t) 
\leq 
\mathcal{Z}_{\mu}^{\sigma}(t_\ast)+(1+\sigma^{4\mu-2}) \cdot 
C\left(\mathcal{E}_{k}^{\sigma}(t=0), v_{x^\ast=0}^{k-1},\cdots,v_{x^\ast=0}^{1} \right)
=: M_{\mu}(\sigma, t_\ast)  
\text{,} 
\end{align} 
for any $0<t_\ast<t<t_{max}$. 
By applying Lemmas \ref{D_x^iD_t^j-f-fitting}, \ref{Younggeneral}, and \ref{lem:interpolation-ineq} to \eqref{eq:mainesti}, we obtain 
\begin{align}
\label{unif-bdd:D-2kj+k-3-derivative} 
&\sum\limits_{l=1}^{q} \| D_{x}^{2k\mu-1} \gamma_{l,x} \|^2_{L^2(I_{l})} 
+\sigma^{4\mu-4} \sum\limits_{l=1}^{q-1} \| D_{x}^{2\mu-1} \chi_{l,x} \|^2_{L^2(I_{l})} 
\leq M_{\mu}(\sigma, t_\ast) 
\text{,} 
\end{align} 
for any $0<t_\ast<t<t_{max}$. 
Thus, the required uniform estimates for \eqref{eq:Z_mu} up to $t=t_{max}$ are derived for any $\mu\in\mathbb{N}$. 
Therefore, the claim for the long-time existence follows.

Notice that, when $\mu \in \mathbb{N} \cap \big[1, 1 + \frac{\ell}{2k} \big]$,  
we can apply the regularity of local solutions  
from Theorem \ref{thm:existence-RS-even-fitting}, the definition of $\mathcal{Z}_{\mu}^{\sigma}$ in \eqref{eq:Z_mu}, and \eqref{eq:mainesti}    
to ensure that the uniform bounds in 
\eqref{eq:mainesti} and \eqref{unif-bdd:D-2kj+k-3-derivative}  still hold as letting $t_\ast \rightarrow 0$, i.e.,    
\begin{align} 
\label{eq:Z_1-esti_t=0} 
\mathcal{Z}_{\mu}^{\sigma}(t) 
&\leq M_{\mu}(\sigma, t_\ast=0), 
\\ \label{unif-bdd:D_t=0} 
\sum\limits_{l=1}^{q} \| D_{x}^{2k\mu-1} \gamma_{l,x} \|^2_{L^2(I_{l})} 
+\sigma^{4\mu-4} \sum\limits_{l=1}^{q-1} \| D_{x}^{2\mu-1} \chi_{l,x} \|^2_{L^2(I_{l})}  
&\leq M_{\mu}(\sigma, t_\ast=0), 
\end{align} 
$\forall\, t\in[0,t_{max}), 
\forall\, \mu \in \mathbb{N} \cap \big[1, 1 + \frac{\ell}{2k} \big]$.

{\bf The asymptotics.} 
The asymptotic behaviour of the global solutions is established using a similar argument to that in \cite{LST22}. For completeness, we restate the proof below, albeit concisely. 
From the uniform bounds in 
\eqref{unif-bdd:D-2kj+k-3-derivative}, 
we may select a subsequence of 
$\gamma \left(t_{j},\cdot \right)$, 
which converges smoothly to $\gamma_{\infty }(\cdot)$ as $t_{j}\rightarrow \infty$.  
We deduce from the long-time existence and \eqref{eq:mainesti} that 
\begin{align}
\label{eq:Z_mu-L^1}
\mathcal{Z}_{1}^{\sigma} \in L^{1}\left((0,\infty )\right).
\end{align} 
From \eqref{eq:Z_mu}, a direct computation of 
$\frac{d}{dt}\mathcal{Z}_{\mu}^{\sigma}(t)$, 
and applying the H\"{o}lder inequality, 
we derive 
\begin{align} 
\label{eq:Z_mu-ineq}
\left|\frac{d}{dt}\mathcal{Z}_{\mu}^{\sigma}(t)\right| 
\leq  
\left|\mathcal{Z}_{\mu}^{\sigma}(t)\right| 
+\left|\mathcal{Z}_{\mu+1}^{\sigma}(t)\right| 
\text{.}
\end{align}  
From applying \eqref{eq:mainesti}, \eqref{eq:Z_1-esti_t=0}, and  setting $\mu=1$ in \eqref{eq:Z_mu-ineq}, we further yield 
\begin{align}
\label{eq:d_t(Z_mu)} 
\left|\frac{d}{dt}\mathcal{Z}_{1}^{\sigma}(t)\right| 
\leq 
C\bigg( 
\mathcal{E}_{k}^{\lambda, \sigma}(t=0), v_{x^\ast=0}^{k-1},\cdots,v_{x^\ast=0}^{1}, 
\mathcal{Z}_{1}^{\sigma}(0), \mathcal{Z}_{2}^{\sigma}(t_\ast) \bigg), 
~ \forall\, t\in(t_\ast,\infty) 
\text{,} 
\end{align} 
where $(t_\ast,\infty)$ can be set as $(0,\infty)$ as 
$2 \in \big[1, 1 + \frac{[\varrho]}{2k} \big]$.
From \eqref{eq:Z_mu-L^1} and \eqref{eq:d_t(Z_mu)}, 
we conclude that  
$\mathcal{Z}_{1}^{\sigma}(t) \rightarrow 0$, as $t\rightarrow \infty$. 
Together with the regularity of $(\gamma, \chi_{1}, \ldots, \chi_{q-1})$, we further conclude that  
$$
\big(\gamma_\infty, \chi_{1, \infty}, \ldots, \chi_{q-1, \infty} \big) \in \Theta_{\mathcal{P}}, 
$$ 
where 
$\gamma_\infty$ satisfies 
$$\mathcal{L}^{2k}_{x}(\gamma_{l,\infty})=0 
~\text{ on }~ (x_{l-1},x_l),  
~~~ \forall\, l\in \{1, \ldots, q\}, 
$$ 
and 
$\chi_\infty$ satisfies 
$$
D_{x}\partial_{x}\chi_{l,\infty}=0 ~\text{ on }~ (x_{l-1},x_l), ~~~ \forall\, l \in \{1, \ldots, q-1\}.  
$$


\section{The Proof of Theorem \ref{thm:Main_Thm-fitting-appl}}
\label{Sec:penalty_method}

In this section, we demonstrate in Theorem \ref{thm:Main_Thm-fitting-appl} that Theorem \ref{thm:Main_Thm-fitting} can be applied to obtain solutions to the spline interpolation problems on Riemannian manifolds by introducing suitable penalty terms and penalty parameter $\sigma$.  
As the penalty parameter $\sigma\rightarrow 0$, 
we show that the sequence of solutions to the gradient flow from Theorem \ref{thm:Main_Thm-fitting} admits a convergent subsequence, which converges to global solutions of the gradient flow 
discussed in Theorem \ref{thm:Main_Thm-fitting-appl}. 
In our argument, this procedure requires introducing elliptic and parabolic Sobolev spaces.

To avoid confusion between the typical notation of elliptic Sobolev spaces and Solonnikov's notation for parabolic Sobolev spaces, we use Solonnikov's notation for elliptic Sobolev spaces instead of the typical one for elliptic Sobolev spaces in this section. 
In fact, we adopt Solonnikov's notation from \cite{Solonnikov65} for the parabolic Sobolev spaces $W_{p,t,x}^{~ \ell,2b\ell}([0,T]\times (0,1))$ by reversing the order of the time and space variables, $t$ and $x$, in the notation of parabolic Sobolev spaces in \cite{Solonnikov65} to align with our terminology (see Definition \ref{definition-parabolic-Sobolev-spaces} in \S\ref{subsection-Notation-terminology} for the details). 
Notice that the elliptic Sobolev space, i.e., the space of weakly differentiable functions defined on $(0,T)\times (0,1)\subset\mathbb{R}^2$ whose derivatives up to order $k$ are integrable in the $L^p$-sense w.r.t. the Lebesgue measure $dx\, dt$, 
is denoted as $W^{k}_{p}((0,T)\times (0,1))$ in Solonnikov's notation for elliptic Sobolev spaces.

\vspace{6cm}

\begin{figure}[h]
\setlength{\unitlength}{0.3 mm}
\begin{picture}(60, 45)
\hspace{4cm}
\begin{tikzpicture}[scale=3]
\put(-68.5,-43){$p_0$} 
\put(-45.5,43){$p_1$}
\put(-2.5,10){$p_2$}
\put(36.5,-42){$p_3$}
\put(67,8){$p_4$}

\put(-72.0,-1){$\gamma_1$}
\put(-15.5,27){$\gamma_2$}
\put(22.5,-10){$\gamma_3$}
\put(56.5,-30){$\gamma_4$}
\put(55,-99){$M$}

{\color{red}
\put(-79.5,-41){\circle*{2}}
\put(-45.5,38) {\circle*{2}}
\put(13.5,10){\circle*{2}}
\put(36.5,-33.5) {\circle*{2}}
\put(80,0){\circle*{2}}
}

\draw (0,0) circle(1);
\draw[dashed] (0,0) ellipse (1 and .2);
\draw [black, xshift=1cm] plot [smooth, tension=2] coordinates { (-0.2,0) (-0.5,-0.4) (-1.3,0.4) (-1.8,-0.4)};

\qbezier(-82.0,0)(-84,-3)(-86,-6) 
\qbezier(-82.0,0)(-81,-3)(-80,-6) 
\qbezier(-15.5,39)(-18.5,41)(-21.5,43)      
\qbezier(-15.5,39)(-18.5,37)(-21.5,35)    
\qbezier(18,-10)(19,-7)(20,-4)    
\qbezier(18,-10)(16,-9)(14,-8)   
\qbezier(66.5,-45)(63.5,-42)(60.5,-39)       
\qbezier(66.5,-45)(63.5,-47)(60.5,-49)  

\end{tikzpicture}
\end{picture}
\vspace{0.5cm}
\caption{The orientation of $\gamma$ as $q=4$.} 
\end{figure}


The argument is divided into several steps.

\bigskip 

Step $1^\circ$ (Uniform integral bounds for derivatives of 
$\{\gamma_{l}^{\sigma}\}_{\sigma\in(0,1), t>0}$ and $\{\chi_{l}^{\sigma}\}_{\sigma\in(0,1), t>0}$ over the time-slice domains $\{t\}\times \bar{I}_l$.):  
\\ 
First, note that since the compatibility conditions are assumed to hold to arbitrarily high order, the uniform bounds in \eqref{eq:Z_1-esti_t=0} and \eqref{unif-bdd:D_t=0} remain valid for any $\mu \in \mathbb{N}$. 
Moreover, because the initial data $\chi^{\sigma}_{l,0} = \chi_{l,0} = p_l$ for all $l \in \{1, \ldots, q-1\}$ are constant maps, it follows that, for each fixed $\sigma$,
$$
\|\chi^{\sigma}_l(t,\cdot)-\chi_{l,0}\|_{C^j(I_l)}\rightarrow 0, 
~~~~~~\text{ as } t\rightarrow 0, \forall\, j\in\mathbb{N}_0. 
$$ 
Thus, we have 
\begin{align}
\label{eq:E[chi_0]=0}  
\mathcal{E}_{k}^{\lambda, \sigma}\big[(\gamma^{\sigma}(0,\cdot), (\chi^{\sigma}_1, \ldots, \chi^{\sigma}_{q-1} ) (0,\cdot) )\big]
=\mathcal{E}_{k}^{\lambda}\big[\gamma^{\sigma}(0,\cdot) \big]
=\mathcal{E}_{k}^{\lambda}\big[\gamma^{0}(0,\cdot) \big]. 
\end{align} 
In other words, the total energy $\mathcal{E}_{k}^{\lambda, \sigma}\big[(\gamma^{\sigma}(0,\cdot), (\chi^{\sigma}_1, \ldots, \chi^{\sigma}_{q-1} ) (0,\cdot) )\big]$, 
or sometimes denoted as $\mathcal{E}_{k}^{\lambda, \sigma}(t=0)$, 
is independent of $\sigma$. 

\smallskip

The assumption on the compatibility conditions of any order and  
the existence of global solutions $(\gamma^{\sigma},\chi^{\sigma}_1,\cdots,\chi^{\sigma}_{q-1})$ in the H\"{o}lder spaces, i.e., 
\begin{align*}
\begin{cases}
\gamma^{\sigma}_{l}\in C^{\frac{2k+\varrho_1}{2k},2k+\varrho_1}([0,\infty)\times  \bar{I}_{l}), 
\\
\chi^{\sigma}_{l}\in C^{\frac{2+\varrho_2}{2},2+\varrho_2}([0,\infty)\times \bar{I}_{l}), 
\end{cases} 
\end{align*}  
$\forall\, l$,  
ensure that $\mathcal{E}_{k}^{\lambda, \sigma}[(\gamma^{\sigma}(t,\cdot), (\chi^{\sigma}_1, \ldots, \chi^{\sigma}_{q-1})(t,\cdot))]$ 
and $\mathcal{Z}_{\mu}^{\sigma}(t)$ are continuous for all $t\in[0,\infty)$, $\forall\, \mu\in\mathbb{N}$. 
Thus, together with \eqref{eq:E[chi_0]=0}, 
\eqref{eq:Z_1-esti_t=0}, and the definition of $\mathcal{Z}_{\mu}^{\sigma}$ in \eqref{eq:Z_mu},
we have that, 
$\forall\, \mu\in\mathbb{N}$, $\forall\, t\in(0,\infty)$, 
\begin{align}
\label{eq:chi_W^{2,2}_bdd} 
\mathcal{Z}_{\mu}^{\sigma}(t)  
&\leq (1+\sigma^{4\mu-2}) \cdot 
C\left( 
\mathcal{E}_{k}^{\lambda}\big[\gamma_{0} \big], 
v_{x^\ast=0}^{k-1},\cdots,v_{x^\ast=0}^{1}, 
\sum\limits_{l=1}^{q}\| D_t^{\mu-1} \partial_{t} \gamma^{\sigma}_{l}(0,\cdot)\|^2_{L^2(I_{l})} 
\right) 
\\ \nonumber 
&  
\leq (1+\sigma^{4\mu-2}) \cdot 
C\left(
\mathcal{E}_{k}^{\lambda}\big[\gamma_{0} \big], v_{x^\ast=0}^{k-1},\cdots,v_{x^\ast=0}^{1}, 
\sum\limits_{l=1}^{q}\|\gamma_{l,0}\|^2_{W^{2k\mu}_{2}(I_l)}
\right) 
=:\tilde{M}_{\mu}(\sigma)
\text{,}
\end{align} 
where $\gamma^{\sigma}_{l,0}=\gamma_{l,0}$, $\forall\,\sigma>0$, $\forall\, l\in\{1,\cdots,q\}$
is also used in deriving the last inequality. 
Note that the upper bound $C$ in \eqref{eq:chi_W^{2,2}_bdd} also depends on less relevant quantities, 
$k, m, q, \|R_M\|_{C^\ell}$, 
where $\ell=\ell(k,\mu)$, as explained in \eqref{eq:epsilon}.
Moreover, for each fixed $\mu$,  
\begin{align} 
\label{eq:uni_bdd_M_mu}
0\le \tilde{M}_{\mu}(\sigma)\le \tilde{M}_{\mu}(1)
= 2\cdot C\left(
\mathcal{E}_{k}^{\lambda}\big[\gamma_{0} \big], 
v_{x^\ast=0}^{k-1},\cdots,v_{x^\ast=0}^{1}, 
\sum\limits_{l=1}^{q}\|\gamma_{l,0}\|^2_{W^{2k\mu}_{2}(I_l)}
\right), 
~~~ \forall\, \sigma\in(0,1). 
\end{align} 
Using \eqref{eq:E[chi_0]=0}, \eqref{eq:chi_W^{2,2}_bdd}, \eqref{eq:uni_bdd_M_mu}, 
applying Lemmas \ref{D_x^iD_t^j-f-fitting}, \ref{Younggeneral}, \ref{lem:interpolation-ineq} to the definition of $\mathcal{Z}^{\sigma}_{\mu}$ in \eqref{eq:Z_mu}, 
and applying the flow equation in \eqref{eq:E_s-flow-fitting}, 
we obtain that, $\forall\, \mu\in\mathbb{N}$, 
$\forall\, t\in(0,\infty)$, $\forall\, \sigma\in(0,1)$, 
\begin{align}
\label{eq:gamma_W^{2k,2}_bdd} 
&\sum\limits_{l=1}^{q} \| D_{x}^{2k\mu-1} \gamma^{\sigma}_{l,x} \|^2_{L^2(I_{l})} 
\leq 
\tilde{M}_{\mu}(1),  
\\ 
\label{eq:chi_uni_bdd_by_M(1)} 
&\sum\limits_{l=1}^{q-1}\| D_t^{\mu-1} \partial_{t} \chi^{\sigma}_{l}(t,\cdot)\|^2_{L^2(I_{l})} 
\le \tilde{M}_{\mu}(1),   
\\ 
\label{eq:chi_uni_bdd_by_M(1)-1}
&\sum\limits_{l=1}^{q-1}\| D_t^{\mu-1} D_{x} \partial_x \chi^{\sigma}_{l}(t,\cdot)\|^2_{L^2(I_{l})} 
\le \tilde{M}_{\mu}(1) 
\text{.} 
\end{align}

Note that we cannot directly obtain uniform bounds for $\sum\limits_{l=1}^{q-1}\| D_x^{\mu} \partial_x  \chi^{\sigma}_{l}(t,\cdot)\|^2_{L^2(I_{l})}$ for arbitrarily large $\mu$, solely by applying 
\eqref{eq:chi_uni_bdd_by_M(1)} and \eqref{eq:E_s-flow-fitting}. Therefore, in what follows, we consider an alternative Sobolev norm for $\chi_l$ over the space-time domain $(0,T)\times I_l$ to circumvent this difficulty. 
By letting $\mu=1$ in \eqref{eq:chi_uni_bdd_by_M(1)-1} and $\mu=2$ in \eqref{eq:chi_uni_bdd_by_M(1)} and applying the flow equation in \eqref{eq:E_s-flow-fitting}, 
we have 
\begin{align} 
\label{eq:chi_uni_bdd_by_M_2(1)}
\sum\limits_{l=1}^{q-1}
\big(\| D_{x} \partial_x \chi^{\sigma}_{l}(t,\cdot)\|^2_{L^2(I_{l})}
+ \| D_t \partial_t \chi^{\sigma}_{l}(t,\cdot)\|^2_{L^2(I_{l})} 
\big)
\le \tilde{M}_{2}(1), 
~~~ \forall\, \sigma\in(0,1), 
\end{align} 
where $\tilde{M}_{1}(1)$ is absorbed by $\tilde{M}_{2}(1)$ in deriving the last inequality. 
By letting $\mu=2$ in \eqref{eq:chi_uni_bdd_by_M(1)-1}, using \eqref{eq:chi_uni_bdd_by_M_2(1)}, and applying the Gagliardo-Nirenberg interpolation inequalities to  $\| D_t \partial_x \chi^{\sigma}_{l}(t,\cdot)\|^2_{L^2(I_{l})}$, $\forall\, l$, 
we have 
\begin{align} 
\label{eq:chi_uni_bdd_by_M_2(1)-2}
\sum\limits_{l=1}^{q-1}
\| D_t \partial_x \chi^{\sigma}_{l}(t,\cdot)\|^2_{L^2(I_{l})}
\le \tilde{M}_{2}(1), 
~~~ \forall\, \sigma\in(0,1)  
\text{.}
\end{align} 
Moreover, the uniform bound 
\begin{align}
\label{eq:eq:chi_W^{2,2}_bdd}
\| \partial_x \chi^{\sigma}_{l}(t,\cdot)\|^2_{L^2(I_{l})}\le \mathcal{E}_{k}^{\lambda}\big[\gamma^{\sigma}(0,\cdot)\big], 
~~~ \forall\, \sigma\in(0,1), 
\end{align} 
can be obtained from applying 
\eqref{ineq:f_x-higher-order-f} and \eqref{eq:E[chi_0]=0}. 
From the Dirichlet boundary conditions of 
$\chi_l^\sigma$ in \eqref{eq:BC-2-order-1-fitting}, integrating $\partial_x\chi_l^\sigma(t,\cdot)$ w.r.t. space variable $x$, and using \eqref{eq:eq:chi_W^{2,2}_bdd}, we have 
\begin{align}
\label{eq:|chi_l-p_l|->0}
|\chi^{\sigma}_{l}(t,x)-p_l| \le \sigma^2 \cdot 
\sqrt{2\cdot \mathcal{E}_{k}^{\lambda}\big[\gamma^{0}(0,\cdot)\big]},
~~~ \forall\,\sigma>0, ~ \forall\, (t,x)\in [0,T]\times \bar{I}_l, ~ \forall\, l\in\{1,\cdots,q-1\}
\text{,}
\end{align} 
and 
\begin{align} 
\label{eq:chi_uni_bdd_C^0}
\sum\limits_{l=1}^{q-1}
\| \chi^{\sigma}_{l}(t,\cdot)\|^2_{L^\infty(I_{l})}\le C\big(\mathcal{E}_{k}^{\lambda}\big[\gamma^{\sigma}(0,\cdot)\big], p_0, \ldots, p_q\big), 
~~~ \forall\, \sigma\in(0,1)
\text{.}
\end{align} 
From \eqref{eq:chi_uni_bdd_by_M(1)}, \eqref{eq:chi_uni_bdd_by_M(1)-1},  \eqref{eq:chi_uni_bdd_by_M_2(1)}, \eqref{eq:chi_uni_bdd_by_M_2(1)-2}, 
\eqref{eq:eq:chi_W^{2,2}_bdd}, \eqref{eq:chi_uni_bdd_C^0},  
we obtain that, for all $\sigma\in(0,1)$, 
\begin{align*} 
\nonumber 
\sum\limits_{l=1}^{q-1}
&\big( \| \chi^{\sigma}_{l}(t,\cdot)\|^2_{L^2(I_{l})}
+\| \partial_x \chi^{\sigma}_{l}(t,\cdot)\|^2_{L^2(I_{l})}
+ \|\partial_t \chi^{\sigma}_{l}(t,\cdot)\|^2_{L^2(I_{l})} 
+\| D_{x} \partial_x \chi^{\sigma}_{l}(t,\cdot)\|^2_{L^2(I_{l})} 
\\ 
& ~~~~~~~~~~~~~~ 
+ \| D_t \partial_t \chi^{\sigma}_{l}(t,\cdot)\|^2_{L^2(I_{l})} 
+ \| D_t \partial_x \chi^{\sigma}_{l}(t,\cdot)\|^2_{L^2(I_{l})}
\big)
\le 
C\big(\tilde{M}_{2}(1), p_0, \ldots, p_q\big). 
\end{align*} 
Thus, we conclude that 
$\{\chi^{\sigma}_{l}\}_{\sigma\in(0,1)}$ is a bounded subset in the elliptic Sobolev space, 
$W^{2}_{2}((0,T) \times I_{l})$, for any fixed $T>0$. 

\bigskip 

Step $2^\circ$ (Strong convergence in H\"{o}lder norms:   
$\sum\limits_{l=1}^{q}\|\gamma_{l}^{\sigma}-\gamma_{l}^{0}\|_{C^{\mu}([0,T]\times \bar{I}_l)}\rightarrow 0$, as $\sigma\rightarrow 0$, for any $\mu\in\mathbb{N}$, for any $T>0$;  
$\sum\limits_{l=1}^{q-1}\|\chi_{l}^{\sigma}-p_{l}\|_{C^{0}([0,T]\times \bar{I}_l)}\rightarrow 0$, as $\sigma\rightarrow 0$, for any $T>0$.): 
\\ 
From \eqref{eq:gamma_W^{2k,2}_bdd} and together with applying the definition of covariant differentiation given in \eqref{eq:co-deri} and the Gagliardo-Nirenberg interpolation inequalities, we yield that,  
$\forall\, \sigma\in(0,1), \, \forall\,\mu\in\mathbb{N}, 
T\in(0,\infty)$, 
\begin{align}
\label{eq:Z_mu-uniform_bdd}
\sum_{l=1}^{q} 
\|\gamma_{l}^{\sigma}\|_{W_{2,t,x}^{~ \mu,2k\mu}((0,T)\times I_l)}
\le C\left(T, 
\mathcal{E}_{k}^{\lambda}\big[\gamma_{0} \big], 
v_{x^\ast=0}^{k-1},\cdots,v_{x^\ast=0}^{1}, 
\sum\limits_{l=1}^{q}\|\gamma_{l,0}\|^2_{W^{2k\mu}_{2}(I_l)}
\right)  
\text{.}  
\end{align}
Applying \eqref{eq:Z_mu-uniform_bdd}, we obtain that for each $l$ and any $\mu\in\mathbb{N}$, the family 
$\{\gamma^{\sigma}_{l}\}_{\sigma\in(0,1)}$ is a bounded subset in the elliptic Sobolev space 
$W^{\mu}_{2}((0,T) \times I_{l})$ 
with bounds given by $C(\tilde{M}_{\mu}(1),T)$. 
By the extension theorem for Sobolev functions on Lipschitz domains in Euclidean spaces, 
the Rellich-Kondrachov compactness theorem, 
and the Sobolev embedding theorem, we obtain that 
for any $\mu \in \mathbb{N} \setminus \{1\}$, 
$$ 
W^{\mu}_{2}((0,T) \times I_{l})\subset W^{\mu-1}_{p}((0,T) \times I_{l})\subset C^{\mu-2+\alpha_1}([0,T]\times \bar{I}_{l}), 
$$ 
for any $p\in[1,\infty)$, and some $\alpha_1\in(0,1)$.  
Consequently, 
$\{\gamma^{\sigma}_{l}\}_{\sigma\in(0,1)}$ forms a bounded subset in 
$C^{\mu-2+ \alpha_1}([0,T]\times \bar{I}_{l})$. 
By applying the Arzel\`{a}-Ascoli theorem inductively on the order of differentiation to extract convergent subsequences of functions from the previous sequence of functions, we conclude that there exists a convergent subsequence 
$\{\gamma^{\sigma_j}_{l}\}_{j\in\mathbb{N}}
\subset C^{\mu-2}([0,T]\times \bar{I}_{l})$, 
and a limit function 
$\gamma^{0}_{l}\in C^{\mu-2}([0,T]\times \bar{I}_{l})$ 
such that \eqref{eq:gamma^0_in_C^s} holds for any 
$\mu\in \mathbb{N}\setminus\{1\}$. 

\smallskip  

Since Step $1^\circ$ establishes that the family $\{\chi^{\sigma}_{l}\}_{\sigma\in(0,1)}$ is bounded in the elliptic Sobolev space $W^{2}_{2}((0,T) \times I_{l})$  for any fixed $T > 0$, we can apply the Rellich-Kondrachov compactness theorem. This implies that the elliptic Sobolev space $W^{2}_{2}((0,T)\times I_l)$ is compactly embedded into $W^{1}_{p}((0,T)\times I_l)$ for any $p \in [1, \infty)$. 
Moreover, since the time-space domain $\Omega := (0,T) \times I_l \subset \mathbb{R}^2$ has dimension two, the Sobolev embedding theorem guarantees the compact embedding $W^1_p(\Omega) \hookrightarrow C^{0,\beta}(\Omega)$ for any $p > 2$ and some $\beta \in (0,1)$. Combining these results, we obtain the compact embedding,  
$$
W^2_2(\Omega) \hookrightarrow C^{0,\beta}(\Omega), 
$$ 
for some $\beta \in (0,1)$. 
Hence, from \eqref{eq:|chi_l-p_l|->0}, we deduce 
$$
\sum\limits_{l=1}^{q-1}\|\chi_{l}^{\sigma}-p_{l}\|_{C^{0}([0,T]\times \bar{I}_l)}\rightarrow 0,  
~~~~~~ \text{ as } 
\sigma\rightarrow 0.
$$  
In addition, the continuity of the maps, 
$(\gamma^{\sigma}, \chi^{\sigma}_1, \dots, \chi^{\sigma}_{q-1})$, 
implies that 
\begin{align*} 
|\gamma^{\sigma}_{l}(t,x_l)-p_l|\rightarrow 0, 
~~ 
|\gamma^{\sigma}_{l+1}(t,x_l)-p_l|\rightarrow 0, 
~~~~~~\text{ as } \sigma\rightarrow 0,  
\end{align*} 
for all $t \in [0,T]$ and $l \in \{1,\dots,q-1\}$.

\bigskip 

Step $3^\circ$ (Global solutions.): 
\\ 
To extend the result in Step $2^\circ$ from 
$[0,T]\times\bar{I}$ to $[0,\infty)\times\bar{I}$, we may simply apply the diagonal argument along with $0<T_{1}<\cdots<T_{p}<\cdots$. 
to select a convergent subsequence 
$$
\{\gamma^{\sigma_j}_{l}\}_{j\in\mathbb{N}}
\subset C^{\mu-2}([0,\infty)\times\bar{I}_{l}), 
~~~~~~ \forall\, l\in\{1,\cdots,q\}, 
$$ 
in the corresponding topology of H\"{o}lder spaces, with a limit function 
$\gamma^{0}_{l}\in C^{\mu-2}([0,\infty)\times\bar{I}_{l})$. 
Thus, we conclude the strong convergence \eqref{eq:gamma^0_in_C^s},  
which ensures that $\gamma^0$ is indeed a classical global solution to the flow described by \eqref{eq:E_s-flow}$\sim$\eqref{eq:BC-higher-order-3}, and $\gamma^0(t,\cdot)$ passes through the points $p_l$ for all $l \in \{1,\dots,q\}$.

By following Steps $1^\circ$ through $3^\circ$, we have now completed the proof of Statement {\bf (i)}.

\bigskip 

Step $4^\circ$ (Asymptotics.): 
Since $\gamma^{0}_{l}\in C^\infty([0,\infty)\times \bar{I}_l)$, $\forall\, l$,
and $\gamma^{0}$ is a classical global solution to the nonlinear  parabolic problem, we can follow the same procedure outlined in the previous section to obtain the asymptotic limit. 
For details of the argument, we refer the interested reader to the argument on the asymptotics in the proof of 
Theorem \ref{thm:Main_Thm-fitting}. 
The proof of Statement {\bf (ii)} is now completed. 
 


\section{Appendices}
\label{Sec:Appendix} 

\subsection{Notation and Terminology}\label{subsection-Notation-terminology}

For tangent vector fields $v, v_{1}, \ldots, v_{k}$ on $M$,
we denote by 
\begin{equation*}
v_{1}\ast \cdot\cdot\cdot \ast v_{k}=\left\{
\begin{array}{l} 
\langle v_{i_{1}}, v_{i_{2}}\rangle \cdot \cdot \cdot \langle v_{i_{k-1}}, v_{i_{k}}\rangle \text{ \ \ \ \ \ \ \ \ ,\ for }k\text{\ even,} \\
\langle v_{i_{1}}, v_{i_{2}}\rangle \cdot \cdot \cdot \langle v_{i_{k-2}}, v_{i_{k-1}}\rangle \cdot v_{i_{k}}
\text{ , for }k\text{\, odd,}
\end{array}
\right.
\end{equation*}
where $i_{1}, \ldots, i_{k}$ is any permutation of $1, \ldots, k$, $\langle\cdot,\cdot\rangle=\langle\cdot,\cdot\rangle_g$, and $g$ is the Riemannian metric of $M$. 
Slightly more generally, we allow some of the tangent vector fields $v_{i}$ to be scalar functions, in which the $\ast $-product reduces to multiplication. 
The notation $P^{a,c}_b(v)$ represents the summation of finitely many terms 
$D_{x}^{\mu_1} v \ast \cdots \ast D_{x}^{\mu_b} v $  
with $\mu_1+\cdots+\mu_b=a$, $\max\{\mu_1, \ldots, \mu_b\}\le c$, and with universal coefficients, 
depending on the Riemannian curvature tensor $R_M$ and Riemannian metric $g_M$ and their derivatives. 
The universal coefficients and the numbers of terms in $P_b^{a,c}(v)$ 
are assumed to be bounded from above and below by constants depending only on 
$a$, $b$, $K_0$, where $K_0$ is determined by the $C^{\infty}$-norm of the Riemannian curvature tensor $R_M$. 
Note that it is easy to verify the following:
\begin{equation*}
\left\{
\begin{array}{l} 
D_{x}\left( P_{b_{1}}^{a_{1},c_{1}}\left( v \right) \ast P_{b_{2}}^{a_{2},c_{2}}\left( v \right) \right) 
=D_{x} P_{b_{1}}^{a_{1},c_{1}}\left( v \right) \ast P_{b_{2}}^{a_{2},c_{2}}\left( v \right) 
+P_{b_{1}}^{a_{1},c_{1}}\left( v \right) \ast D_{x} P_{b_{2}}^{a_{2},c_{2}}\left( v \right) \text{,} 
\\ \\
P_{b_{1}}^{a_{1},c_{1}}\left( v \right) \ast P_{b_{2}}^{a_{2},c_{2}}\left( v \right)
=P_{b_{1}+b_{2}}^{a_{1}+a_{2},max\{c_1,c_2\}}\left( v \right) \text{, }D_{x}P_{b_{2}}^{a_{2},c_{2}}\left( v \right) 
=P_{b_{2}}^{a_{2}+1,c_{2}+1}\left( v \right) 
\text{.}
\end{array}
\right. 
\end{equation*} 
The notation $Q_{b}^{a, c}(\gamma_{x})$ is introduced to simplify the summation of polynomials. Namely,   
\begin{equation*}
Q_{b}^{a, c}(\gamma_{x}) 
:= \sum_{ \substack{  [\![\mu, \nu]\!]\leq [\![a,b]\!] \\\xi \leq c} }   P_{\nu}^{\mu, \xi}(\gamma_{x}) 
=\sum_{\mu=0}^{a}\sum_{\nu=1}^{2a+b-2\mu}\sum_{\xi=0}^{c}\text{ } P_{\nu}^{\mu, \xi}(\gamma_{x})
\text{,} 
\end{equation*}  
where 
$[\![\mu,\nu]\!]:=2\mu+\nu$, and $a, c \in \mathbb{N}_{0}$, $b \in \mathbb{N}$. 
For simplicity, we also use the notation $P^{a}_{b}\left(\gamma_{x} \right)$ and 
$Q_{b}^{a}(\gamma_{x})$, 
to represent $P^{a,a}_{b}\left(\gamma_{x} \right)$ and 
$Q_{b}^{a, a}(\gamma_{x})$
respectively. 
Additionally, we denote  
$\gamma_x=\partial_{x}\gamma$, $\gamma_t=\partial_{t}\gamma$, and similarly 
$\gamma_{l,x}=\partial_{x}\gamma_l$, $\gamma_{l,t}=\partial_{t}\gamma_l$. 
We may also abuse the notation for covariant differentiation by writing $D^i_x\gamma_l=D^{i-1}_x\partial_x\gamma_l$ or $D^j_t\gamma_l=D^{j-1}_t\partial_t\gamma_l$.


\smallskip 

For the reader’s convenience, we recall some notation for parabolic H\"{o}lder spaces from \cite[Page 66]{Solonnikov65}, with slight modifications, which will be used throughout our discussion.

\begin{defi}
[{The space of parabolic H\"{o}lder spaces 
$C^{\frac{\varrho}{2k}, \varrho}_{~t, ~x}
\left([0,T] \times [0,1]\right)$} ]
\label{definition-parabolic-Holder-spaces} 

\text{ }

\begin{itemize}
\item[(i)] 
The parabolic continuously differentiable function space 
$C^{r, 2kr}\left([0,T] \times [0,1]\right)$, 
$ r \in \mathbb{N}$, is the set of functions 
$g: [0,T] \times [0,1] \to \mathbb{R}$, 
with 
$\partial_{x}^{\mu}\partial_{t}^{\nu} g \in C^{0}([0,T] \times [0,1])$,  
$\forall\, \nu, \mu \in \mathbb{N}_{0}$, $2\nu+\mu \leq \ell$, 
and with the finite norm 
\begin{align*}
\norm{g}_{C^{r, 2kr}\left( [0,T] \times [0,1]\right)}
=&\sum\limits^{2kr}_{2k\nu+\mu=0} \underset{(t,x) \in  [0,T] \times [0,1]}{\sup} | \partial_{x}^{\mu}\partial_{t}^{\nu}g(t,x) |. 
\end{align*} 

\item[(ii)] 
The parabolic H\"{o}lder space,  
$C^{\frac{\ell+\alpha}{2k}, \ell+\alpha}\left([0,T] \times [0,1]\right)$, $\alpha \in (0,1)$ and $\ell \in \mathbb{N}_{0}$,  
is the set of functions 
$g: [0,T] \times [0,1] \to \mathbb{R}$, 
with  
$\partial_{x}^{\mu}\partial_{t}^{\nu} g\in C^{0}([0,T] \times [0,1])$, 
$\forall\, \nu, \mu \in \mathbb{N}_{0}$, $2\nu+\mu \leq \ell$, 
and with the finite norm 
\begin{align*}
\norm{g}_{C^{\frac{\ell+\alpha}{2k}, \ell+\alpha}\left( [0,T] \times [0,1]\right)}
=&\sum\limits^{\ell}_{2k\nu+\mu=0} \underset{(t,x) \in  [0,T] \times [0,1]}{\sup} | \partial_{x}^{\mu}\partial_{t}^{\nu}g(t,x) |
+\sum\limits_{2k\nu+\mu=\ell}[\partial_{x}^{\mu}\partial_{t}^{\nu}g]_{\alpha, x}
\\
&+\sum\limits_{0<\ell+\alpha-2k\nu-\mu<2k}[\partial_{x}^{\mu}\partial_{t}^{\nu}g]_{\frac{\ell+\alpha-2k\nu-\mu}{2k},t}, 
\end{align*} 
where 
for any $\varrho\in (0,1)$  
\begin{align*}
[ g ]_{\varrho, x}=&\underset{(t,x),(t,x^{\prime}) \in  [0,T] \times [0,1]}{\sup} \frac{| g(t,x)-g(t,x^{\prime}) |}{|x-x^{\prime}|^{\varrho}},
\\
[ g ]_{\varrho,t}=& \underset{(t,x),(t^\prime,x) \in  [0,T] \times [0,1]}{\sup} \frac{| g(t,x)-g(t^\prime,x)|}{|t-t^{\prime}|^{\varrho}}.
\end{align*}

\end{itemize} 
\end{defi}

\begin{rem}[{The parabolic H\"{o}lder spaces and elliptic H\"{o}lder spaces}]
For simplicity, we denote $C^{\frac{\varrho}{2k}, \varrho}_{~t, ~x}
\left([0,T] \times [0,1]\right)$ by $C^{\frac{\varrho}{2k}, \varrho}
\left([0,T] \times [0,1]\right)$, where $\varrho \in (0,\infty) \setminus \mathbb{Z}$. 
To avoid potential confusion between parabolic and elliptic Hölder spaces, we adopt the notation $C^{\varrho}(\bar{I})$ or $C^{\varrho}(\bar{\Omega})$ for elliptic spaces instead of the standard forms 
$C^{[\varrho],\{\varrho\}}(\bar{I})$ or $C^{[\varrho],\{\varrho\}}(\bar{\Omega})$, respectively. Here, $\{\varrho\} = \varrho - [\varrho]$ denotes the fractional part of $\varrho$ and $[\varrho]$ is the greatest integer less than $\varrho$.  
\end{rem}


\bigskip

\begin{defi}[{parabolic Sobolev spaces}]\label{definition-parabolic-Sobolev-spaces}
    For integers $l \geq 0$, $k \geq 1$, 
and real number $p>1$, the parabolic Sobolev space 
$W^{~ l, 2kl}_{p,t,x}([0,T] \times [0,1])$ is defined as the closure of the set of smooth functions on $[0,T] \times [0,1]$ under the norm 
\begin{align*}
\|u\|_{p,2kl}=\sum_{0 \leq 2k \mu+\nu \leq 2k l} \langle D_x^{\nu} D_t^{\mu} u \rangle_{p,0} 
\text{,}
\end{align*}
where 
\begin{align*}
\langle u \rangle_{p,0}:=\bigg( \int\limits_{[0,T] \times [0,1]} |u(t,x)|^p \, dt dx \bigg)^{\frac{1}{p}}.
\end{align*} 
\end{defi}


For (elliptic) Sobolev spaces, denoted as $W^k_p(\Omega)$ in this article instead of $W^{k,p}(\Omega)$, we adopt the simplified notation for Sobolev norms as used in \cite{DP14} and our previous works: 
\begin{equation*}
\|g\|_{k,p} 
=\|g\|_{W^{k}_{p}(\Omega)} 
:= \sum_{\ell=0}^{k} \| D_x^{\ell} g \|_{p}, 
~~ \mbox{ where } ~~
\|D_x^{\ell} g\|_{p}
=\|D_x^{\ell} g\|_{L^p(\Omega)} 
:= 
\Big( \int_{\Omega} |D_x^{\ell} g |^p dx \Big)^{1/p} \, 
\text{,}
\end{equation*} 
which is slightly different from the scale-invariant Sobolev norms defined in \cite{DP14}.

\bigskip 

\subsection{Technical Lemmas}\label{subsection-Technical-lemmas}

\begin{lem}
\label{lem:D_t-D_x-gamma}  
Let $\tilde\gamma : (t_0,T_0)\times (0,1)\rightarrow M$ be a smooth map, 
and $V: (t_0,T_0)\times (0,1)\rightarrow T_{\tilde\gamma}M$ be a smooth vector field associated with $\tilde\gamma$.  
Then, 
\begin{align}
D_{t} \partial_{x} \tilde\gamma
=& D_{x} 
\partial_{t} \tilde\gamma 
\text{,}
\label{eq:D_t-D_x^j=D_x^j-D_t}
\\ 
D_{t}D_{x}^{\mu}V
=& D_{x}^{\mu}D_{t}V
+\sum_{\nu=1}^{\mu} D_{x}^{\mu-\nu}\left[R(\tilde\gamma_{t}, \tilde\gamma_{x}) D_{x}^{\nu-1} V \right], 
~~\forall\, \mu \geq 1 
\text{,}
\label{eq:D_t-D_x^j->D_x^j-D_t}
\\ 
D_{t}^{\mu} D_{x}V
=& D_{x}D_{t}^{\mu}V 
+\sum_{\nu=1}^{\mu} D_{t}^{\mu-\nu}\left[R(\tilde\gamma_{t}, \tilde\gamma_{x}) D_{t}^{\nu-1} V \right], 
~~\forall\, \mu \geq 1 
\text{.} 
\label{eq:D_t^j-D_x->D_x-D_t^j}
\end{align}
\end{lem}

\begin{proof} 
By applying Nash's isometric embedding theorem, we can assume that $M$ is a submanifold of $\mathbb{R}^n$ for some integer $n$. 
Using covariant differentiation, we derive the following formulas:   
\begin{align*} 
D_{t}\partial_{x}\tilde\gamma
=\partial_{t}\partial_{x}\tilde\gamma-\sum_{\alpha=m+1}^{n}\langle\partial_{t}\partial_{x}
\tilde\gamma, e_\alpha\rangle e_\alpha, 
\\ 
D_{x}\partial_{t}\tilde\gamma
=\partial_{x}\partial_{t}\tilde\gamma-\sum_{\alpha=m+1}^{n}\langle\partial_{x}\partial_{t}
\tilde\gamma, e_\alpha\rangle e_\alpha. 
\end{align*}
Since $\partial_{t}\partial_{x}\tilde\gamma=\partial_{x}\partial_{t}\tilde\gamma$, we conclude \eqref{eq:D_t-D_x^j=D_x^j-D_t}.

We prove \eqref{eq:D_t-D_x^j->D_x^j-D_t} inductively. 
As $\mu=1$, the equality 
\begin{align}
\label{eq:D_t-D_x->D_x-D_t} 
D_{t}D_{x}V=&D_{x}D_{t}V+R(\tilde\gamma_{t}, \tilde\gamma_{x})V
\end{align}
comes from using the definition of Riemannian curvature tensors $R(\boldsymbol{\cdot},\boldsymbol{\cdot})\boldsymbol{\cdot}$. 
By an induction argument and applying 
\eqref{eq:D_t-D_x->D_x-D_t}, we obtain  
\begin{align*}
D_{t}D_{x}^{\mu+1}V=&D_{t}D_{x} \left(D_{x}^{\mu}V\right)
=D_{x}D_t \left(D_{x}^{\mu}V\right)+R(\tilde\gamma_{t}, \tilde\gamma_{x} ) D_{x}^{\mu}V \\
=&D_{x} \left(D_{x}^{\mu}D_{t}V
+\sum_{\nu=1}^{\mu} D_{x}^{\mu-\nu}\left[R(\tilde\gamma_{t}, \tilde\gamma_{x}) D_{x}^{\nu-1} V \right] \right)
+R(\tilde\gamma_{t}, \tilde\gamma_{x} ) D_{x}^{\mu}V \\
=&D_{x}^{\mu+1}D_{t}V
+\sum_{\nu=1}^{\mu+1} D_{x}^{\mu+1-\nu}\left[R(\tilde\gamma_{t}, \tilde\gamma_{x}) D_{x}^{\nu-1} V \right], 
\end{align*}
which finishes the proof of 
\eqref{eq:D_t-D_x^j->D_x^j-D_t}.

Below, we prove \eqref{eq:D_t^j-D_x->D_x-D_t^j}.   
The case $\mu=1$ is proven in \eqref{eq:D_t-D_x^j->D_x^j-D_t}.
By an induction argument and \eqref{eq:D_t-D_x^j->D_x^j-D_t}, 
we obtain  
\begin{align*}
D_{t}^{\mu+1}D_{x}V=&D_{t} \left(D_{x}D_{t}^{\mu}V
+\sum_{\nu=1}^{\mu} D_{t}^{\mu-\nu}\left[R(\tilde\gamma_{t}, \tilde\gamma_{x}) D_{t}^{\nu-1} V \right]\right)\\
=&D_{x}D_t \left(D_{t}^{\mu}V\right)+R(\tilde\gamma_{t}, \tilde\gamma_{x} ) D_{t}^{\mu}V 
+\sum_{\nu=1}^{\mu} D_{t}^{\mu-\nu+1}
\left[R(\tilde\gamma_{t}, \tilde\gamma_{x}) D_{t}^{\nu-1} V \right] \\ 
=&D_{x} D_{t}^{\mu+1}V
+\sum_{\nu=1}^{\mu+1} D_{t}^{\mu+1-\nu}
\left[R(\tilde\gamma_{t}, \tilde\gamma_{x}) D_{t}^{\nu-1} V \right], 
\end{align*}
which finishes the proof of \eqref{eq:D_t^j-D_x->D_x-D_t^j}. 

\end{proof}

The following technical lemma plays a crucial role in converting higher-order covariant derivatives with respect to the time variable $t$ into covariant derivatives with respect to the space variable $x$, particularly for networks $(\gamma, \chi_{1}, \ldots, \chi_{q-1} )$ satisfying the flow equation \eqref{eq:E_s-flow-fitting}.

\begin{lem} 
\label{D_x^iD_t^j-f-fitting}
Suppose $\gamma$ and $(\chi_1, \ldots, \chi_{q-1})$ are smooth solutions 
fulfilling the parabolic systems in 
\eqref{eq:E_s-flow-fitting} respectively. 
Then, for any $\mu \in\mathbb{N}$, $\ell \in\mathbb{N}_{0}$, 
we have     
\begin{align}
\label{eq:1-higher-order-fitting}
D_{t}^{\mu} D^{\ell}_{x} \gamma_{l}
=&(-1)^{(k+1) \mu} D_{x}^{2k\mu+\ell-1} \gamma_{l,x}+Q^{2k\mu+\ell-3}_{3}(\gamma_{l,x}), 
\\
\label{eq:1-second-order-fitting}
D_{t}^{\mu} D^{\ell}_{x} \chi_{l}
=& \sigma^{2\mu} D_{x}^{2\mu+\ell-1} \chi_{l,x}
+ \sigma^{2\mu} \cdot Q^{2\mu+\ell-3}_{3}(\chi_{l,x}), 
\\
\label{eq:2-second-order-fitting}
D^{\ell}_{x}  D_{t}^{\mu} \chi_{l}
=& \sigma^{2\mu} D_{x}^{2\mu+\ell-1} \chi_{l,x}
+ \sigma^{2\mu} \cdot Q^{2\mu+\ell-3}_{3}(\chi_{l,x}),
\end{align} 
where $D^{\ell}_{x} \gamma_{l}:=D^{\ell-1}_{x} \partial_x\gamma_{l}$, 
$D^{\ell}_{x} \chi_{l}:=D^{\ell-1}_{x} \partial_x\chi_{l}$, as $l\ge 1$ 
and 
$D_{x} \gamma_{l}:=D_{x} \partial_x\gamma_{l}$, 
$D_{x} \chi_{l}:=D_{x} \partial_x\chi_{l}$. 
\end{lem} 

\begin{proof}  
We prove the lemma by induction on $\mu$. 

\smallskip 

$1^\circ$ Base Case: $\mu=1$ and $\ell=0$ or $1$. 

For $\ell=0$, the result follows directly 
from applying \eqref{eq:1-higher-order-fitting}. 

For $\ell=1$, it follows as a consequence of applying 
\eqref{eq:D_t-D_x^j=D_x^j-D_t} and \eqref{eq:E_s-flow-fitting}. 

\smallskip 

$2^\circ$ Inductive Step: $\mu=1$ and $\ell \geq 2$. 

Applying Lemma \ref{lem:D_t-D_x-gamma}  
and \eqref{eq:1-higher-order-fitting}, we obtain the following:  
\begin{align}
\label{D_t-D^s_x}
\nonumber
 D_{t} D^{\ell}_{x} \gamma_{l}=& D_{t} D^{\ell-1}_{x} \gamma_{l, x}=
D^{\ell-1}_{x}D_{t} \gamma_{l,x}+\sum^{\ell-1}_{r=1}D_{x}^{\ell-1-r}\left[R(\gamma_{l, t}, \gamma_{l, x}) D^{r-1}_{x}\gamma_{l, x}\right] \\  \nonumber
=&D^{\ell}_{x} \gamma_{l, t}+\sum^{\ell-1}_{r=1}D_{x}^{\ell-1-r}\left[R(\gamma_{l, t}, \gamma_{l, x}) D^{r-1}_{x}\gamma_{l, x}\right] \\
=& (-1)^{k+1}D_{x}^{2k+\ell-1}\gamma_{l, x}+Q^{2k+\ell-3}_{3}(\gamma_{l,x}). 
\end{align}
Hence, the lemma holds for $\mu=1$. 

By applying \eqref{D_t-D^s_x} and the definition of 
$P^{a, c}_{b}(\gamma_{l,x})$ along with an induction argument, 
we obtain 
\begin{align}
\label{general-formula-Q^a-c_b}
D_{t}P^{a,c}_{b}(\gamma_{l,x})=Q^{a+2k,c+2k}_{b}(\gamma_{l,x}).
\end{align} 
Using \eqref{D_t-D^s_x} and \eqref{general-formula-Q^a-c_b}, we obtain   
\begin{align*}
 D_{t}^{\mu} D^{\ell}_{x} \gamma_{l} =& D_{t} (D_{t}^{\mu-1} D^{\ell}_{x} \gamma_{l} ) 
=D_{t}\left((-1)^{(k+1) (\mu-1)} D_{x}^{2k(\mu-1)+\ell-1} \gamma_{l,x}+Q^{2k(\mu-1)+\ell-3}_{3}(\gamma_{l,x})\right)\\
=&(-1)^{(k+1)\mu} D_{x}^{2k\mu+\ell-1} \gamma_{l,x}+Q^{2k\mu+\ell-3}_{3}(\gamma_{l,x}). 
\end{align*}
This completes the proof of \eqref{eq:1-higher-order-fitting}.

The proof of \eqref{eq:1-second-order-fitting} follows the same reasoning as that of \eqref{eq:1-higher-order-fitting}. 
Meanwhile, the proof of \eqref{eq:2-second-order-fitting} results from applying Lemma \ref{lem:D_t-D_x-gamma} and \eqref{eq:1-second-order-fitting}.
The detailed steps are omitted and left to the interested readers. 

\end{proof}

\begin{lem}\label{lem:bdry_vanishing} 
Let $\gamma_{l}: (t_0,T_0)\times I_{l}\rightarrow M$ be a smooth solution to 
\eqref{eq:E_s-flow} for any $l\in\{1,\cdots,q\}$. 
Then, we have the followings:   
\begin{itemize}
\item[(i)] 
\begin{align}
\label{BC-1}
D^{\mu}_{x}D_{t}^{\ell} \gamma_{1}(t,x_0)=0, 
~ 
D^{\mu}_{x}D_{t}^{\ell} \gamma_{q}(t,x_q)=0, 
~ \forall \, \mu \in \{0, \ldots, k-1\}, \, \ell \in \mathbb{N}, \,  t \in (t_0,T_0) 
\text{.} 
\end{align}
\item[(ii)] 
As 
$l \in \{1, \ldots, q-1\}$, 
\begin{align}
\label{BC-2} 
D^{\mu}_{x}D_{t}^{\ell}\partial_{t}\gamma_{l+1}(t,x_{l}^{+})=D^{\mu}_{x}D_{t}^{\ell}\partial_{t}\gamma_{l}(t,x_{l}^{-}), 
~ \forall\, \mu \in \{0, \ldots, 2k-2\}, \, \ell \in \mathbb{N}_{0}, \, t \in (t_0,T_0) 
\text{.}
\end{align}
\end{itemize}
\end{lem}

\begin{proof}

We prove \eqref{BC-1} and \eqref{BC-2} inductively on $\ell$. 

(i) 
As $\ell=1$, we apply \eqref{eq:D_t-D_x^j=D_x^j-D_t} and \eqref{eq:D_t-D_x^j->D_x^j-D_t} to obtain 
\begin{align} \label{eq:identity-1} 
D_{x}^{\mu} D_{t} \gamma_{l}=D^{\mu-1}_{x} D_{t} D_x \gamma_{l}
= & D_t D^{\mu-1}_{x} \gamma_{l,x} 
- \sum_{\nu=1}^{\mu-1} D_{x}^{\mu-1-\nu}\left[R(\gamma_{l,t}, \gamma_{l,x}) D_{x}^{\nu-1} \gamma_{l,x}  \right] 
\text{.} 
\end{align} 
From applying the boundary conditions 
\eqref{eq:BC-higher-order-1}, \eqref{eq:BC-higher-order-2} at the two end-points, 
the right hand side of \eqref{eq:identity-1} equals zero. 
Thus, \eqref{BC-1} holds as $\ell=1$. 

As $\ell\ge 2$, we apply \eqref{eq:D_t-D_x^j->D_x^j-D_t} to obtain    
\begin{align} 
\label{eq:identity-2}
D^{\mu}_{x}D_{t}^{\ell} \gamma_{l}
=&D^{\mu}_{x}D_t (D_{t}^{\ell-1} \gamma_{l} )
=D_t D^{\mu}_{x} (D_{t}^{\ell-1} \gamma_{l} ) 
- \sum_{\nu=1}^{\mu} D_{x}^{\mu-\nu}\left[R(\gamma_{l,t}, \gamma_{l,x}) D_{x}^{\nu-1} D_{t}^{\ell-1} \gamma_{l} \right]
\text{.}
\end{align}
By applying the boundary conditions \eqref{eq:BC-higher-order-1}, \eqref{eq:BC-higher-order-2} at the two end-points, and by an induction argument on $\ell$, 
the right hand side of \eqref{eq:identity-2} vanishes. 
Therefore, \eqref{BC-1} holds.   

(ii) 
The proof of \eqref{BC-2} is also an induction argument. We proceed as above by applying the boundary conditions \eqref{eq:BC-higher-order-1} and \eqref{eq:BC-higher-order-3} to \eqref{eq:identity-2} 
to conclude that the right hand side of \eqref{eq:identity-2} is continuous at the boundary points $\{x_1,\cdots,x_{q-1}\}$, which offers the required continuity. Thus, \eqref{BC-2} is proven. 

\end{proof}

\begin{lem}[{\cite[(3.5)]{MSK10}}]
\label{first-variation-higher-order}
Let $0=x_0<x_1 <\cdots<x_{q}=1$, $I_{l}=(x_{l-1},x_{l})$, and $\Omega$ denote the set of all $C^{2k-2}$ paths $\gamma : [0,q] \to M$ such that 
$\gamma_{l}:=\gamma_{|_{\bar{I}_{l}}}$ is smooth $C^{\infty}$. We define the tangent space of $\Omega$ at a path $\gamma$, $T_{\gamma}\Omega$, 
as the set of all $C^{k-1}$ vector fields $W: [0,q] \to TM$ such that $W_{|_{[x_{l-1},x_{l}]}}$ is smooth.
If $\alpha : (-a,a) \times [0,q] \to M$ is a one-parameter variation of $\gamma \in \Omega$ is defined by 
\begin{align*}
\alpha(\varepsilon,x)=exp_{\gamma}(\varepsilon W(x))
\end{align*}
 and $W \in T_{\gamma}\Omega$ is
the variational vector field associated to $\alpha$, 
then
\begin{align*}
&\frac{d}{d\varepsilon} \big\vert_{\varepsilon=0} \mathcal{E}_{k}[\alpha(\varepsilon,x)] 
\\ 
& =
(-1)^{k} \sum^{q}_{l=1} \int_{I_{l}} \bigg{\langle}   
D_{x}^{2k-1} \partial_{x} \gamma_{l} 
+\sum^{k}_{\mu=2} (-1)^{\mu} 
R\left(D^{2k-\mu-1}_{x} \partial_{x} \gamma_{l}, 
D^{\mu-2}_{x} \partial_{x} \gamma_{l}\right) \partial_{x} \gamma_{l} , 
W \bigg{\rangle} \text{ }dx 
\\ 
~~~ & +\sum^{k}_{\ell=1} \sum^{q}_{l=1}(-1)^{\ell-1}  
\bigg{\langle} D^{k-\ell}_{x}W, 
D^{k+\ell-2}_{x}\partial_{x} \gamma_{l} \bigg{\rangle} \bigg{|}^{x^{-}_{l}}_{x^{+}_{l-1}} 
\\ 
~~~ &+\sum^{k-1}_{\mu=2}\sum^{k-\mu}_{\ell=1}\sum^{q}_{l=1} (-1)^{\ell-1}  \bigg{\langle} 
D^{k-\mu-\ell}_{x} 
\left[R(W, \partial_{x} \gamma_{l}) D^{\mu-2}_{x}\partial_{x} \gamma_{l}\right], 
D^{k+\ell-2}_{x}\partial_{x} \gamma_{l} 
\bigg{\rangle} \bigg{|}^{x^{-}_{l}}_{x^{+}_{l-1}}
\end{align*}
and 
\begin{align*}
\frac{d}{d\varepsilon} \big\vert_{\varepsilon=0} \mathcal{T}[\alpha(\varepsilon,x)]
= -\sum_{l=1}^{q}\int_{I_{l}} \langle D_{x}\partial_{x} \gamma_{l} , W\rangle \text{ }dx 
+\sum_{l=1}^{q} \langle W,  \partial_{x} \gamma_{l} \rangle \bigg{\vert}^{x_{l}}_{x_{l-1}}. 
\end{align*}
\end{lem}

The following three lemmas are straightforward extensions of the results presented in \cite[Remark B.1]{DLP21}, \cite[Remark B.2]{DLP21}, and \cite[Remark B.5]{DLP21}, corresponding to the cases when $k=2$ respectively.

\begin{lem}[{cf. \cite[Remark B.1]{DLP21}}]
\label{lem:RemarkB1} 
For $\alpha \in (0,1)$, $T>0$, $m\leq \ell$ and $m, \ell \in \mathbb{N}_0$, we have 
$$ C^{ \frac{\ell+\alpha}{2k}, \ell+\alpha}([0,T]\times[0,1]) \subset  C^{ \frac{m+\alpha}{2k},m+\alpha}([0,T]\times[0,1]) . $$
Moreover, if $v \in C^{ \frac{\ell+\alpha}{2k}, \ell+\alpha}([0,T]\times[0,1])$, then $\pa_x^l v \in C^{ \frac{\ell-l+\alpha}{2k}, \ell-l+\alpha}([0,T]\times[0,1])$ for all $0 \leq l\leq \ell$ so that
$$\|\pa_x^l v \|_{C^{ \frac{\ell-l+\alpha}{2k}, \ell-l+\alpha}([0,T]\times[0,1])} \leq \| v\|_{C^{ \frac{\ell+\alpha}{2k}, \ell+\alpha}([0,T]\times[0,1])}.$$
In particular at each fixed $x \in [0,1]$, we have 
$\pa_x^l v (\cdot,x )\in C^{s,\beta}([0,T])$ 
with $s =[\frac{\ell-l+\alpha}{2k}]$ and $\beta = \frac{\ell-l+\alpha}{2k}-s$. 
\end{lem}

\begin{proof}
The proof of this lemma follows from the definition. 
\end{proof}

\begin{lem}[{cf. \cite[Lemma B.2]{DLP21}}]
\label{lem:tech2} 
For $\ell \in \mathbb{N}_0$, $\alpha \in (0,1)$ and $T>0$, we have 
if $v,w \in C^{\frac{\ell+\alpha}{2k},\ell+\alpha}([0,T]\times [0,1])$, then
$$ \| v w \|_{C^{\frac{\ell+\alpha}{2k}, \ell+\alpha}} \leq C \| v \|_{C^{\frac{\ell+\alpha}{2k},\ell+\alpha}} \| w \|_{C^{\frac{\ell+\alpha}{2k},\ell+\alpha}} \, ,$$
with $C=C(\ell)>0$. 
\end{lem} 
\begin{proof}
The proof of this lemma follows by a direct computation and the definition of parabolic H\"{o}lder spaces. 
\end{proof}

\begin{lem} 
\label{lem:T^beta-norm0} 
Let $l,m \in \mathbb{N}_0$, $l+m<2k$, $\alpha\in(0,1)$, $k\in\mathbb{N}$, $T\in(0,1)$ and $v \in C^{ \frac{2k+\alpha}{2k}, 2k+\alpha}([0,T]\times[0,1])$ such that $v(0,x)= 0$, for any $x \in [0,1]$.  
Then, 
\begin{align}
\label{interpolation-inequality}
\| \pa_x^l v \|_{C^{\frac{m+\alpha}{2k}, m+\alpha}([0,T]\times[0,1])} \leq C(m) T^{\beta} \|  v \|_{C^{\frac{2k+\alpha}{2k}, 2k+\alpha}([0,T]\times[0,1])}  
\text{,} 
\end{align}
where $\beta=\min \{ \frac{1-\alpha}{2k}, \frac{\alpha}{2k} \}$ as $l\in\mathbb{N}_0$, and $\beta=\frac{\alpha}{2k}$ as $l\in\mathbb{N}$. 
Moreover, for each fixed $x \in [0,1]$, we have  
\begin{align} 
\label{interpolation-inequality-x}
\|\pa_x^l v (\cdot, x)\|_{C^{\frac{m+\alpha}{2k}}([0,T])} 
\leq 
\| \pa_x^l v \|_{C^{\frac{m+\alpha}{2k}, m+\alpha}([0,T]\times[0,1])} 
\leq C(m) T^{\beta} \|  v \|_{C^{\frac{2k+\alpha}{2k}, 2k+\alpha}([0,T]\times[0,1])} 
\text{.}  
\end{align} 
In addition, if $\alpha^\prime\in [0, \alpha]$, then we have 
\begin{align}
\label{interpolation-inequality-x-1}
\| v (\cdot, x)\|_{C^{\frac{2k+\alpha^\prime}{2k}}([0,T])} 
\leq \|  v \|_{C^{\frac{2k+\alpha}{2k}, 2k+\alpha}([0,T]\times[0,1])} 
\text{.} 
\end{align} 
\end{lem}

\begin{proof}
The proof mainly follows the argument presented in \cite[Lemma B.5]{DLP21} with the distinction that $k=2$ therein.

Recall that $0\le m,l <2k$ and by the definition of H\"{o}lder-norms, we have 
\begin{align}\label{eq:normapp1}
\|\partial_{x}^{l} v \|_{C^{ \frac{m+\alpha}{2k}, m+\alpha}([0,T]\times[0,1])} =
\sum_{j=0}^{m} \sup_{[0,T] \times [0,1]} |\partial_{x}^{j+l}v (t,x)| +  [\partial_{x}^{m+l}v ]_{\alpha,x} + \sum_{j=0}^{m}  [\partial_{x}^{j} \partial_{x}^{l} v ]_{\frac{m+\alpha -j}{2k},t}
\text{.} 
\end{align} 
For $j \in \{ 0, \ldots, m \}$, we have $0<l+j \leq l+m  < 2k$, and 
\begin{align*}
\sup_{(t,x) \in [0,T] \times [0,1]} |\partial_{x}^{j+l}v (t,x)| &= \sup_{(t,x) \in [0,T] \times [0,1]} \frac{|\partial_{x}^{j+l}v (t,x) -\partial_{x}^{j+l}v (0,x) |}{|t-0|^{\frac{2k+\alpha -(l+j)}{2k}}}|t-0|^{\frac{2k+\alpha -(l+j)}{2k}}\\
&\leq [\partial_{x}^{j+l} v ]_{\frac{2k+\alpha -(l+j)}{2k},t} \, T^{\frac{2k+\alpha -(l+j)}{2k}} \leq [\partial_{x}^{j+l} v ]_{\frac{2k+\alpha -(l+j)}{2k},t} \, T^{\frac{\alpha }{2k}}.
\end{align*} 
For the case $j=l=0$, we have 
\begin{align*}
\sup_{(t,x) \in [0,T] \times [0,1]} |v (t,x)| = \sup_{(t,x) \in [0,T] \times [0,1]}  \frac{|v(t,x) -v(0,x)|}{|t-0|} |t| \leq 
T  \cdot \sup_{ [0,T] \times [0,1]} |\partial_{t}v|.
\end{align*}

\smallskip 

Next, with the similar idea and using the fact that $|x-y| \leq 1$, $ l+m+1\leq 2k$, we have 
\begin{align*}
[\partial_{x}^{m+l}v ]_{\alpha,x} &=\sup_{(t,x), (t,y) \in [0,T] \times [0,1]}\frac{|\partial_{x}^{m+l}v(t,x)-\partial_{x}^{m+l}v(t,y) |}{|x-y|^{\alpha}} \\
&=\sup_{(t,x), (t,y) \in [0,T] \times [0,1]}\frac{|\partial_{x}^{m+l}v(t,x)-\partial_{x}^{m+l}v(t,y) |}{|x-y|}|x-y|^{1-\alpha}\\
& \leq \sup_{(t,x) \in [0,T] \times [0,1]} |\partial_{x}^{m+l+1}v(t,x)|\\
& \leq [\partial_{x}^{m+l+1} v ]_{\frac{2k+\alpha -(m+l+1)}{2k},t} \, T^{\frac{2k+\alpha -(m+l+1)}{2k}} 
\leq [\partial_{x}^{m+l+1} v ]_{\frac{2k+\alpha -(m+l+1)}{2k},t} \, T^{\frac{\alpha}{2k}} 
\text{.} 
\end{align*}

\smallskip 

For $ j\in \{0, \ldots, m \}$, 
\begin{align*}
[\partial_{x}^{j} \partial_{x}^{l} v ]_{\frac{m+\alpha -j}{2k},t}
&=  \sup_{(t,x), (t^\prime,x) \in [0,T] \times [0,1]}  \frac{ | \partial_{x}^{j+l} v (t,x) - \partial_{x}^{j+l} v (t^\prime,x)|}{|t-t^\prime|^{\frac{m+\alpha +1 -j}{2k}}} |t-t^\prime|^{\frac{1}{2k}}\\
& \leq [\partial_{x}^{l+j} v ]_{\frac{m+l+1+\alpha -(j+l)}{2k},t}\, T^{\frac{1}{2k}}. 
\end{align*}
The computation above doesn't make sense, if $m=2k-1$ and $j=0$ (and hence $l=0$), since the power of $|t-t^\prime|$ in the denominator is $\frac{m+\alpha+1 -j}{2k} >1$.
Thus, in the case $m=2k-1$, $l=j=0$, the estimate is treated as: 
\begin{align*}
[\partial_{x}^{j} \partial_{x}^{l} v ]_{\frac{m+\alpha -j}{2k},t}
=[ v ]_{\frac{2k-1+\alpha }{2k},t} 
&= \sup_{(t,x), (t^\prime,x) \in [0,T] \times [0,1]}  \frac{|v(t,x)-v(t^\prime,x)|}{|t-t^\prime|^{\frac{2k-1+\alpha}{2k} +\frac{1-\alpha}{2k}}} |t-t^\prime|^{\frac{1-\alpha}{2k}} \\ 
& \leq  \left( \sup_{[0,T]\times [0,1]} |\partial_{t} v| \right) T^{\frac{1-\alpha}{2k}}.
\end{align*} 

By putting all the estimates above together, 
we obtain: 
\begin{align*}
&\|\partial_{x}^{l} v \|_{C^{ \frac{m+\alpha}{2k}, m+\alpha}([0,T]\times[0,1])} 
\\
&\leq \Big ( \sum_{l\neq0 , j=0 }^{m} [\partial_{x}^{j+l} v ]_{\frac{2k+\alpha -(l+j)}{2k},t}+ [\partial_{x}^{m+l+1} v ]_{\frac{2k+\alpha -(m+l+1)}{2k},t} + \sum_{ j=0  }^{m}  [\partial_{x}^{l+j} v ]_{\frac{m+l+1+\alpha -(j+l)}{2k},t}
\Big) T^{\frac{\alpha}{2k}}\\
& \quad + C\left( \sup_{ [0,T] \times [0,1]} |\partial_{t}v| \right) T^{\frac{1-\alpha}{2k}} .
\end{align*}
and finish the proof of \eqref{interpolation-inequality}. 

The proof of \eqref{interpolation-inequality-x} follows from the observation that 
\begin{equation} 
\label{eq:estHolbdy}
\| \pa_x^l v (\cdot, x) \|_{C^{\frac{m+\alpha}{2k}}([0,T])} 
\leq \| \pa_x^l v \|_{C^{\frac{m+\alpha}{2k}, m+\alpha}([0,T]\times[0,1])} \, , 
\end{equation}
due to \eqref{eq:normapp1}. 

The proof of \eqref{interpolation-inequality-x-1} follows from applying \eqref{eq:estHolbdy}, the definition of H\"{o}lder-norms (or see \eqref{eq:normapp1}), and 
\begin{align*} 
\sup_{t, t^\prime \in [0,T]} 
\frac{|\partial_t v(t,x)-\partial_t v(t^\prime,x)|}
{|t-t^\prime|^{\frac{\alpha^\prime}{2k}+\frac{\alpha-\alpha^\prime}{2k}}} |t-t^\prime|^{\frac{\alpha-\alpha^\prime}{2k}} 
\leq 
\sup_{t, t^\prime \in [0,T]} 
\frac{|\partial_t v(t,x)-\partial_t v(t^\prime,x)|}
{|t-t^\prime|^{\frac{\alpha}{2k}}} 
\text{.}
\end{align*}

\end{proof}

\begin{lem}[{\cite[Lemmas C.3 $\sim$ C.6, and 4.1]{DP14}}] 
\label{lem:C_3}
Let $\gamma: I \rightarrow M \subset \mathbb{R}^n$ be a smooth map, and $\phi$ be a smooth vector field on $M$. 
Then, we have the following identities:   
\begin{align*} 
(i.)~ & | \partial_{x}( |\phi|) | \leq  |D_{x}\phi| ~~~~~ \text{ almost everywhere},  
 \\ 
(ii.)~ & 
 \| D_{x}\phi \|_{2} \leq c(n) \left(\varepsilon  \| \phi \|_{2, 2}+\frac{1}{\varepsilon} \| \phi \|_{2} \right) 
~~~~~~~~~~~~~~~~~~~~~~~~~~~~~~~~~~~~\forall\,  \varepsilon \in (0,1), 
\\ 
(iii.)~ & \| D_{x}^{i}\phi \|_{2} \leq  c(i, n, k) \cdot (\varepsilon  \| \phi \|_{k, 2}
+\varepsilon^{\frac{i}{i-k}} \| \phi \|_{2}), 
~~\forall\, \varepsilon \in (0,1), k\geq 2, 1\le i\le k-1, 
\\ 
(iv.)~ & \| D_{x}^{i}\phi \|_{2} \leq  c(i, n, k)  \| \phi \|_{k, 2}^{\frac{i}{k}} \| \phi \|_{2}^{\frac{k-i}{k}}, 
~~~~~~~~~~~~~~~~~~~~~~~~~~\forall\, k\geq 2, 1\le i\le k-1, 
\\ 
\nonumber 
(v.)~ &  \| D_{x}^{i}\phi \|_{p} \leq 
c(i, p, n, k) \cdot \| D_{x}^{i}\phi \|_{k-i, 2}^{\frac{1}{k-i}\left(\frac{1}{2}-\frac{1}{p}\right)} \,  \,  \| D_{x}^{i}\phi \|_{2}^{1-\frac{1}{k-i}\left(\frac{1}{2}-\frac{1}{p}\right)}, 
\\ 
& ~~~~~~~~~~~~~~~~~~~~~~~~~~~~~~~~~~~~~~~~~~~~~~~~~~~~~~~~~~~ 
\forall\, p \geq 2, k \geq 1, 1\le i\le k-1, 
\\ 
\nonumber 
(vi.)~ & \| D_x^{i}\phi \|_{p} \leq  c(i, p, n, k) \cdot \|\phi \|_{k, 2}^{\alpha}  \cdot \|\phi \|_{2}^{1-\alpha}, 
\\
& ~~~~~~~~~~~~~~~~~~~~~~~~~~~ \forall\,  0 \leq i \leq k, k \in \mathbb{N}, p \geq 2, 
~\text{ where }~ \alpha=\frac{1}{k}\left(i+\frac{1}{2}-\frac{1}{p}\right). 
\end{align*}
\end{lem} 

\begin{proof}
The proof follows exactly the same argument in Lemmas C.3 $\sim$ C.6 and 4.1 in \cite{DP14}, 
except that the covariant differentiation $D$ is 
the Levi-Civita connection of an $m$-dimensional Riemannian manifold $M$ 
defined in \eqref{eq:co-deri}, i.e., 
$D_x \phi (\gamma)=\partial_x \phi (\gamma)-\Pi_\gamma (\phi, \partial_x\gamma)$.
Note that the covariant differentiation $D_x$ in an extrinsic setting is also a projection to the tangent subspace $T_{\gamma}M\subset\mathbb{R}^n$ after taking partial differentiation of a tangent vector field $\phi$ along $\gamma(x)$ in $M$, while the covariant differentiation $\nabla_s$ in \cite{DP14} is defined as 
a projection to the normal subspace after taking partial differentiation of a normal vector field $\phi$ w.r.t. the arclength parameter of $\gamma$. 
Since the proof is the same as the one in \cite{DP14}, we leave the proof to the interested reader. 

\end{proof}

\begin{lem}[{\cite[Lemma $A.1$]{DLP19}}]
\label{Younggeneral}
Let $a_1, \ldots, a_n$ be positive numbers and assume $0<i_1+\cdots +i_n<2$, 
where $i_\mu>0$ for all $\mu=1, \ldots, n$. 
Then, for any $\varepsilon>0$, we have
\begin{align*}
a^{i_1}_{1} \cdot a^{i_2}_{2}\cdot \ldots \cdot a^{i_n}_{n} \leq \varepsilon (a^2_1+\ldots +a^2_{n})+C_{\varepsilon}\text{.}
\end{align*}
\end{lem}

\begin{lem}
\label{lem:interpolation-ineq} 
Let $\gamma: [t_0,T_0) \times I \rightarrow M$ be a smooth map and $\ell \in \mathbb{N}_{0}$. 
Assume that the Riemannian curvature tensor $R_M$ is smooth and uniformly bounded in $C^{\infty}$-norm. 
If $A, B\in \mathbb{N}$ with $B \geq 2$ and $A+\frac12 B<2\ell+5$, then we have
\begin{align*} 
 \sum_{\substack{[\![a,b]\!] \leq [\![A,B]\!]\\c \leq \ell+2, ~ b\geq 2}} \int_{I}\big\vert P^{a,c}_{b}( \gamma_{x}  ) \big\vert\text{ }dx \leq 
C \cdot \| \gamma_{x}  \|_{\ell+2, 2}^{\alpha} \, \| \gamma_{x}  \|_{L^2(I)}^{B-\frac{A}{\ell+2}}
\end{align*}
where  $\alpha= (A + \frac12 B-1)/(\ell+2)$ and $C=C(n, \ell, A, B, N_{L})$ 
for some constant $L\in\mathbb{N}_0$ such that $\|R_M\|_{C^{L}} \leq N_L$.
\end{lem}

\begin{proof} 
Assume that $i_{1}, \ldots, i_{b} \in \{0, 1, \ldots, c\}$, $i_{1}+\cdots+i_{b}=a$. 
From the definition of $P^{a,c}_{b}( \gamma_{x} )$ 
and by applying the H\"{o}lder's inequality, and Lemma \ref{lem:C_3} (vi.), 
we obtain 
\begin{align*} 
\nonumber 
&\int_{I}| P^{a,c}_{b}( \gamma_{x}  )| \text{ }dx 
\leq C \cdot \int_{I} \prod^{b}_{\mu=1}|D_{x}^{i_{\mu}} \gamma_{x} | 
\text{ }dx 
\leq C \cdot \prod^{b}_{\mu=1} \| D_{x}^{i_{\mu}} \gamma_{x}  \|_{L^{b}(I)} 
\\
& ~~~~~~~~~  
\leq C \cdot \prod^{b}_{\mu=1} \|  \gamma_{x}  \|_{\ell+2, 2}^{\frac{1}{\ell+2}\left[i_{\mu}+\frac{1}{2}-\frac{1}{b}\right]} 
 \cdot \| \gamma_{x}  \|_{L^2(I)}^{\frac{\ell+2-i_{\mu}}{\ell+2}}
\leq C \cdot \|   \gamma_{x}  \|_{\ell+2, 2}^{\frac{1}{\ell+2}\left(a+\frac{b}{2}-1\right)}  
\cdot \|  \gamma_{x}  \|_{L^2(I)}^{b-\frac{a}{\ell+2}} 
\end{align*}
where $C=C\left(m,n, \ell, a, b, N_L\right)$.  
The proof of the lemma follows from applying the inequalities. 
\end{proof}


\subsection{A Brief Overview of Solonnikov's Theory \cite{Solonnikov65}}
\label{subsec:Solonnikov}

We provide a brief overview of Solonnikov's convention and well-posedness result for linear parabolic systems, in particular for the form $\mathcal{L}u=f$, where $\mathcal{L}$ is defined as 
\begin{equation}\label{eq:Spline.flow}
\mathcal{L}u=\partial_t u + (-1)^b \mathcal{D}u,
\end{equation}
and $\mathcal{D}$ is an elliptic operator of order $2b$, with $b \in \mathbb{N}$. This summary is based on \cite[Theorem 4.9, Page 121]{Solonnikov65}.

Let $Q=\Omega \times [0,T]$ be a cylindrical domain in the space $\mathbb{R}^{n+1}$, and  $\Omega$ be a bounded and open domain with a smooth boundary $S=\partial \Omega$ in the space $\mathbb{R}^{n}$. The side surface of the cylinder $Q$ is denoted as $\Gamma=S \times [0,T]$. In the cylinder $Q$ we consider systems parabolic (boundary value problem) of $m$ equations with constant coefficients containing $m$ unknown functions $u_1, \ldots, u_m$, 
\begin{equation}
\label{Solonnikov:eq0}
\begin{cases}
\sum\limits^{m}_{j=1} l_{kj}(x, t, \frac{\partial}{\partial{x}}, \frac{\partial}{\partial{t}})u_{j}(x,t)=f_{k}(x,t) ~~~(k=1, \ldots, m),\\
\sum\limits^{m}_{j=1} B_{qj}(x, t, \frac{\partial}{\partial{x}}, \frac{\partial}{\partial{t}})u_{j}(x,t) \vert_{\Gamma}=\Phi_{q}(x,t) ~~~(q=1, \ldots, br),\\
\sum\limits^{m}_{j=1} C_{\alpha j}(x, \frac{\partial}{\partial{x}}, \frac{\partial}{\partial{t}})u_j(x,t) \vert_{t=0}=\varphi_{\alpha}(x) ~~~(\alpha=1, \ldots, r),
\end{cases}
\end{equation}
where the $l_{kj}$, $B_{qj}$ are linear differential operators with coefficients depending on $t$ and $x\in\partial\Omega$, 
$C_{\alpha j}$ are linear differential operators with coefficients depending on $x\in\Omega$,  
(see \cite[Equation (1.4) Page 12]{Solonnikov65}), 
and $f_{k}$, $\Phi_{q}$, $\varphi_{\alpha}$ are specified functions. 
The functions $l_{kj}$, $B_{qj}$ are polynomials in $t$ and $x$, and  $C_{\alpha j}$ are polynomials in $x$. 
As there is no confusion on the smoothness, the differential operators $\mathcal{B}=(B_{qj})$ and $\mathcal{C}=(C_{\alpha j})$ at the boundaries of the time-space domain $Q$ can be understood as defined in the interior of $Q$ and passing to the boundaries of $Q$. Keeping this in mind is useful later in understanding the Compatibility Conditions introduced by Solonnikov (see \cite[Page 98]{Solonnikov65} or the definition listed below).

The linear parabolic system \eqref{eq:Spline.flow}, as discussed in this article, corresponds to the case $b = k$ and $r = 1$ in the general framework of \eqref{Solonnikov:eq0} and the subsequent discussions (see the definitions of $b$ and $r$ in \cite[Pages 8-9]{Solonnikov65}). 
Note that $\mathrm{i}\xi$ and $p$ are the phase variables corresponding to $\partial_x$ (or more precisely to $(\partial_{x_1},\cdots,\partial_{x_n})$) and $\partial_t$, respectively. Accordingly, since the spatial dimension $n$ equals $1$, we treat $\xi$ as a real number rather than a vector $\xi=(\xi_1,\cdots,\xi_n)\in\mathbb{R}^n$.

Let $s_{k}, t_j \in \mathbb{Z}$, where $k, j \in \{1, \ldots, m\}$. 
Assume that the degree of polynomials 
$l_{kj}(x,t, \mathrm{i}\xi \lambda, p \lambda^{2b})$ with respect to the variable $\lambda$ at each point $(x,t) \in Q$ does not exceed $s_k+t_j$,  
and $l_{kj}=0$ if $s_k+t_j<0$. 
The polynomial $l_{kj}^{0}(x, t, \mathrm{i} \xi, p)$ is referred to as the principal part of $l_{kj}$ and the matrix $\mathcal{L}_0$ of elements $l_{kj}^{0}$ as the principal part of the matrix $\mathcal{L}$. 
Let 
\begin{align}
\label{def:L=det L_0}
L(y, t, \mathrm{i}\xi, p):= \det\mathcal{L}_0(y, t, \mathrm{i} \xi, p)
\text{.}
\end{align}
It is easy to verify that the condition  
$$ 
L(x,t, \mathrm{i} \xi \lambda, p \lambda^{2b})= \lambda^{\sum\limits_{j=1}^m (s_j+t_j)} L(x,t,\mathrm{i} \xi , p)  
$$ 
holds. 
By setting $\xi=0$, we see that 
\begin{align}
\label{eq:2br}
\sum\limits_{j=1}^m (s_j+t_j)=2br, 
\end{align} 
for some $r$, depending on the parabolic system. 
We may also let 
$$\underset{j}\max \, s_j=0, 
~~~\text{ and } ~~~~~~ 
\sum\limits^{m}_{j=1}(s_{j}+t_{j})=2br, ~ r>0.
$$ 
Let $\beta_{qj}$, $\gamma_{\alpha j}$ be the degree of the polynomials $B_{qj}(x,t, \mathrm{i}\xi \lambda,p \lambda^{2b})$, $C_{\alpha j}(x, \mathrm{i}\xi \lambda, p \lambda^{2b})$ with respect to $\lambda$, respectively. If $B_{qj}=0$, $C_{\alpha j}=0$, take for $\beta_{qj}$, $\gamma_{\alpha j}$ any integer. 
Define 
\begin{align*}
\begin{cases}
\sigma_{q}=\max \{\beta_{qj}-t_{j}: j \in \{1, \ldots, m\}\},
\\
\varrho_{\alpha}=\max \{\gamma_{\alpha j}-t_{j}: j \in \{1, \ldots, m\}\}.
\end{cases}
\end{align*} 
Denote by $B_{qj}^{0}$ and $C_{\alpha j}^{0}$ the principal part of $B_{qj}$ and $C_{\alpha j}$ respectively. Let $\mathcal{B}_{0}:=(B_{qj}^{0})$ and $\mathcal{C}_{0}:=(C_{\alpha j}^{0})$.

\begin{defi}[{\cite[Parabolicity condition, Pages 8 and 9]{Solonnikov65}}]
\label{def:parabolicity-condition} 
The (uniform) parabolicity condition is given by the requirement that the roots of the polynomial $L(x,t, \mathrm{i}\xi,p)$, defined in \eqref{def:L=det L_0},
considered as a function of the variable $p$, must satisfy the inequality 
$$
\operatorname{Re} p \leq -\delta |\xi|^{2b}
$$ 
for all real vectors $\xi=(\xi_1,\cdots,\xi_n)$ and all points $(x,t)\in Q$, where $\delta>0$ is a fixed constant. 
\end{defi}

\begin{rem}
\label{rem:parabolicity order}
In this article, \it{the order of parabolicity} or \it{the parabolicity order} is referred to be the $2br$ in \eqref{eq:2br}, where $r=1$ in the parabolic systems discussed in this article.  
\end{rem}

\begin{defi}[{\cite[Complementary conditions, Pages 11 and 12]{Solonnikov65}}] 
\label{def:complementary-condition}
Below are definitions of complementary conditions at parabolic boundaries. 
\begin{enumerate} 
\item[(i)] The polynomial $M^{+}$.  
Let $s \in S$, $\nu(x)=(\nu_1, \ldots, \nu_n)$ be the unit vector of the interior normal to $S$ at the point $x$ and let $\zeta(x)$ be a vector lying in the plane tangent to $S$ at the same point. For any $\zeta$, the polynomial $L(x, t, i(\zeta + \tau \nu), p)$ as a function of $\tau$ has $br$ roots $r_s^+$ with a positive imaginary part and $br$ roots with a negative imaginary part, provided $\operatorname{Re} p \geq -\delta_1 \zeta^{2b}$ ($0 < \delta_1 < \delta$) 
and $|p|^2 + |\zeta|^{4b} > 0$.
Let
\[
M^+(x, t, \zeta, \tau, p) = \prod_{s=1}^{br} (\tau - r_s^+(x, t, \zeta, \tau, p)). 
\]
\item[(ii)] 
The complementary conditions at the boundaries 
are fulfilled if, at each point $(x,t) \in \Gamma$ and for any tangent $\zeta(x)$, 
the rows of the matrix
\[
\mathcal{A}(x, t, i(\zeta + \tau \nu), p) = \mathcal{B}_0(x, t, i(\zeta + \tau \nu), p) \, \widehat{\mathcal{L}}_0(x, t, i(\zeta + \tau \nu), p)
\] 
are linearly independent modulo the polynomial $M^+$, 
$\operatorname{Re} p \geq -\delta_1 \zeta^{2b}$, 
and $|p|^2 + \zeta^{4b} > 0$, 
where $\widehat{\mathcal{L}}_0 =L \mathcal{L}_0^{-1}$. 
\item[(iii)]  
The complementary conditions of the initial datum 
are fulfilled if, at each point $x \in \Omega$, 
the rows of the matrix
\[
\mathcal{D}(x, p) = \mathcal{C}_0(x, 0, p) \, \widehat{\mathcal{L}}_0(x, 0, 0, p)
\] 
are linearly independent modulo the polynomial $p^r$, 
where $\mathcal{C}_0$ is the matrix of elements $C_{qj}^0$.  
\end{enumerate}
\end{defi}

\begin{rem}\label{Remark:Complementary-condition} 
In this article, we apply Solonnikov's theory to equation \eqref{Solonnikov:eq0} on the domain $Q = \Omega \times [0,T]$, where $\Omega = (0,1)$. In this one-dimensional spatial setting, there is no interior unit normal vector $\nu$ to consider, as the boundary $S = \{0,1\}$ consists of isolated points. Moreover, the polynomial $L$ is independent of the time variable $t$, and thus its roots $r_s$ are also time-independent. As a result, according to the definition of $M^+$, the function $M^+(x, t, \zeta, \tau, p)$ can be replaced by $M^+(y^{\ast}, \xi, p)$, where $y^{\ast} \in \{0,1\}$ denotes a boundary point and the spatial phase variable $\zeta$ is written as $\xi$ in this context. 
\end{rem}

The linear parabolic system \eqref{Solonnikov:eq0} can be formally written as 
\begin{equation}
\label{eq:Solonnikov}
\begin{cases}
\mathcal{L}u=f, 
\\
\mathcal{B}u |_{\Gamma}=\Phi, 
\\ 
\mathcal{C} u |_{t=0}=\phi,
\end{cases}
\end{equation} 
and is assumed to satisfy Assumption \ref{assumption:complementary-condition} stated below.

\begin{assumption}
[{\cite[Page 97 after (4.18) therein]{Solonnikov65}}]
\label{assumption:complementary-condition}
\text{ }
\begin{enumerate}
\item[(i)] the first equation in \eqref{eq:Solonnikov} satisfies the parabolicity conditions;
\item[(ii)] the second equation in \eqref{eq:Solonnikov} satisfies the complementarity conditions at the boundaries;
\item[(iii)] the third equation in \eqref{eq:Solonnikov} satisfies the complementarity conditions of the initial datum.
\end{enumerate}
\end{assumption}

\bigskip 
\begin{defi}
[{\cite[The Compatibility Conditions on Page 98]{Solonnikov65} }]
\label{def:compatibility-conditions-general-equa}
The compatibility conditions of order $\ell$ is fulfilled 
if
\begin{equation}
\label{compatibility-condition-general}
\pa^{i_q}_t \sum_{j=1}^{br} B_{qj}(x^\ast,t, \pa_x,\pa_t) u_j |_{\substack{x^\ast\in \partial\Omega\\t=0}} = \pa^{i_q}_t \Phi_q |_{t=0},
\end{equation}
for $2b i_q+\sigma_q\leq \ell$. 
The boundary operators in 
\eqref{compatibility-condition-general} should be interpreted as differential equations in the spatial variable $x$. 
These equations are derived by substituting each temporal derivative $\partial_t$ with the corresponding spatial derivative $\partial_x$, based on the transformation relations given in the first equation of \eqref{eq:Solonnikov}. Specifically, we use the relation $\partial_t u=(-1)^{b+1} \mathcal{D}u-f$ to perform this substitution. 
\end{defi}

\bigskip 

\begin{teo}[{\cite[Theorem 4.9, Page 121]{Solonnikov65}}]
\label{Solonnikov-theorem}    
Let $l$ be a positive, non-integer number satisfying $l>max\{0,\sigma_{1}, \ldots, \sigma_{br}\}$. Let $S \in C^{l+t_{max}}$, and let the coefficients of the operators $l_{kj}$ belong to the class $C^{\frac{l-s_{k}}{2b},l-s_{k}}(Q)$, those of the operator $C_{\alpha j}$ to the class $C^{l-\varrho_\alpha}(\Omega)$, and those of the operator $B_{qj}$ to the class $C^{\frac{l-\sigma_{k}}{2b},l-\sigma_{k}}(\Gamma)$. 
\\
If $f_{j} \in C^{\frac{l-s_{k}}{2b},l-s_{k}}(Q)$, $\phi_{\alpha} \in C^{l-\varrho_{\alpha}}(\Omega)$, $\Phi_{q} \in C^{\frac{l-\sigma_{k}}{2b},l-\sigma_{k}}(\Gamma)$ and if compatibility conditions of order $l^{\prime}=[l]$ are fulfilled, then problem \eqref{eq:Solonnikov} has a unique solution $u=(u_{1}, \ldots, u_{m})$ with $u_{j} \in C^{\frac{l+t_{j}}{2b},l+t_{j}}(Q)$ for which the inequality 
\begin{align*} 
\sum^{m}_{j=1} \Vert u_{j} \Vert_{C^{\frac{l+t_{j}}{2b},l+t_{j}}(Q)} \leq C\left(\sum^{m}_{j=1} \Vert f_{j} \Vert_{C^{\frac{l-s_{j}}{2b},l-s_{j}}(Q)}+\sum^{r}_{\alpha=1} \Vert \phi_{\alpha} \Vert_{C^{l-\varrho_{\alpha}}(\Omega)}+\sum^{br}_{q=1} \Vert \Phi_{q} \Vert_{C^{\frac{l-\sigma_{k}}{2b},l-\sigma_{k}}(\Gamma)}\right)
\end{align*}
is valid.
\end{teo}

\begin{rem}
In Theorem \ref{Solonnikov-theorem}, the constant $C>0$  is independent of the non-homogeneous terms $f$ and the given initial-boundary datum $\phi$, $\Phi$. 
\end{rem}

{\bf Acknowledgement.} 
The authors wish to express their sincere gratitude to the anonymous referees for their thoughtful reviews and valuable feedback, which led to substantial improvements in this paper.  
This project is supported by the following grants:
C.-C. L. received support from the research grant (MoST 111-2115-M-003-015) of the National Science and Technology Council (NSTC) of Taiwan, the iCAG program under the Higher Education Sprout Project of National Taiwan Normal University and the Ministry of Education (MoE) of Taiwan, and the National Center for Theoretical Sciences (NCTS) at Taipei.
T.D.T. received support from the research grant (MoST 111-2115-M-003-014) of the NSTC of Taiwan. 
T.D.T. would like to thank the VIASM for the hospitality and for the excellent working condition.

{\bf Data availability.} 
Data sharing is not applicable to this article as no datasets were generated or analyzed during the current study.

{\bf Declarations.} 
The authors declare no conflicts of interest related to this publication.

\end{document}